\documentclass[onefignum,onetabnum]{siamart190516}

\PassOptionsToPackage{square,numbers}{natbib}

\usepackage[utf8]{inputenc} 
\usepackage[T1]{fontenc}    
\usepackage{url}            
\usepackage{booktabs}       
\usepackage{amsfonts}       
\usepackage{nicefrac}       
\usepackage{microtype}      
\usepackage{mathtools}
\usepackage{amsfonts}
\usepackage{mathrsfs}

\usepackage{tikz}
\usetikzlibrary{calc}
\usepackage{relsize}
\usepackage{scalerel}
\tikzset{fontscale/.style = {font=\relsize{#1}}
    }
\usepackage{pgfplots}

\usepackage{color}
\usepackage{graphicx}
\graphicspath{{Figs/}}
\usepackage{indentfirst,latexsym,bm} 
\usepackage{longtable}
\usepackage{cuted}
\usepackage{booktabs,multirow,array,multicol}
\usepackage{enumitem}
\usepackage{lipsum}
\usepackage{amsfonts}
\usepackage{graphicx}
\usepackage{epstopdf}
\usepackage{algorithmic}
\usepackage{subfigure}
\usepackage{amssymb}
\usepackage{amsopn}

\hypersetup{hypertex=true,
            colorlinks=true,
            linkcolor=blue,
            anchorcolor=blue,
            citecolor=blue}

\definecolor{ivory}{RGB}{218,215,203}

\definecolor{cuhkp}{RGB}{98,56,105} 	
\definecolor{cuhkpl}{RGB}{152,24,147} 	
\definecolor{cuhkb}{RGB}{219,160,1} 	
\definecolor{cuhkbd}{RGB}{178,129,0} 	
\definecolor{cuhkr}{RGB}{88,35,155}  	

\definecolor{blackp}{RGB}{0,0,0} 
\definecolor{redp}{RGB}{255,0,0}
\definecolor{orangep}{RGB}{255,128,0}
\definecolor{brownp}{RGB}{128,77,0}
\definecolor{yellowp}{RGB}{255,230,0}
\definecolor{greenp}{RGB}{128,230,0}
\definecolor{bluep}{RGB}{0,128,255}
\definecolor{purplep}{RGB}{152,24,147}
\definecolor{pinkp}{RGB}{230,0,128}  
\definecolor{turqp}{RGB}{64,224,208}
\definecolor{minrev}{RGB}{66,175,55}

\definecolor{amg}{RGB}{17,140,17}

\newtheorem{thm}{theorem}[section]
\newtheorem{cor}[thm]{Corollary}
\newtheorem{lem}[thm]{Lemma}
\newtheorem{prop}[thm]{Proposition}

\newtheorem{assum}[thm]{Assumption}

\DeclareMathOperator*{\argmin}{argmin}

\newcommand{\Rn}{\mathbb{R}^n}
\newcommand{\R}{\mathbb{R}}

\newcommand{\xopt}{x^\star}
\newcommand{\AAn}{\mathsf{AA}}
\newcommand{\AAr}{\mathsf{AA}\text{-}\mathsf{R}}
\newcommand{\AAnr}{\mathsf{AA}\text{(-}\mathsf{R}\text{)}}
\newcommand{\GMRES}{\mathsf{GMRES}}
\newcommand{\CR}{\mathsf{CR}}
\newcommand{\CG}{\mathsf{CG}}
\newcommand{\cV}{\mathcal V}

\newcommand{\B}{\mathbb{B}}

\newcommand{\n}{\mathbb{N}}
\newcommand{\vp}{\varphi}

\newcommand{\cA}{\mathcal{A}}

\newcommand{\spa}{\mathrm{span}}

\newcommand{\cK}{\mathcal{K}}

\newcommand{\kap}{\kappa_r}

\makeatletter
\newcommand{\rmnum}[1]{\romannumeral #1}
\newcommand{\Rmnum}[1]{\expandafter\@slowromancap\romannumeral #1@}
\makeatother

\ifpdf
  \DeclareGraphicsExtensions{.eps,.pdf,.png,.jpg}
\else
  \DeclareGraphicsExtensions{.eps}
\fi

\newsiamremark{remark}{Remark}
\newsiamremark{hypothesis}{Hypothesis}
\crefname{hypothesis}{Hypothesis}{Hypotheses}
\newsiamthm{claim}{Claim}
\newsiamthm{assumption}{Assumption}
\headers{Descent of a {Restarted} Anderson Accelerated Gradient Method}{W. Ouyang, Y. Liu, and A. Milzarek}

\title{Descent Properties of an Anderson Accelerated Gradient Method With Restarting\thanks{Submitted to the editors DATE.
\funding{A. Milzarek is partly supported by the Internal Project Fund from Shenzhen Research Institute of Big Data (SRIBD) under Grant T00120230001 and by the Shenzhen Science and Technology Program under Grant GXWD20201231105722002-20200901175001001. Y. Liu is supported by the Hong Kong Innovation and Technology Commission (InnoHK Project CIMDA).}}}

\author{Wenqing Ouyang\thanks{School of Data Science (SDS), Shenzhen Research Institute of Big Data (SRIBD), The Chinese University of Hong Kong, Shenzhen, China (\email{wenqingouyang1@link.cuhk.edu.cn} and \email{andremilzarek@cuhk.edu.cn}).}
  \and Yang Liu\thanks{Mathematical Institute, University of Oxford, UK (\email{yang.liu@maths.ox.ac.uk}).}
  \and Andre Milzarek\footnotemark[2]
 }

\hypersetup{
  pdftitle={Descent of a {Restarted} Anderson Accelerated Gradient Method},
  pdfauthor={W. Ouyang, Y. Liu, and A. Milzarek}
}

\begin{document}

\maketitle

\begin{abstract}
Anderson Acceleration ($\AAn$) is a popular acceleration technique to enhance the convergence of fixed-point schemes. The analysis of $\AAn$ approaches often focuses on the convergence behavior of a corresponding fixed-point residual, while the behavior of the underlying objective function values along the accelerated iterates is currently not well understood. In this paper, we investigate local properties of $\AAn$ with restarting applied to a basic gradient scheme {($\AAr$)} in terms of function values. Specifically, we show that {$\AAr$} is a local descent method and that it can decrease the objective function {at a rate no slower} than the gradient method {up to higher-order error terms}. These new results theoretically support the good numerical performance of {$\AAnr$} when heuristic descent conditions are used for globalization and they provide a novel perspective on the convergence analysis of {$\AAr$}  that is more amenable to nonconvex optimization problems. Numerical experiments are conducted to illustrate our theoretical findings. 
\end{abstract}




\section{Introduction}
In this work, we consider the smooth optimization problem
\begin{align}
 \label{eq-fixed}
 \min_{x\in\R^n}~f(x),   
\end{align}
where $f:\R^n\to\R$ is a continuously differentiable function. If the gradient mapping $\nabla f$ is additionally Lipschitz continuous with modulus $L$, then the basic gradient descent method with fixed step size,
\begin{align}
   \label{picard-iteration}
   x^{k+1}=x^k-\frac{1}{L}\nabla f(x^k) =: g(x^k),  
\end{align} 
can be utilized to solve problem~\eqref{eq-fixed}. Here, $g:\R^n\to\R^n$ represents the associated gradient step mapping with step size $\frac{1}{L}$. The gradient descent step~\eqref{picard-iteration} can be viewed as a fixed-point iteration and the fixed-points of $g$ are exactly the stationary points of the objective function $f$. 

Anderson Acceleration ($\AAn$) applies to fixed-point iterations of the form \eqref{picard-iteration} and is a popular technique to accelerate the convergence of such iterative fixed-point schemes. For instance, $\AAn$-based algorithms have been applied successfully in computer graphics~\cite{peng2018anderson,zhang2019accelerating,ouyang2020anderson}, reinforcement learning~\cite{geist2018anderson,shi2019regularized}, machine learning~\cite{WeiBaoLiu21}, and numerical methods for PDEs~\cite{pollock2019anderson}. In iteration $k$ and based on the past $m$ iterations $\{x^{k-m},\dots,x^k\}$, $\AAn$ first computes the mixing coefficients $\alpha^k = (\alpha^k_1,\dots,\alpha^k_{m})^\top \in \R^m$ as solution of the following optimization problem:
\begin{align}
   \label{AAsubp}
   \min_{\alpha \in \R^{m}}~\left \|h(x^{k-m})+{\sum}_{i=1}^{m}\alpha_i(h(x^{k-m+i})-h(x^{k-m}))\right\|^2,
\end{align}
where $h(x):=g(x)-x$ denotes the residual map and $m$ is a corresponding memory parameter. $\AAn$ then performs the  accelerated iteration:
 \begin{equation}  \label{eq:aa-step-i}    x^{k+1}= g(x^{k-m})+{\sum}_{i=1}^{m}\alpha^k_i(g(x^{k-m+i})-g(x^{k- m})).                      \end{equation}
The parameter $m$ is usually chosen to be fixed or it is allowed to increase each iteration until a given threshold is reached after which $m$ is reinitialized. We will refer to such a restarted version of the Anderson accelerated gradient scheme \eqref{eq:aa-step-i} as $\AAr$ (cf. \cref{algo1} in \cref{sec:algo}). The goal of this paper is to analyze and establish novel descent properties of $\AAr$. In particular, we show that $\AAr$ not only decreases the norm of the residual $\|h(x)\|$, but it can also decrease the underlying objective function $f$. Therefore, we answer the following question affirmatively: \vspace{1ex}
\begin{center}
\begin{minipage}{0.9\textwidth}
\textit{Can Anderson accelerated schemes achieve descent on the underlying objective function values?}
\end{minipage}
\end{center}
 
\subsection{Related Work and Literature}
Originally proposed by Anderson~\cite{anderson1965iterative} for solving partial differential equations, $\AAn$ has gained steadily growing attention during the last decade~\cite{peng2018anderson,pavlov2018aa,henderson2019damped,mai2020anderson,ouyang2020anderson}. Though widely used in various fields and applications, the theoretical analysis and properties of $\AAn$ are still somewhat limited. $\AAn$ is known to belong to the class of multi-secant quasi-Newton methods~\cite{eyert1996comparative,fang2009two,rohwedder2011analysis}. When applied to linear problems, $\AAn$ is equivalent to the generalized minimal residual method ($\GMRES$)~\cite{walker2011anderson,potra2013characterization}. For nonlinear problems, $\AAn$ is also closely related to the nonlinear generalized minimal residual method ($\mathsf{NGMRES}$) \cite{wang2021asymptotic}. The convergence analysis in~\cite{toth2015convergence} shows that $\AAn$ converges locally r-linearly under a smoothness condition on the map $g$ and uniform boundedness of the coefficients $\{\alpha^k\}_k$, but the obtained linear rate is slower than the rate of the original scheme. Later, in \cite{evans2020proof}, the authors prove that $\AAn$ can achieve an improved linear rate with additional quadratic error terms which overall yields r-linear convergence. This result is further improved in \cite{pollock2021anderson} by assuming sufficient linear independence on the set of difference vectors of the residuals $h(x^k)$ and q-linear convergence of $\AAn$ is established with a rate faster than the Picard iteration~\eqref{picard-iteration}. Moreover, if the coefficients $\{\alpha^k\}_k$ are assumed to be constant in each iteration, an asymptotic rate is given in \cite{wang2021asymptotic}. The convergence behavior of $\AAn$ applied to nonsmooth algorithmic schemes is also considered in \cite{mai2020anderson,bian2021anderson}. 

Since $\AAn$ is known to only converge locally~\cite{toth2015convergence,mai2020anderson}, globalization mechanisms are required to use it in practice. A simple and heuristic choice is to check whether $\AAn$ decreases the objective function value $f$ and to perform a fixed-point iteration if the decrease of the $\AAn$ step is not sufficient. Such a strategy is utilized in~\cite{peng2018anderson,ouyang2020anderson,scieur2016regularized,guo2018consistency}. However, to the best of our knowledge, no consistent global-local convergence proofs are known in this case. Alternatively, one can check whether $\AAn$ decreases the residual $\|h(x)\|$ and to reject the step if no sufficient decrease is observed. This strategy is more common and has been used in~\cite{zhang2020globally,ouyang2020nonmonotone,fu2020anderson,ermis2020a3dqn}. Transition to local fast convergence of such a globalized $\AAn$ approach is provided in \cite{ouyang2020nonmonotone}. Unfortunately, the convergence analyses in \cite{zhang2020globally,ouyang2020nonmonotone,fu2020anderson,ermis2020a3dqn} require global nonexpansiveness of $g$\footnote[1]{Though residual-based globalizations of $\AAn$ without nonexpansiveness are possible, it is not fully clear which type of convergence guarantees can be achieved.}, which often necessitates convexity of $f$. Notably, it is also possible to combine function value- and residual-based globalization techniques, see, e.g., \cite{guo2018consistency,tang2022fast}.  

Restarting strategies are part of many numerical algorithms. For instance, restarting is used in the conjugate gradient method ($\CG$) \cite{powell1977restart}, the generalized minimal residual method ($\GMRES$) \cite{saad1986gmres}, and in quasi-Newton methods \cite{meyer1976convergence}. Restarting strategies have also been widely applied in the context of $\AAn$. In \cite{artacho2008siesta,fang2009two}, $\AAn$ is restarted whenever the ratio of the square of the current residual to the sum of the squares of the previous residuals exceeds a predetermined constant. Similar ideas are discussed in \cite[Section 3.2]{zhang2020globally} and \cite{pham2021use}. Restarting strategies are sometimes also used in tandem with regularization techniques to enhance the numerical stability of $\AAn$, see, e.g., \cite{henderson2019damped,shi2019regularized,scieur2016regularized}. A comprehensive comparison between $\AAn$ with restarting and the original $\AAn$ scheme on linear problems with parallel implementation can be found in \cite{loffeld2016considerations}. The results in \cite{loffeld2016considerations} suggest that the performance of $\AAn$ with restarting is generally comparable to the performance of the original $\AAn$ method. Further supporting observations for the effectiveness of restarted $\AAn$ are provided in  \cite{pratapa2015restarted}. 

\subsection{Contributions}
The convergence analyses of $\AAn$~\cite{toth2015convergence,evans2020proof,pollock2021anderson,ouyang2020nonmonotone} focus on the decrease of the residual $\|h(x)\|$, which is natural since $\AAn$ aims to minimize this norm in the $\AAn$ subproblem~\eqref{AAsubp}. However, as mentioned, such globalization strategies usually require global nonexpansiveness of $g$ in order to obtain global convergence results, see, e.g., \cite{zhang2020globally,ouyang2020nonmonotone,fu2020anderson,ermis2020a3dqn}. Hence, the heuristic idea to base the acceptance of an $\AAn$ step on the decrease of the objective function value seems attractive, since the Picard iteration~\eqref{picard-iteration} can decrease the function value even if $f$ is nonconvex. So far, there seems to be no theoretical backing ensuring that $\AAnr$ can achieve descent on the underlying objective function --- even if strong convexity is assumed and an appropriate initial point is selected. Our aim is to investigate this gap and to show that $\AAr$ can decrease the objective function value locally. This result provides theoretical guarantees for  algorithms that utilize descent conditions for $f$ as globalization mechanism without hindering the local fast convergence of $\AAr$. We now summarize our main contributions:
\begin{itemize}
   \item To the best of our knowledge, we establish the first descent properties of $\AAr$ iterations for the gradient descent scheme~\eqref{picard-iteration}. On the one hand, this illuminates the success of algorithms that have used heuristic descent-type conditions to globalize $\AAnr$~\cite{peng2018anderson,ouyang2020anderson,scieur2016regularized}. On the other hand, our findings can be utilized in the design of novel globalization techniques for $\AAr$ methods.
   \item We verify that the iterates generated by $\AAr$ in one restarting cycle are equivalent to the iterates generated by $\GMRES$ when being run on a perturbed linear system model (for the same amount of iterations). This model, without perturbation, is exactly the quadratic expansion of $f$ and the Hessian of this model is symmetric and positive definite if we assume local strong convexity. Hence, $\AAr$ is close to running the conjugate residual method ($\CR$)~\cite{stiefel1955relaxationsmethoden} on such a quadratic model of $f$. Motivated by these observations, we analyze and specify the error between $\GMRES$ and $\CR$ under small perturbations of the system matrix which will allow us to link $\AAr$, $\CR$, and $\CG$. Based on classical results for $\CG$~\cite{hestenes1952methods}, we then show that $\AAr$ not only decreases the objective function value, but the overall achieved descent is actually no smaller than the one obtained by performing a gradient descent step with step size $\frac1L$ up to higher-order error terms. Some byproducts of our results indicate that $\CR$ itself decreases the objective function value no slower than the gradient descent method. 
  \item We design a practical function value-based globalization mechanism for $\AAr$. Unlike residual-based globalizations, this allows to apply $\AAr$ directly to nonconvex  problems without requiring any adjustments of the underlying $\AAr$ scheme. We illustrate the numerical performance of our simple globalization and numerically confirm the derived theoretical descent guarantees of $\AAr$ on several nonconvex large-scale problems.
\end{itemize}

\subsection{Organization}
This work is organized as follows. In \cref{sec:algo}, we introduce the algorithmic details of $\AAr$ and list the standing assumptions. In \cref{sec:convergence_analysis}, we derive the core descent properties of $\AAr$. This is done step by step. In \cref{sec:q_linear_res}, we first establish q-linear convergence. The mentioned equivalence between $\AAr$ and $\GMRES$ is shown in \cref{aa_gmres}. Next, in \cref{gmres_cr}, we analyze the error between the sequences generated by $\CR$ and $\GMRES$ which allows to connect $\GMRES$ and $\CR$. The detailed connection between $\CR$ and $\CG$ is investigated in~\cref{cr_cg}. An objective function value-based globalization of $\AAr$ is presented in \cref{globalization}. Finally, in \cref{sec:exp}, we verify our theoretical results and test the proposed globalized $\AAr$ algorithm on several examples.  

\subsection{Notation}
Throughout this work, we consider the fixed-point mapping $g(x) = x- \nabla f(x)/L$ and the corresponding residual $h(x):=g(x)-x = - \nabla f(x)/L$. For a given sequence of iterates $\{x^k\}_k$, we define the terms:
\[     M_{h}^k:=\max_{k-\hat m\leq i\leq k}\|h(x^i)-h(x^{k-\hat m})\|,\quad M_x^k:=\max_{k-\hat m\leq i\leq k}\|x^i-x^{k-\hat m}\|,                 \]
where $\hat m = \mathrm{mod}(k,m+1)$. We further introduce the following matrices and notations:
\begin{align*}
    {X}_k&:=[x^{k-\hat m+1}-x^{k-\hat m},\dots,x^{k}-x^{k-\hat m}] \in \R^{n \times \hat m},\notag\\
    H_k&:=[h(x^{k-\hat m+1})-h(x^{k-\hat m}),\dots,h(x^{k})-h(x^{k-\hat m})] \in \R^{n \times \hat m}, \notag\\
    G_k&:=[g(x^{k-\hat m+1})-g(x^{k-\hat m}),\dots,g(x^{k})-g(x^{k-\hat m})] \in \R^{n \times \hat m}, \notag
\end{align*}
$\hat x^0:=x^0$, $\hat{x}^k:=x^{k-\hat m}+X_k\alpha^k$, $g^k:= g(x^k)$, $\hat{g}^k:=g^{k-\hat m}+G_k\alpha^k$, $h^k:=h(x^k)$, and $\hat{h}^k:=\hat{g}^k-\hat{x}^k$. The definition of $\hat x^k$ follows \cite[Equation (2.4)]{evans2020proof}. We further note that the $\AAn$ subproblem \cref{AAsubp} can be viewed as finding the minimal value of the linearized residual of $\hat x^k$ which is $\hat h^k$. Based on these notations, we can express the solution to \eqref{AAsubp} explicitly by $\alpha^k=-(H_k^\top H_k)^{-1}H_k^\top h^{k-\hat m}$ provided that $H_k^\top H_k$ is invertible. 
    
 For given $n \in \mathbb{N}$, we set $[n] := \{1,\dots,n\}$. For a matrix $A$, $\sigma_{\max}(A)$ ($\lambda_{\max}(A)$) denotes the largest singular value (eigenvalue) of $A$ and $\sigma_{\min}(A)$ ($\lambda_{\min}(A)$) is the smallest singular value (eigenvalue) of $A$. The condition number of $A$ is given by $\kappa(A) := \sigma_{\max}(A)/\sigma_{\min}(A)$. Unless specified otherwise, the norm $\|\cdot\|$ refers to the standard Euclidean norm for vectors and the spectral norm for matrices. We will use $\|\cdot\|_F$ to denote the Frobenius norm of a matrix. The space $\cK^t(A,b):=\spa\{b,Ab,\dots,A^{t-1}b\}$ is used to denote the $t$-th Krylov space generated by $A$ and $b$.


\section{Anderson Acceleration with Restarting}
\label{sec:algo}
Throughout this paper, we assume that:
\begin{assum} \label{assum:a}
    There is some $r>0$ and a stationary point $\xopt\in \R^n$ of $f$ (i.e., $\nabla f(\xopt) = 0$) such that:
\begin{enumerate}[label=\textup{\textrm{(A.\arabic*)}},topsep=0pt,itemsep=0ex,partopsep=0ex]
        \item \label{A1} The function $f$ is $L$-smooth on $\R^n$.
        \item \label{A2} The function $f$ is $\mu$-strongly convex on $\B_r(\xopt) := \{x: \|x-\xopt\| < r\}$. 
        \item \label{A3} The Hessian $\nabla^2 f$ is Lipschitz continuous with modulus $L_{H}$ on $\B_r(\xopt)$. 
\end{enumerate}
\end{assum}

We note that the conditions formulated in~\cref{assum:a} are common in the convergence analysis of Anderson accelerated gradient methods, see, e.g., \cite{toth2015convergence,evans2020proof} for comparison. Let $\kap := \frac{L}{\mu}$ denote the condition number of the Hessian $\nabla^2 f$ on $\B_r(\xopt)$. Then, under~\cref{assum:a}, it follows:
\begin{align}
    \label{contract_g}
    \|g(x)-g(y)\|\leq \left\|I-\frac{1}{L}\bar{F}\right\|\|x-y\|\leq \left(1-\frac{1}{\kap}\right)\|x-y\|, \quad \forall~x,y\in\B_r(\xopt), 
\end{align}
where $\bar F:= \int_0^1\nabla^2f(y+t(x-y))\,\mathrm{dt}$. Hence, $g$ is contractive on $\B_r(\xopt)$ with Lipschitz constant $1-\frac{1}{\kap}$.

We study local properties of $\AAn$ with restarting applied to the gradient mapping $g(x)=x-\frac{1}{L}\nabla f(x)$. In particular, given some initial point $x^0$ which is sufficiently close to $\xopt$, we apply $\AAn$ on $\{x^0,\dots,x^{k-1}\}$ to obtain the new iterate $x^k$. After $m$ iterations ($m$ is the fixed memory parameter), this procedure is stopped and $\AAn$ is restarted with $x^{m+1}$ as new initial point of the next cycle. The full algorithm of the restarted $\AAn$ scheme for problem~\eqref{eq-fixed} -- {$\AAr$} -- is shown below in~\cref{algo1}.

\begin{algorithm}[ht]
    \caption{$\AAn$ with Restarting {($\AAr$)}}
    \label{algo1}
    \begin{algorithmic}[1]         
    \REQUIRE  Choose the initial point $x^0$ and the memory parameter $m \in \mathbb{N}$. Set the current memory parameter as $\hat m=0$.
        \FOR{$k=0,1,\dots$}
         \STATE $\hat m = \mathrm{mod}(k,m+1)$.
         \IF{$\hat m=0$}
             \STATE  Set $x^{k+1}=g(x^k)$.
             \ELSE
             \STATE  Calculate the coefficient $\alpha^k$ based on the sequence $\{h(x^k),\dots,h(x^{k-\hat m})\}$ via solving \eqref{AAsubp} and set $x^{k+1}=g^{k-\hat m}+G_k\alpha^k$.
             \ENDIF
        \ENDFOR
    \end{algorithmic}
\end{algorithm}


\section{Convergence and Descent Properties of {$\AAr$}}
\label{sec:convergence_analysis}
Most classical convergence analyses of $\AAn$ are based on the same idea --- linearization~\cite{toth2015convergence,evans2020proof,mai2020anderson,ouyang2020nonmonotone}. It is well-known that $\AAnr$ is equivalent to $\GMRES$ if the mapping $g$ is affine~\cite{potra2013characterization} and therefore, $\AAnr$ can be viewed as $\GMRES$ being applied to a linear approximation of $g$ along with some linearization error. When specialized to the gradient mapping, these analyses ignore the structural information that $g$ has a symmetric Hessian. Taking this information into account, we can deduce that the system matrix of the $\GMRES$ procedure is essentially close to a symmetric positive definite matrix, which means that $\GMRES$ is close to $\CR$ in this case. This observation motivates us to utilize classical tools for $\CR$ to show that $\AAr$ locally performs descent steps for $f$. The main technical difficulty lies in the fact that the iterates generated by $\AAr$ only coincide with the iterates generated by $\GMRES$ after one additional gradient step. We resolve this complication by connecting $\GMRES$ and $\CG$ (via $\CR$) and by analyzing the relevant properties via a $\CG$-based perspective.
 
\subsection{Q-Linear Convergence of $\AAr$}
\label{sec:q_linear_res}
We first present an additional assumption and several basic properties of $\AAr$ (and $\AAn$) that allow to establish q-linear convergence. This will serve as a foundation for our later results. 
\begin{assum}
\begin{enumerate}[label=\textup{\textrm{(A.\arabic*)}},topsep=0pt,itemsep=0ex,partopsep=0ex,start=4]
        \item \label{A5} The condition number of $X_k^\top X_k$ is bounded by $M^2$ for every $k\in \n$.
\end{enumerate} 
\end{assum}

Let us note that the analogous assumption on $H_k$ is more common, since $H_k$ appears directly in the computation of the coefficient $\alpha_k$ in \eqref{AAsubp}. We address this issue in the following proposition and show that condition \ref{A5} is actually equivalent to assuming that the condition number of $H_k^{\top}H_k$ is bounded locally. Hence, in practice, \ref{A5} can be ensured by monitoring the condition number of $H_k^\top H_k$. For instance, we can restart the current cycle whenever the condition number of $H_k^\top H_k$ exceeds a given tolerance. Alternative strategies are further discussed in \cite[Section 5.1.3]{pollock2021anderson}. It is also possible to mitigate condition \ref{A5} and related boundedness assumptions on the mixing coefficients $\{\alpha^k\}_k$, \cite{toth2015convergence,evans2020proof,mai2020anderson,bian2021anderson} via algorithmic independence checks or adaptive depth mechanisms, see, e.g., \cite{ChuDupLegSer21}. However, such adjustments naturally affect the achievable convergence and acceleration results.

\begin{prop} \label{prop-HXeq} Suppose that the conditions \ref{A1}--\ref{A3} are satisfied. Then, the following statements hold:
\begin{enumerate}[label=\textup{\textrm{(\roman*)}},topsep=0pt,itemsep=0ex,partopsep=0ex]
        \item For every $M_X > 0$ there is a neighborhood $U_1$ of $\xopt$ such that if $x^{k-\hat m}\dots,x^{k}\in U_1$ and $\kappa(X_k^{\top}X_k)\leq M_X^2$, then it holds that $\kappa(H_k^\top H_k)\leq 4\kap^2\kappa(X_k^\top X_k)$.
        \item For every $M_H > 0$ there is a neighborhood $U_2$ of $\xopt$ such that if $x^{k-\hat m},\dots,x^{k}\in U_2$ and $\kappa(H_k^{\top}H_k)\leq M_H^2$, then we have $\kappa(X_k^\top X_k)\leq 4\kap^2\kappa(H_k^\top H_k)$.
    \end{enumerate}
\end{prop}
\begin{proof}
    Without loss of generality and in order to simplify the notation, we assume $\hat m = m$ and $k=m$. Let us first define $U_1:=\B_{\delta_1}(\xopt)$ where $\delta_1: =\min\{r,(1-\frac{1}{\sqrt{2}})\frac{\mu}{\sqrt{m}M_XL_{H}}\}$. We further set $b_i:=\nabla f(x^{i})-\nabla f(x^{0})-\nabla^2f(x^0)(x^i-x^0)$ and $B_k:=[b_{1},\dots,b_k]$. Utilizing~\cite[Lemma 4.1.1]{nesterov2018lectures}, it follows $\|b_i\|\leq \frac{L_{H}}{2}\|x^i-x^{0}\|^2$ for all $i \in [k]$ and we obtain
\[ \|B_k\| \leq \|B_k\|_{F}\leq \sqrt{m} \max_{1\leq i \leq k} \|b_i\| \leq \frac{\sqrt{m}L_{H}}{2}\max_{1\leq i\leq k} \|x^i-x^{0}\|^2=\frac{\sqrt{m}L_{H}(M_x^k)^2}{2}. \]
Defining $E_k:=B_k(X_k^\top X_k)^{-1}X_k^\top$ ($E_k$ is well-defined since $X_k^\top X_k$ is invertible) and $A_k=\nabla^2f(x^0)+E_k$, a direct calculation yields
    \begin{equation} \label{eq:why-not-use-it} A_k(x^i-x^0)=\nabla f(x^i)-\nabla f(x^0) \quad \forall~i \in [k].    \end{equation}
In other words, we have $A_kX_k=-LH_k.$ Moreover, the norm of $E_k$ can be estimated as follows
\begin{align}
    \nonumber
    \|E_k\| & \leq \|B_k\|\|(X_k^\top X_k)^{-1}X_k^\top\| \\ & =\|B_k\| \sqrt{\|(X_k^\top X_k)^{-1}\|} \leq \frac{\|B_k\|_F M_X}{\|X_k\|}\leq \frac{\sqrt{m}L_{H}M_XM_x^k}{2}, \label{bound_Em}
\end{align}
    where we used $\|X_k\|\geq \max_{1\leq i\leq k} \|X_ke_i\|=M_x^k$ with $e_i \in \R^k$ being the $i$-th unit vector. Due to $M_x^k\leq 2\delta_1$, we can further infer $\|E_k\|\leq \sqrt{m}L_{H}M_X\delta_1 \leq (1-\frac{1}{\sqrt{2}})\mu$. Therefore, it holds that $\sigma_{\max}(A_k)\leq \lambda_{\max}(\nabla^2f(x^0))+\sigma_{\max}(E_k)\leq \sqrt{2}L$ and $\sigma_{\min}(A_k)\geq   \lambda_{\min}(\nabla^2f(x^0))-\sigma_{\max}(E_k)\geq \frac{1}{\sqrt{2}}\mu$. Consequently, we have $\kappa(A_k)\leq 2\kap$. This allows to bound the condition number of $H_k$:
    \[           \kappa(H_k^\top H_k)=\kappa(X_k^\top A_k^\top A_k X_k)\leq \kappa(X_k^\top X_k)\kappa(A_k^\top A_k)\leq 4\kap^2M_X^2      \]
and proves part (i). We now turn to the proof of the second statement. We define $U_2:=\B_{\delta_2}(\xopt)$, $\delta_2:=\min\{r,(1-\frac{1}{\sqrt{2}})\frac{\mu^2}{\sqrt{m}LL_{H}M_H}\}$, $\tilde b_i=\frac{1}{L}[(\nabla^2f(x^0))^{-1}(\nabla f(x^i)-\nabla f(x^0))-(x^i-x^0)]$, and $\tilde B_k=[\tilde b_1,\dots,\tilde b_k]$. As before, we obtain
    \[\|\tilde B_k\|=\frac{1}{L}\|\nabla f(x^0)^{-1}B_k\|\leq \frac{1}{L\mu}\|B_k\|\leq \frac{\sqrt{m}L_{H}(M_x^k)^2}{2L\mu}.           \]
Let $j \in [k]$ be given with $M_x^k= \max_{1\leq i\leq k} \|x^i-x^0\| = \|x^j-x^0\|$. Using \ref{A2}, it then holds that
\begin{equation}
    \label{eq_bound_H}
    \|H_k\| \geq M_h^k\geq \|h^j-h^0\|=\frac{1}{L}\|\nabla f(x^j)-\nabla f(x^0)\| \geq \frac{1}{\kap}\|x^j-x^0\|=\frac{M_x^k}{\kap}.   
\end{equation}
Therefore, setting $\tilde E_k:=\tilde B_k(H_k^\top H_k)^{-1}H_k^\top$, it follows
     \begin{align*}
        \|\tilde E_k\|\leq \|\tilde B_k\|\sqrt{\|(H_k^\top H_k)^{-1}\|}\leq\frac{\sqrt{m}L_{H}(M_x^k)^2M_H}{2L\mu\|H_k\|}\leq \frac{\sqrt{m}L_{H} M_HM_x^k}{2\mu^2}   
     \end{align*}
    and by the definition of $\delta_2$, we have $M_x^k\leq 2\delta_2$ and $\|\tilde E_k\|\leq \frac{\sqrt{m}L_{H} M_H\delta_2}{\mu^2}\leq (1-\frac{1}{\sqrt{2}})\frac{1}{L}$. Next, defining $\tilde A_k :=(\nabla^2 f(x^0))^{-1}+\tilde E_k$, we can again infer  
    \[     \tilde A_k(\nabla f(x^i)-\nabla f(x^0))=x^i-x^0, \quad i \in [k] \quad \implies \quad L\tilde A_k H_k=-X_k,          \]
$\sigma_{\max}(\tilde A_k)\leq \lambda_{\max}(\nabla^2f(x^0)^{-1})+\sigma_{\max}(\tilde E_k)\leq \frac{\sqrt{2}}{\mu}$, and $\sigma_{\min}(\tilde A_k)\geq \lambda_{\min}(\nabla^2f(x^0)^{-1})-\sigma_{\max}(\tilde E_k)\geq {1}/{(\sqrt{2}L)}$. This yields $\kappa(\tilde A_k)\leq 2\kap$ and $\kappa(X_k^\top X_k)\leq \kappa(H_k^\top H_k)\kappa(\tilde A_k^\top \tilde A_k)\leq 4\kap^2M_H^2$.                   
\end{proof}

We note that the matrix $A_k = \nabla^2 f(x^0) + E_k$ defined in the proof of \cref{prop-HXeq} is a key technical ingredient in this paper and will be used in the subsequent sections. The matrix $A_k$ consists of the symmetric Hessian $\nabla^2 f(x^0)$ and the perturbation matrix $E_k$. Our goal in the next subsections is to suitably control the norm of $E_k$ promoting a link between $\AAr$ and $\CR$.
 
Based on condition \ref{A5}, we now verify q-linear convergence of the sequence $\{\|h^k\|\}_k$. Let us remark that assumption \ref{A5} (or its equivalent formulation for the matrices $H_k^\top H_k$) is generally stronger than the condition appearing in \cite{pollock2021anderson}. Namely, in \cite[Theorem 5.1]{pollock2021anderson}, q-linear convergence of $\AAnr$ is shown under a sufficient linear independence condition on each of the columns of (a permutation of) $H_k$. Here, we will work with the slightly stronger assumption \ref{A5} as it allows us to study the behavior of $(H_k^\top H_k)^{-1}H_k^\top$ under perturbations which is required to link $\AAr$ and $\CR$. More details can be found in the proof of \cref{thm4-6}. In addition, in \cite{pollock2021anderson}, contraction and Lipschitz differentiability of $\nabla f$ is assumed on the whole space $\Rn$, while we consider the local case in a neighborhood of $\xopt$. Following the derivation in \cite{pollock2021anderson}, we first analyze the behavior of the residuals $\{\|h^k\|\}_k$ in one cycle of {$\AAr$}.

\begin{prop}
    \label{q_linear_conv}
    Let the conditions \ref{A1}--\ref{A5} be satisfied and let us further assume $g(x^i), x^i\in U_1$ with $i=k-\hat m,\dots,k$ and $\hat x^k,\hat g^k\in U_1$ (where $U_1$ is introduced in~\cref{prop-HXeq} for $M_X =M$). Then, it holds that:
    \begin{equation} \label{eq:hk-q} \|h^{k+1}\|\leq \left(1-\frac{1}{\kap}\right)\|h^k\|+\mathcal O\left(\|h^k\|\left({\sum}_{i=k-\hat m}^k\|h^i\|\right)\right).               \end{equation}
\end{prop}
\begin{proof} This result basically follows from \cite[Theorem 5.1]{pollock2021anderson}. A comprehensive proof is presented in~\cref{app:q-proof}. \end{proof}

In order to transfer the statement in~\cref{q_linear_conv} to the full sequence $\{\|h^k\|\}_k$, we need to show that $\AAr$ stays indeed local. 

\begin{lem}
    \label{lemma3-5}
    Let the conditions \ref{A1}--\ref{A3} hold and assume $x^{k-\hat m},\dots,x^k\in\B_{r}(\xopt)$. If $\kappa(H_k^\top H_k)\leq M_H^2$, then we have:
    \[    \|\hat x^k-x^{k-\hat m}\|\leq \sqrt{m}M_H\kap\|h^{k-\hat m}\| \quad \text{and} \quad \|\hat g^k-x^{k-\hat m}\|\leq (1+\sqrt{m}M_H\kap)\| h^{k-\hat m}\|.                                \] 
\end{lem}
\begin{proof}
    We start with bounding the coefficient $\alpha^k$. As mentioned, the closed-form expression for $\alpha^k$ is given by $\alpha^k=-(H_k^\top H_k)^{-1}H_k^\top h^{k-\hat m}$. Therefore, it holds that
    \[   \|\alpha^k\|\leq \sqrt{\|(H_k^\top H_k)^{-1}\|} \|h^{k-\hat m}\|\leq \frac{M_H}{\|H_k\|}  \|h^{k-\hat m}\|\leq \frac{M_H\kap}{M_x^k}\|h^{k-\hat m}\|,                   \]
    where we used \eqref{eq_bound_H} in the last inequality. Due to $\|X_k\|_F \leq \sqrt{\hat m} M_x^k$, the definition of $\hat h^k$, and \eqref{AAsubp}, this implies
    \[   \|\hat x^k-x^{k-\hat m}\|=\|X_k\alpha^k\|\leq \|X_k\|\|\alpha^k\|\leq \|X_k\|_F\|\alpha^k\|\leq \sqrt{m}M_H\kap\|h^{k-\hat m}\|                 \]
    and $\|\hat g^k-x^{k-\hat m}\|\leq \|\hat h^k\|+\|\hat x^k-x^{k-\hat m}\|\leq (1+\sqrt{m}M_H\kap)\|h^{k-\hat m}\|$.
\end{proof}

\begin{prop}
    \label{prop-local}
    Let \ref{A1}--\ref{A5} be satisfied and let $\{x^k\}_k$ be generated by \cref{algo1}. Then there exists a neighborhood $U$ of $\xopt$ such that if $x^0\in U$, it follows $\{x^k\}_k\subset U$ and $\|h(x^{k+1})\|\leq (1-\frac{1}{2\kap})\|h(x^k)\|$ for all $k\in\mathbb{N}$.
\end{prop}
\begin{proof}
    We define $S_\epsilon=\{x\in\R^n:\|h(x)\|\leq \epsilon\}\cap \B_r(\xopt)$. Due to \ref{A2}, we obtain 
    \[  \|h(x)\|=\frac{1}{L}\|\nabla f(x)\|\geq \frac{\mu}{L}\|x-\xopt\| = \frac{1}{\kap} \|x-\xopt\| \quad \forall~x\in\B_r(\xopt).     \]
    Let $U_1$ be defined as in \cref{prop-HXeq} for $M_X =M$. The previous inequality implies that there is some $\epsilon_1>0$ such that $S_{\epsilon_1}\subset U_1$. Moreover, by \cref{q_linear_conv}, there exists another neighborhood $S_{\epsilon_2}$ such that if $x^{k},\dots,x^{k-\hat m},\hat x^k,\hat g^k\in S_{\epsilon_2}$, then we have
    \begin{equation} \label{eq:just-use-it-later}    \|h(\hat g^k)\|\leq (1-{(2\kap)^{-1}}) \|h(x^k)\|.         \end{equation}
    We now take $\bar\epsilon=\min\{\epsilon_1,\frac{\epsilon_2}{2+2\sqrt{m}\kap^2M},\frac{r}{1+(2\sqrt{m}\kap M+1)\kap}\}$ and set $U=S_{\bar\epsilon}$. Let us further suppose $x^0\in U$. We use an induction to show $\|h^{k}\|\leq (1-{(2\kap)^{-1}})\|h^{k-1}\|$ and $x^k\in U$ for all $k$. It is clear that we only need to prove this conclusion for $k=1,\dots,m+1$, since the analysis is identical for the next cycle of the restarted $\AAn$ method. We start with $k=1$. By definition, we have $x^1=g(x^0)$ and according to \eqref{contract_g}, it follows:
    \[  \|x^1-\xopt\|=\|g(x^0)-g(\xopt)\|\leq (1-{\kap^{-1}})\|x^0-\xopt\|\leq r.    \]
    This proves $x^1\in \B_r(\xopt)$. Moreover, it holds that:
    \[ \|h(x^1)\|=\|g(x^1)-x^1\|=\|g(x^1)-g(x^0)\|\leq (1-{\kap^{-1}})\|h(x^0)\|\leq \bar\epsilon,         \]
    which shows $\|h(x^1)\|\leq (1-{(2\kap)^{-1}})\|h(x^0)\|$ and $x^1\in S_{\bar\epsilon}$. Next, let us assume that the induction hypothesis is true for $i=1,\dots,k$ and let us prove the conclusion for $i = k+1$. Applying \cref{prop-HXeq} and \ref{A5}, we obtain $\kappa(H_k^\top H_k)\leq 4\kap^2M^2$ and using \cref{lemma3-5}, we can further infer:
    \[ \|x^{k+1}-\xopt\|\leq \|\hat g^k-x^0\|+\|x^0-\xopt\|\leq(1+2\sqrt{m}M\kap^2)\|h^0\|+\kap\|h^0\| \leq r. \] 
    This establishes $x^{k+1}\in\B_r(\xopt)$. Similarly, we also have $\hat x^k\in\B_r(\xopt)$ and it holds that
    \[ \|h(\hat x^k)\|\leq \|h(\hat x^k)-h(x^0)\|+\|h^0\|\leq \|\hat x^k-x^0\|+ \bar\epsilon\leq (1+2\sqrt{m}\kap^2M)\bar\epsilon\leq \epsilon_2 \]
   and $\|h(\hat g^k)\|\leq \epsilon_2$. Hence, by \eqref{eq:just-use-it-later}, we can conclude $\|h^{k+1}\|=\|h(\hat g^k)\|\leq(1-\frac{1}{2\kap})\|h^k\|$ and $x^{k+1}\in S_{\bar\epsilon}$.
\end{proof}

\subsection{Local Descent Properties of $\AAr$}
 
We now formulate and present one of the main theoretical results of this paper.
\begin{thm}
    \label{thm3-18}
    Let the conditions \ref{A1}--\ref{A5} hold and let $\{x^k\}_k$ be generated by \cref{algo1}. There is a neighborhood $U$ of $x^\star$ such that if $x^0\in U$, then we have:
    \begin{equation} \label{eq:super-duper}   f(x^{k+1})\leq f(g(x^k))+\mathcal O(\|\nabla f(x^{k-\hat m})\|^3) \end{equation}
    for all $k \in \mathbb N$, where $\hat m= \mathrm{mod}(k,m+1)$. 
\end{thm}

As already outlined, the proof of \cref{thm3-18} relies on subtle connections between $\AAr$, $\GMRES$, $\CR$, and $\CG$. We will establish and discuss these connections step-by-step in the subsequent subsections.

Before proceeding with further details, let us briefly discuss $\GMRES$, $\CR$, and $\CG$, cf. \cite{saad2003iterative}. All of these algorithms are designed to solve linear systems of form $Ax=b$ via Krylov subspace techniques. For $\GMRES$ and $\CR$, the $k$-th iterate is the point in the $k$-th Krylov subspace $\cK^{k}(A,b)$ with minimal residual norm $\|Ax-b\|$. Here, $\CR$ typically requires the matrix $A$ to be symmetric positive semidefinite, which can be exploited in (faster) implementations. No additional assumptions (on $A$) need to be made when applying $\GMRES$. $\CG$ is connected to $\CR$ and requires $A$ to be symmetric, positive (semi)definite. Instead of finding elements with minimal residual norm, $\CG$ aims at minimizing the quadratic form $\frac12 x^\top Ax-b^\top x$ within the subspace $\cK^{k}(A,b)$.

We now summarize the core components of our proof. In \cref{aa_gmres}, we show that the $\AAr$ iterate $x^{k+1}$ coincides with an iterate $\bar x^k_G$ that can be generated via performing an additional gradient step on the $\GMRES$ iterate $x^k_G$. Here, $\GMRES$ is applied to a non-symmetric perturbed linear system $A(x-x^0)=b$ that is connected to $\AAr$. In addition, the gradient step $g(x^k)$ can be viewed as applying two gradient steps on the previous $\GMRES$ iterate $x^{k-1}_G$ resulting in $\tilde x^{k-1}_G$. Since the system matrix $A$ can be interpreted as a perturbed version of the symmetric Hessian $\nabla^2 f(x^0)$, our idea is to run $\CR$ on the linear system $B(x-x^0)=b$ with $B=\nabla^2f(x^0)$ and $b = -\nabla f(x^0)$ and to bound the differences between the $\CR$ and $\GMRES$ iterates. In \cref{gmres_cr}, we verify that this error has order $\mathcal O(\|b\|^2)$. Finally, in \cref{cr_cg}, we connect $\CR$ and $\CG$ and use the rich computational properties of $\CG$, \cite{hestenes1952methods}, to show that the $\CR$ iterates achieve the desired descent on a local quadratic model of $f$. In the last subsection, we combine these different components to prove that the $\AAr$ iterate $x^{k+1}$ itself decreases the objective function value up to certain higher-order error terms.

\subsection{Connecting $\AAr$ and $\GMRES$}
\label{aa_gmres}
We now establish equivalence of $\GMRES$ and $\AAr$ when running one restarting cycle. Let us introduce the matrix $A := A_m = \nabla^2f(x^0)+E_m$, where $A_m$ and $E_m$ have been defined in the proof of \cref{prop-HXeq}, i.e., it holds that
\begin{equation} A = \nabla^2 f(x^0) + E_m, \quad E_m = B_m(X_m^\top X_m)^{-1}X_m^\top, \quad B_m = [b_1,\dots, b_m], \label{eq:lets-define-A} \end{equation} 
where $b_i = \nabla f(x^i) - \nabla f(x^0) - \nabla^2 f(x^0)(x^i-x^0)$, $i \in [m]$.

The matrix $A$ can be utilized to construct a new perturbed gradient mapping $\bar g(x) := x-\frac{1}{L}(A(x-x^0)+\nabla f(x^0))$. Recalling \eqref{eq:why-not-use-it}, it can be shown that the function $\bar g$ is exact at $x^k$ for all $k=0,\dots,m$, i.e., we have $\bar g(x^k)=g(x^k)$. We now study $\GMRES$ applied to the linear system $A(x-x^0) = -\nabla f(x^0)$. While there are various implementations of $\GMRES$~\cite{saad1986gmres,bai1994newton}, each of these variants will yield the same iteration sequence $\{x_G^k\}_k$. The following proposition (which holds for general input data $A \in \R^{n \times n}$ and $b, x^0 \in \R^n$) is taken from \cite[Equation (4)]{potra2013characterization} and characterizes the iterates $x_G^k$ more explicitly. 
\begin{prop}
    \label{gmres_property}
   Let $A \in \R^{n \times n}$ be nonsingular and let $b, x^0 \in \R^n$ be given. Suppose further that $\{x^k_G\}_k$ is generated by $\GMRES$ to solve the system $A(x-x^0)=b$ with $x^0_G=x^0$. Then, we have:
    \[   x^k_G={\argmin}_{x\in x^0+\cK^k(A,b)}~\|A(x-x^0)-b\|^2.                 \]
\end{prop}

Next, we verify the equivalence of $\AAr$ and $\GMRES$ (in one restarting cycle) in the general nonlinear setting. As we will see, the matrix $A$ in \eqref{eq:lets-define-A} and the perturbed gradient mapping $\bar g$ will play an important role when connecting $\AAr$ and $\GMRES$. We further note that the linear case has been already covered in \cite[Proposition 2]{potra2013characterization}.
\begin{prop}
    \label{prop4-4}
    Let $\{x^k\}_{k=0,\dots,m+1}$ be generated by~\cref{algo1} and suppose that \ref{A5} is satisfied. Let the sequence $\{x^k_G\}_{k=0,\dots,m}$ be generated by $\GMRES$ applied to $A(x-x^0)=-\nabla f(x^0)$ with initial point $x^0_G=x^0$, where $A$ is defined as in \eqref{eq:lets-define-A}. Suppose that the matrix $A$ is nonsingular. Setting $\hat x^0 :=x^0$, it then holds that 
    \[ \bar x^{k}_G := \bar g(x_G^k) =x^{k+1} \quad \text{and} \quad x^{k}_G=\hat x^{k} \quad \forall~k = 0,\dots,m. \]
    In particular, we have $\kappa((\bar X^k_G)^\top \bar X^k_G)\leq M^2$ for each $k=1,\dots,m$, where $\bar X^k_G:=[\bar x^0_G-x^0,\dots,\bar x^{k-1}_G-x^0]$.
\end{prop}
\begin{proof}
    We prove~\cref{prop4-4} by induction. The base case $k=0$ is obviously satisfied. Next, let us suppose that the induction hypothesis is true for any $i\leq k-1$. By definition, we have $\hat x^{k}=x^{0}+\sum_{i=1}^{k}\alpha^k_i(x^{i}-x^{0})$ where $\alpha^k$ is the solution to
    \[             \min_{\alpha}~\bigl\|h(x^{0})+{\sum}_{i=1}^{k}\alpha_i(h(x^{i})-h(x^{0}))\bigr\|^2.            \]
    Based on the definition of the matrix $A$ (see \eqref{eq:lets-define-A}) and by \eqref{eq:why-not-use-it}, we further obtain:
    \begin{equation} \label{eq:a-nice-eq}    A(x^i-x^0)=\nabla f(x^i)-\nabla f(x^0)=-L(h(x^i)-h(x^0)) \quad  \forall~ i=0,\dots,m.                    \end{equation}
    Hence, $\alpha^k$ is also the solution to the problem $ \min_{\alpha}\,\|A[{\sum}_{i=1}^{k}\alpha_i(x^{i}-x^{0})]+\nabla f(x^0)\|^2$ and it holds that $\hat x^k=\argmin_{x\in x^0+\spa\{x^1-x^0,\dots,x^k-x^0\} }\|A(x-x^0)+\nabla f(x^0)\|^2$. In addition, using assumption \ref{A5}, we can infer that the vectors $\{x^1-x^0,\dots,x^k-x^0\}$ are linearly independent. Applying \cref{gmres_property}, we have $x_G^k - x^0 \in \cK^{k}(A,-\nabla f(x^0))$ and thus, it follows $\bar x_G^k - x^0 =  x_G^k - x^0 - \frac{1}{L}(A(x_G^k-x^0)+\nabla f(x^0)) \in \cK^{k+1}(A,-\nabla f(x^0))$ (for all $k$). Combining these observations and using the induction hypothesis, this yields $\spa\{x^1-x^0,\dots,x^k-x^0\}=\spa\{\bar x_G^1-x^0,\dots,\bar x_G^{k-1}-x^0\}=\cK^{k}(A,-\nabla f(x^0))$, where the last equality follows from $\mathrm{dim}(\spa\{x^1-x^0,\dots,x^k-x^0\}) = k$. Thus, by \cref{gmres_property}, we can deduce:
    \begin{align*} x^{k}_G & ={\argmin}_{x\in x^0+\cK^{k}(A,-\nabla f(x^0))}\,\|A(x-x^0)+\nabla f(x^0)\|^2 \\ & ={\argmin}_{x\in x^0+\spa\{x^1-x^0,...,x^k-x^0\}}\,\|A(x-x^0)+\nabla f(x^0)\|^2=\hat x^k \end{align*}
     and thanks to \eqref{eq:a-nice-eq}, we obtain
     \[ \bar x^k_G=\bar g(x^k_G)=\bar g(\hat x^k)= \hat x^k - {L}^{-1} \cdot {\sum}_{i=1}^k \alpha_i^k A(x^i-x^0) + h(x^0) = \hat g^k = x^{k+1}. \] 
    The last assertion in~\cref{prop4-4} now follows from $\bar x^k_G=x^{k+1}$ and \ref{A5}. 
\end{proof}

\subsection{Connecting $\GMRES$ and $\CR$}
\label{gmres_cr}
In this section, we study  $\GMRES$ and $\CR$ applied to the general linear systems
\[ A(x - x^0) = b \quad \text{and} \quad B(x - x^0) = b, \quad A, B \in \R^{n \times n}, \quad b, x^0 \in \R^n. \]
Let $x^k_{G}$ denote the $k$-th iteration of $\GMRES$ applied to $A(x-x^0)=b$ and let $x^k_{R}$ denote the $k$-th iteration of $\CR$ applied to the linear system $B(x-x^0)=b$ (in our case, we will have $B=\nabla^2f(x^0)$ and $b=-\nabla f(x^0)$). Here, we want to investigate and bound the distance between the iterates $x^k_G$ and $x^k_R$. In our analysis, we further assume $x^0_G=x^0_R=x^0$ and $b\neq 0$. We will largely utilize the following simple fact:
\begin{prop}
\label{prop4-6}
Let $a_1,b_1,a_2,b_2$ be given scalars, vectors, or matrices with appropriate dimensions such that $a_1b_1$, $a_2b_2$, $a_1-a_2$, $b_1-b_2$, and $a_1b_1-a_2b_2$ are well-defined. Then, we have $\|a_1b_1-a_2b_2\|\leq \|a_1-a_2\|\|b_1\|+\|b_1-b_2\|\|a_2\|$.                     
\end{prop} 

As usual, the concrete implementation of $\CR$ is not of our concern and we only require the following property of $\CR$:
\begin{prop}
    \label{cr_property}
    \cite[Section 2.2]{fong2012cg} Suppose $B \in \R^{n \times n}$ is symmetric and positive definite and let $b, x^0 \in \R^n$ be given. Let $\{x^k_R\}_k$ be generated by $\CR$ applied to the linear system $B(x-x^0)=b$ with $x^0_R=x^0$. Then, we have:
    \[     x^k_R={\argmin}_{x\in x^0+\cK^k(B,b)}~\|B(x-x^0)-b\|^2.        \] 
\end{prop}

Based on our earlier discussion, we now introduce several additional terms:
\begin{equation}
    \label{defn_bar_x}
    \begin{aligned}
        \bar x^k_G&=x^k_G-\frac1L(A(x^k_G-x^0)-b) ,\quad  &\bar x^k_R&= x^k_R-\frac1L(B(x^k_R-x^0)-b), \\
        \tilde x^k_G&=\bar x^k_G-\frac1L(A(\bar x^k_G-x^0)-b), \quad&\tilde x^k_R&=\bar x^k_R-\frac1L(B(\bar x^k_R-x^0)-b),\\
        \bar X^k_G&=[\bar x^0_G-x^0,\dots,\bar x^{k}_G-x^0], \quad  &\bar X^k_R&=[\bar x^0_R-x^0,\dots,\bar x^{k}_R-x^0].
    \end{aligned}
\end{equation}
Our first lemma in this section allows to connect the residuals of $\bar x^k_G$ and $x^k_G$. 
\begin{lem}
    \label{lemma4-7}
Let $B \in \R^{n \times n}$ be a symmetric, positive definite matrix with $L \geq \lambda_{\max}(B)$, $\lambda_{\min}(B)\geq \mu>0$ and suppose that $A \in \R^{n \times n}$ satisfies $\|A-B\|<\mu$. Let $x,b \in \R^n$ be given and set $\bar x=x-L^{-1}(Ax-b)$. It holds that:
    \[   \|A\bar x-b\|\leq \|Ax-b\|.            \]
\end{lem}
\begin{proof}
    First notice that $\sigma_{\min}(A)\geq \sigma_{\min}(B)-\|A-B\|>0$, which shows that $A$ is nonsingular. We can then define $x^*:=A^{-1}b$ and rewrite $Ax-b=A(x-x^*)$. Furthermore, it holds that
    \[ \|A\bar x-b\|=\|A(\bar x-x^*)\| =\|(I-L^{-1}A)A(x-x^*)\|\leq\|I-L^{-1}A\| \|A(x-x^*)\|.                    \]
    Hence, it suffices to verify $\|I-L^{-1}A\|\leq 1$. Indeed, we have $\|I-L^{-1}A\|\leq \|I-L^{-1}B\|+\|L^{-1}(A-B)\|\leq 1-\frac{\mu}{L}+\frac{\mu}{L}\leq 1$ which finishes the proof.         
\end{proof} 

Next, we present our main result of this subsection.
\begin{thm}
    \label{thm4-6}
    Let $B \in \R^{n \times n}$ be a symmetric, positive definite matrix with $\lambda_{\max}(B)\leq L$ and $\lambda_{\min}(B)\geq \mu>0$ and let $A \in \R^{n \times n}$, $b, x^0 \in \R^n$, and $\mathbb{N} \ni m \leq n$ be given. Let the sequences $\{x^k_G\}_k$ and $\{x^k_R\}_k$ be generated by $\GMRES$ and $\CR$ applied to the linear systems $A(x-x^0)=b$ and $B(x-x^0)=b$ with $x^0_G=x^0_R=x^0$, respectively. Suppose further that there are constants $C_1,C_2,C_3>0$ such that:
\begin{enumerate}[label=\textup{\textrm{(\roman*)}},topsep=0pt,itemsep=0ex,partopsep=0ex]
    \item $\|A-B\|\leq C_1 \|b\|$.
    \item For each $k=0,\dots,m$, we have $\|\bar x^k_G-x^0_G\|\leq C_2\|b\|$.
    \item We have $\kappa((\bar X^k_G)^\top\bar X^k_G)\leq C_3$ for all $k=1,\dots,m$.
\end{enumerate}    
 There exists a constant {$\epsilon_\sharp>0$} such that if $\|b\|\leq \epsilon_\sharp$, then there is $C>0$ such that:
\[      \|x^k_G- x^k_R\|\leq C\|b\|^2, \quad \|\bar x^k_G-\bar x^k_R\|\leq C\|b\|^2, \quad \|\tilde x^k_G-\tilde x^k_R\|\leq C\|b\|^2, \quad \forall~0\leq k\leq m.      \]
\end{thm}
\begin{proof}
    Without loss of generality, we can assume $b\neq 0$. We define the following quantities recursively: $\zeta_0=0$, $c_{k,1}=L^{-1}C_1C_2+L^{-2}C_1+\zeta_k$, $c_{k,2}=(k+1)C_1C_2+ L({\sum}_{i=0}^kc_{i,1}^2)^{\frac12}$, $c_{k,3}=3c_{k,2}L(k+1)C_2$, $c_{4}=\tfrac{16}{9\mu^2}C_3L^2$, $c_{k,5}=c_4c_{k,2}+\tfrac{7}{2}c_4^2c_{k,3}(k+1)C_2$, $\zeta_{k+1}=\frac{7}{4}L(k+1)C_2c_{k,5}+\sqrt{c_4}c_{k,2}$, $c_j=\max_{k=0,\dots,m}c_{k,j}$, for all $j = 1,2,3,5$ and $\epsilon_\sharp := \frac12 \min\{(L^2{C_3\sum_{i=0}^{m}c_{i,1}^2})^{-\frac12},\frac{\mu}{2C_1},\frac{L}{c_2},\frac{1}{c_3c_4}\}$. Next, let us assume $\|b\|\leq \epsilon_\sharp$. Due to $C_1\epsilon_\sharp\leq \frac{1}{4}\mu \leq \frac14 L$, we then immediately obtain:
    \begin{align}
       \label{cond_num_A}
       \|A\|\leq \|B\|+\|A-B\|\leq \tfrac{5}{4}L \quad \text{and} \quad \sigma_{\min}(A)\geq \lambda_{\min}(B)-\|A-B\|\geq  \tfrac{3}{4}\mu. 
    \end{align}
This shows that $A$ is nonsingular. Our goal is now to establish $\|x^k_G-x^k_R\|\leq \zeta_{k}\|b\|^2$ by induction. The base case $k=0$ is trivial. Let us suppose that the induction hypothesis is true for all $0 \leq i \leq k$. By~\cref{gmres_property} and \cref{cr_property}, we have:
    \begin{align*}
        x^{k+1}_G=\argmin_{x\in x^0+\cK^{k+1}(A,b)} \| A(x-x^0)-b\|^2, \quad
        x^{k+1}_R=\argmin_{x\in x^0+\cK^{k+1}(B,b)} \| B(x-x^0)-b\|^2.   
    \end{align*}
    Furthermore, mimicking the proof of~\cref{prop4-4} and using (iii), we can deduce
    \begin{equation} \label{eq:nice-to-use-later}  \bar x^k_G-x^0\in \cK^{k+1}(A,b) \quad \text{and} \quad \bar x^k_R-x^0\in \cK^{k+1}(B,b) \end{equation}  
    for all $k$ and $\spa\{ \bar x^0_G-x^0,\dots,\bar x^k_G-x^0  \} = \cK^{k+1}(A,b)$. In addition, applying (i) and (ii), it holds that
    \begin{align}
        \|\bar x^i_G-\bar x^i_R\|&=\|x^i_G-L^{-1}A(x^i_G-x^0)-x^i_R+L^{-1}B(x^i_R-x^0)\|\notag\\&=\|L^{-1}(A-B)(x^0-x^i_G)+(I-L^{-1}B)(x^i_G-x^i_R)\| \notag\\
        &\leq L^{-1}\|A-B\|(\|\bar x^i_G-x^0\|+\|\bar x^i_G-x^i_G \|)+\|I-L^{-1}B\|\|x^i_G-x^i_R\| \notag\\
        &\leq  C_1L^{-1}\|b\|(\|\bar x^i_G-x^0\|+\|\bar x^i_G-x^i_G\|) +  \|x^i_G-x^i_R\|             \notag\\ 
        &\leq  C_1L^{-1}\|b\|(C_2\|b\|+L^{-1}\|b\|) +  \zeta_i\|b\|^2 = c_{i,1} \|b\|^2,  \label{eq4-1}
    \end{align}
    where we used $\|I-L^{-1}B\| \leq 1-\frac{\mu}{L} \leq 1$ and~\cref{gmres_property} to show that:
    \begin{align*} \|\bar x^i_G-x^i_G\| &=L^{-1}{\min}_{x\in x^0+\cK^i(A,b)}\|A(x-x^0)-b\|\leq L^{-1}\|A(x^0-x^0)-b\|=L^{-1}\|b\|.\end{align*}
    Therefore, we are able to bound the error between $\bar X_G^k$ and $\bar X_R^k$:
    \[  \|\bar X_G^k-\bar X_R^k\|\leq \left({\sum}_{i=0}^{k}\|\bar x^i_G-\bar x^i_R\|^2\right)^{1/2}\leq \left({\sum}_{i=0}^kc_{i,1}^2\right)^{1/2}\|b\|^2\leq \frac{1}{2L\sqrt{C_3}}\|b\|, \]
    where we applied the definition of $\epsilon_\sharp$. Furthermore, due to (iii), we can infer:
    \begin{align}
        \label{lower_bound_Xg}
        \|\bar X_G^k\|\geq \|\bar x^0_G-x^0\|=\tfrac{\|b\|}{L} \quad\implies\quad \sigma_{\min} (\bar X_G^k)\geq \tfrac{\sigma_{\max}(\bar X_G^k)}{\sqrt{C_3}}\geq \tfrac{\|b\|}{(L\sqrt{C_3})}.   
    \end{align}
    Consequently, it holds that $\sigma_{\min}(\bar X_R^k)\geq \sigma_{\min}(\bar X_G^k)-\|\bar X_G^k-\bar X_R^k\| \geq \frac{1}{2L\sqrt{C_3}}\|b\|>0$. Thus, the column vectors of $\bar X_R^k$ are also linearly independent and by \eqref{eq:nice-to-use-later}, it follows $\spa\{ \bar x^0_R-x^0,\dots,\bar x^k_R-x^0  \} = \cK^{k+1}(B,b)$. Combining the previous arguments, we can now rewrite $x^{k+1}_G$ and $x^{k+1}_R$ as:
    \begin{align*}
        x^{k+1}_G&={\argmin}_{x\in x^0+\spa\{ \bar x^0_G-x^0,...,\bar x^k_G-x^0  \} }~\| A(x-x^0)-b\|^2,    \\
        x^{k+1}_R&={\argmin}_{x\in x^0+ \spa\{ \bar x^0_R-x^0,...,\bar x^k_R-x^0  \} }~\| B(x-x^0)-b\|^2.   
    \end{align*}
    The closed-form expressions of $x^{k+1}_G$ and $x^{k+1}_R$ are therefore given by:
    \[       x^{k+1}_G=x^0+Y^k_G((Y_G^k)^\top Y^k_G)^{-1} (Y^k_G)^\top b, \quad         x^{k+1}_R=x^0+Y^k_R((Y_R^k)^\top Y^k_R)^{-1}(Y^k_R)^\top b \]
    where $Y^k_G=A\bar X^k_G$ and $Y^k_R=B\bar X^k_R$. Our first task is to estimate the error between $Y_G^k$ and $Y_R^k$. Using (i) and (ii), we have:
    \begin{align*}
        \|Y_G^k-Y_R^k\| & \leq \|(A-B)\bar X^k_G\|+\|B(\bar X^k_G-\bar X^k_R)\| \\ & \leq (k+1)C_1\|b\|\|\bar X_G^k\|_\infty+L\left({\sum}_{i=0}^kc_{i,1}^2\right)^{1/2}\|b\|^2 \leq c_{k,2}\|b\|^2\leq c_2\|b\|^2.
    \end{align*}
    Moreover, applying \eqref{cond_num_A}, we can infer $\|Y^k_G\|\leq \|A\|\|\bar X_G^k\|\leq \frac{5L}{4}\|\bar X_G^k\|\leq \frac{5L(k+1)C_2}{4}\|b\| $ and due to $c_2\epsilon_\sharp\leq\frac{L}{2}\leq\frac{L(k+1)}{2}$, we have $c_2\|b\|^2\leq \frac{L(k+1)C_2}{2}\|b\|$. This allows to establish a bound for the norm of $Y^k_R$:
    \[          \|Y_R^k\|\leq \|Y_G^k\|+\|Y_G^k-Y_R^k\|\leq \tfrac{7}{4}L(k+1)C_2 \|b\|.      \]
     Therefore, by \cref{prop4-6}, we can bound the norm of the term $(Y^k_G)^\top Y_G^k - (Y^k_R)^\top Y^k_R$ as follows:
    \begin{align*}
        \|(Y^k_G)^\top Y^k_G-(Y^k_R)^\top Y^k_R\|&\leq \|Y_G^k-Y_R^k\|(\|Y^k_G\| +\|Y^k_R\|) \\ & \hspace{-20ex} \leq c_{k,2}\|b\|^2\left( \tfrac{7}{4}L(k+1)C_2+\tfrac{5}{4}L(k+1)C_2\right)\|b\| = 3c_{k,2}L(k+1)C_2\|b\|^3=c_{k,3}\|b\|^3.
    \end{align*}
    Our next task is to bound $\|((Y_G^k)^\top Y_G^k)^{-1}\|$. Using $A^\top A\succeq \frac{9}{16}\mu^2 I$ and \eqref{lower_bound_Xg}, we have: 
    \begin{align*}
        \|((Y_G^k)^\top Y_G^k)^{-1}\|&=\|((\bar X_G^k)^\top A^\top A \bar X_G^k)^{-1}\| \\ & \leq \frac{16}{9\mu^2} \|((\bar X_G^k)^\top\bar X_G^k)^{-1}\| \leq \frac{16C_3}{9\mu^2\|\bar X_G^k\|^2}\leq   \frac{16C_3L^2}{9\mu^2\|b\|^2}=\frac{c_4}{\|b\|^2},  
    \end{align*}
    Since $\epsilon_\sharp$ is chosen such that $1-c_{k,3}c_4\|b\|\geq \frac{1}{2}>0$, we can now apply Banach's perturbation lemma, see, e.g., \cite[Theorem 2.3.4]{golub2013matrix}, which implies:
    \begin{align*}
        \| ((Y_G^k)^\top Y_G^k)^{-1}-((Y_R^k)^\top Y_R^k)^{-1}\|&\leq  \frac{\|((Y_G^k)^\top Y_G^k)^{-1}\|^2\|(Y^k_G)^\top Y^k_G-(Y^k_R)^\top Y^k_R\|}{1-\|((Y_G^k)^\top Y_G^k)^{-1}\|\|(Y^k_G)^\top Y^k_G-(Y^k_R)^\top Y^k_R\|}     \\
        &\leq \frac{c_4^2c_{k,3}}{\|b\|-c_4c_{k,3}\|b\|^2}\leq \frac{2c_4^2c_{k,3}}{\|b\|}.
    \end{align*}
    Consequently, applying \cref{prop4-6}, it follows:
    \begin{align*}
        \|Y_G^k((Y_G^k)^\top Y_G^k)^{-1}- Y_R^k((Y_R^k)^\top Y_R^k)^{-1}\| & \\ & \hspace{-25ex} \leq \|((Y_G^k)^\top Y_G^k)^{-1}\|\|Y_G^k-Y_R^k\|+\|((Y_G^k)^\top Y_G^k)^{-1}-((Y_R^k)^\top Y_R^k)^{-1}\|\|Y_R^k\| \\ & \hspace{-25ex} \leq c_4c_{k,2}+\tfrac{7}{2}c_4^2c_{k,3}(k+1)C_2 = c_{k,5}
    \end{align*}
 and $\|Y_G^k((Y_G^k)^\top Y_G^k)^{-1}\|=\sqrt{\| ((Y_G^k)^\top Y_G^k)^{-1}\|}\leq\frac{\sqrt{c_4}}{\|b\|}$. Altogether, this yields: 
    \begin{align*}
        \|x^{k+1}_G-x^{k+1}_R\|&=\|    Y^k_G((Y_G^k)^\top Y^k_G)^{-1}(Y_G^k)^\top b-Y^k_R((Y_R^k)^\top Y^k_R)^{-1}(Y_R^k)^\top b\|\\
        & \hspace{-15.5ex} \leq  \|Y_G^k((Y_G^k)^\top Y_G^k)^{-1}- Y_R^k((Y_R^k)^\top Y_R^k)^{-1}\|\|Y_R^k\|\| b\| + \|Y_G^k((Y_G^k)^\top Y_G^k)^{-1}\|\|Y_R^k-Y_G^k\|\|b\|\\
        & \hspace{-15.5ex} \leq \tfrac{7}{4}L(k+1)C_2c_{k,5}\|b\|^2+ \sqrt{c_4}c_{k,2}\|b\|^2=\zeta_{k+1}\|b\|^2.
    \end{align*}
    This shows $\|x^k_G-x^{k}_R\|\leq \zeta_k\|b\|^2$ by induction. Mimicking \eqref{eq4-1}, we now obtain: 
    \begin{align*}
        \|\tilde x^k_G-\tilde x^k_R\|& \leq \|L^{-1}(A-B)(\bar x^k_G-x^0)\|+\|(I-L^{-1}B)(\bar x^k_G-\bar x^k_R)\|\\
        &\leq L^{-1}C_1C_2\|b\|^2+\|\bar x^k_G-\bar x^k_R\|\leq (L^{-1}C_1C_2+c_1)\|b\|^2.        
    \end{align*}
     Therefore, it suffices to choose $C:= \max\{ \max_{k=0,\dots,m}\zeta_k, L^{-1}C_1C_2+c_1\}$.
\end{proof}

\subsection{Connecting $\CR$ and $\CG$}
\label{cr_cg}
In this subsection, we assume that the matrix $B$ is symmetric and positive definite. Suppose we apply $\CR$ to the linear system $B(x-x^0)=b$ starting at $x^0$. Then, by \cref{cr_property}, we have:
\[    x^k_R={\argmin}_{x\in x^0+\cK^{k}(B,b)} \|B(x-x^0)-b\|^2= {\argmin}_{x\in x^0+\cK^{k}(B,b)} \|B(x-x^*)\|,                             \]
where $x^*:=B^{-1}b+x^0$ is the optimal solution of the linear system $B(x-x^0)=b$. Next, for $k=0,\dots,m$, we set $y^*:=B^{\frac12}x^*$, $\bar b :=B^{\frac12}b$ and $y^k:=B^{\frac12}x^k_R$. Then, by definition, we obtain:
\[    y^k={\argmin}_{y\in y^0+\cK^k(B,\bar b)}~(y-y^*)^\top B(y-y^*).             \]
According to \cite[Theorem 2]{gutknecht2007brief}, this means that $y^k$ coincides with the $k$-th iteration of $\CG$ applied to the linear system $B(y-y^*)=0$ with initial value $y^0 = B^\frac12 x^0$. Moreover, in this case, it follows $(x^k_R-x^*)^\top B(x^k_R-x^*)=\|y^k-y^*\|^2$. Based on this observation and connection between the $\CR$- and $\CG$-iterates, we now want to apply classical techniques for $\CG$, \cite{hestenes1952methods}, to study the behavior of the distance $\|y^k-y^*\|$ as the iteration $k$ increases. Our goal is to then transfer the obtained results back to $\CR$ and $\AAr$. As usual, we define the following terms: $\bar y^k=y^k-L^{-1}B(y^k-y^*)$,
\[     \tilde y^k=\bar y^k-L^{-1}B(\bar y^k-y^*), \quad \text{and} \quad \psi(y)=\frac{1}{2}(y-y^*)^\top B(y-y^*).\]         
Notice that the introduced linear transformations also preserve the latter gradient descent steps, i.e., it holds that $\bar y^k=B^{\frac{1}{2}}\bar x^k_R$ and $\tilde y^k=B^{\frac{1}{2}}\tilde x^k_R$. Here, the point $\bar y^k$ is obtained by applying one gradient step (for the objective function $\psi$) with stepsize $L^{-1}$ on the $\CG$-iteration $y^k$ and $\tilde y^k$ results from applying two gradient steps with stepsize $L^{-1}$ on $y^k$. Next, we collect several results from \cite{hestenes1952methods} for convenience and to fix the notations. The full $\CG$ algorithm is shown in \cref{algo3}.  

\begin{algorithm}
    \caption{$\CG$ for the linear system $B(y-y^*)=0$.}
    \label{algo3}
    \begin{algorithmic}[1] 
        \STATE Choose an initial point $y^0 \in \Rn$ and set $p^0=r^0=-B(y^0-y^*)$.
        \FOR{$i=0,1,\dots,n$}
        \STATE \textbf{if} {$\|r^{i}\|=0$} \textbf{then} Break; \textbf{end if} 
             \STATE $a_i=\frac{\|r^i\|^2}{\langle p^i,Bp^i\rangle} $.
             \STATE $y^{i+1}=y^i+a_ip^i$.
             \STATE $r^{i+1}=r^i-a_iBp^i.$
             \STATE $b_i=\frac{\|r^{i+1}\|^2}{\|r^i\|^2}$.
             \STATE $p^{i+1}=r^{i+1}+b_ip^i$.
        \ENDFOR
    \end{algorithmic}
\end{algorithm}
\begin{prop} \label{prop:cg}
    Let the sequence $\{y^k\}_k$ be generated by $\CG$ and let $y^{(k)}$ denote the projection of $y^{*}$ onto the affine space $\cA^k := y^0+\spa\{y^1-y^0,\dots,y^k-y^0\}$. Then, the following properties are satisfied:
\begin{enumerate}[label=\textup{\textrm{(\roman*)}},topsep=0pt,itemsep=0ex,partopsep=0ex]
        \item (\cite[Theorem 6.5]{hestenes1952methods}) $y^{(k+1)}=y^{k+1}+\frac{2\psi(y^{k+1})}{\|r^{k}\|^2}p^{k}.$
        \item (\cite[Equation (5:3a)]{hestenes1952methods}) For all $i \neq j$: $\langle r^i,r^j\rangle=0$.
        \item (\cite[Equation (5:3c)]{hestenes1952methods}) For all $i < j$, we have $\langle p^i,r^j\rangle=0$ and for all $i\geq j$, it holds that $\langle p^i,r^j\rangle=\|r^i\|^2$.
        \item (\cite[Equation (5:6b)]{hestenes1952methods}) For all $0\leq i\leq n$: $\|p^{i}\|^{2}=\|r^i\|^4\sum_{j=0}^i\frac{1}{\|r^j\|^2}$.
        \item (\cite[Equation (5:11)]{hestenes1952methods}) For all $0\leq i\leq n-1$: $\langle r^{i+1},Br^i\rangle=\langle r^{i+1},Bp^i \rangle=-a_i^{-1}{\|r^{i+1}\|^2}$.
        \item (\cite[Equation (5:6a)]{hestenes1952methods}) For all $0\leq i\leq j\leq n$: $\langle p^i,p^j\rangle=\frac{\|r^j\|^2\|p^i\|^2}{\|r^i\|^2}.$
        \item (\cite[Equation (5:4b)]{hestenes1952methods}) For all $i\neq j$: $\langle p^i,Bp^j \rangle=0$.
        \item (\cite[Equation (5:3d)]{hestenes1952methods}) For all $i\neq j, i\neq j+1$: $\langle r^i,Bp^j \rangle=0$.
        \item (\cite[Equation (5:12)]{hestenes1952methods}) We have $a_0=\frac{\|r^0\|^2}{\langle r^0,Br^0\rangle}$ and $\frac{\|p^i\|^2}{\langle p^i,Bp^i\rangle}>a_i>\frac{\|r^i\|^2}{\langle r^i,Br^i\rangle}$ for all $i > 0$.
        \item (\cite[Equation (5:8b)]{hestenes1952methods}) For all $i\geq 1$: $r^{i+1}=(1+b^\prime_{i-1})r^i-a_iBr^i-b^\prime_{i-1}r^{i-1}$, where $b^{\prime}_{i-1}=\frac{a_i}{a_{i-1}}b_{i-1}=\frac{a_i}{a_{i-1}}\frac{\|r^{i}\|^2}{\|r^{i-1}\|^2}$. 
    \end{enumerate}
\end{prop}
 
The properties stated in \cref{prop:cg}  will be referred to as Property (\romannumeral1)--(\romannumeral10) in the following. Before studying the convergence behavior of $\CG$ in terms of $\|\bar y^k-y^*\|$ and $\|\tilde y^{k-1}-y^*\|$, let us briefly discuss our underlying motivation.

Our final aim is to prove $f(x^{k+1})\leq f(g(x^k))$ (including some potential higher-order error terms), where $\{x^k\}_k$ is generated by {$\AAr$}. Applying \cref{prop4-4}, the $\AAn$ step $x^{k+1}$ is equal to $\bar x^{k}_G$, which is close to $\bar x^{k}_R$. On the other hand, we have $g(x^k)=\bar g(x^{k})=\bar g(\bar x^{k-1}_G)=\tilde x^{k-1}_G$, which is close to $\tilde x^{k-1}_R$. Hence, up to certain error terms, the descent condition ``$f(x^{k+1})\leq f(g(x^k))$'' can now be formulated as follows:
\[  f(\bar x^{k}_R) \leq f(\tilde x^{k-1}_{R}).           \]
Expanding $f$ at $x^0$ (and again ignoring higher-order error terms), this can be further rewritten as:
\begin{align*}
    \nabla f(x^0)^\top(\bar x^k_R-x^0)+\frac12(\bar x^k_R-x^0)^\top\nabla^2f(x^0)(\bar x^k_R-x^0) & \\ & \hspace{-40ex} \leq \nabla f(x^0)^\top(\tilde x^{k-1}_R-x^0)+\frac{1}{2}(\tilde x^{k-1}_R-x^0)^\top\nabla f(x^0)(\tilde x^{k-1}_R-x^0). 
\end{align*}
Noticing $B=\nabla^2f(x^0)$, $b=-\nabla f(x^0)$, and $x^*=B^{-1}b+x^0$, this is equivalent to
\[        (\bar x^k_R-x^*)^\top B(\bar x^{k}_R-x^*)  \leq (\tilde x^{k-1}_R-x^*)^\top B(\tilde x^{k-1}_R-x^*),                        \] 
which, by the previously introduced transformation, can be expressed as $\|\bar y^k-y^*\|^2\leq \|\tilde y^{k-1}-y^*\|^2$. This is exactly what we want to show in \cref{thm4-9}. We note that the proof of \cref{thm4-9} would be significantly easier if the stepsize in the gradient mapping $g$ is sufficiently small (potentially much smaller than $L^{-1}$). Here, we provide a general result covering the core case $g(x) = x-\frac{1}{L}\nabla f(x)$. 
\begin{thm}
    \label{thm4-9}
    Suppose that  $\{y^k\}_k$ is generated by $\CG$ applied to the linear system $B(y-y^*)=0$, where $B \in \R^{n \times n}$ is symmetric, positive definite with $\nu := \frac{L}{\|B\|}\geq 1$. Then, we have:
    \begin{equation}   \|\bar y^{k}-y^*\|^2+\left[2\nu+\frac{1}{\nu^2}-3\right]\frac{\|r^k\|^2}{L^2}+\left[\nu+\frac{1}{\nu}-2\right]^2\frac{\|r^{k-1}\|^2}{L^2}\leq \|\tilde y^{k-1}-y^*\|^2. \label{eq:a-nice-result}              \end{equation}
\end{thm}
\begin{proof}
    First, by \cite[Equation (21)]{gutknecht2007brief} and \cite[Theorem 5.3]{jorge2006numerical}, we have $\cA^k=y^0+\cK^{k}(B,r^0)$ and $y^k \in y^0+\cK^{k}(B,r^0)$. Hence, both $\bar y^k$ and $\tilde y^{k-1}$ belong to the affine space $y^0+\cK^{k+1}(B,r^0)=\cA^{k+1}$. Furthermore, by the definition of $y^{(k+1)}$, we can derive the following decomposition properties:
    \begin{equation}
       \begin{aligned}
               \|\bar y^{k}-y^*\|^2&=\|y^{(k+1)}-\bar y^{k}\|^2+\| y^{(k+1)}-y^*\|^2,\\
         \|\tilde y^{k-1}-y^*\|^2&=\|y^{(k+1)}-\tilde y^{k-1}\|^2+\| y^{(k+1)}-y^*\|^2.           
       \end{aligned} 
       \label{dist_decom}
    \end{equation}
    Therefore, it holds that:
    \begin{align}
        \label{norm_diff}
        \|\tilde y^{k-1}-y^*\|^2-\|\bar y^{k}-y^*\|^2&= \|y^{(k+1)}-\tilde y^{k-1}\|^2- \|y^{(k+1)}-\bar y^{k}\|^2\notag\\
        &=\|\tilde y^{k-1}-\bar y^k\|^2+2\langle \tilde y^{k-1}-\bar y^k,\bar y^k-y^{(k+1)}\rangle.        
    \end{align}
    Using Property (\rmnum1) and the definition of the $\CG$-step, we have $y^{(k+1)}=y^{k+1}+\frac{2\psi(y^{k+1})}{\|r^{k}\|^2}p^{k}$ and $y^{k+1}=y^k+a_kp^k$. Consequently, setting $\gamma_k={2\psi(y^{k+1})}/{\|r^{k}\|^2}+a_k$ and applying $r^k = - B(y^k-y^*)$, we obtain
    \begin{align}
        \label{eqpy}
        y^{(k+1)}-\bar y^k=\left[\frac{2\psi(y^{k+1})}{\|r^{k}\|^2}+a_k\right]p^{k} + [y^{k}-\bar y^k]= \gamma_kp^k-\frac{1}{L}r^k. 
    \end{align}
    Moreover, we have $y^k-\bar y^{k-1}=a_{k-1}p^{k-1}-\frac{1}{L}r^{k-1}$ and
    \begin{align}
        \bar y^k-\tilde y^{k-1}=(I-L^{-1}B)(y^k-\bar y^{k-1})=(I-L^{-1}B)(a_{k-1}p^{k-1}-{L}^{-1}r^{k-1}). \label{eqy}  
    \end{align}
    We now consider the first term in \eqref{norm_diff}:
    \begin{align*}
        \|\tilde y^{k-1}-\bar y^k\|^2=\|(I- L^{-1}B)(a_{k-1}p^{k-1}-{L}^{-1}r^{k-1})\|^2 = T_1 - {2L^{-1}}T_2 + L^{-2} T_3,
    \end{align*}
    where $T_1 = a_{k-1}^2\|(I-L^{-1}B)p^{k-1}\|^2$, $T_2 = \langle a_{k-1}(I-L^{-1}B)p^{k-1},(I-L^{-1}B)r^{k-1}\rangle$, and $T_3 = \|(I-L^{-1}B)r^{k-1}\|^2$. The update rule for $r^k$ yields
    \begin{align}
       \label{eq-rp}
        a_{k-1}Bp^{k-1}=r^{k-1}-r^k.       
    \end{align}
    We first expand the term $T_1$:
    \begin{align*}
        T_1&= \|a_{k-1}p^{k-1}-{L}^{-1}(r^{k-1}-r^k)\|^2\\
        &=a_{k-1}^2\|p^{k-1}\|^2-{2a_{k-1}}L^{-1}\langle p^{k-1},r^{k-1}-r^k\rangle+L^{-2}\|r^{k-1}-r^k\|^2.
    \end{align*}
    Applying Property (\rmnum2) and (\rmnum{3}), it holds that:
    \[   \|r^{k-1}-r^k\|^2=\|r^{k-1}\|^2+\|r^k\|^2, \quad   \langle p^{k-1},r^{k-1}-r^k\rangle=\|r^{k-1}\|^2,                                    \]
    and thus, it follows $T_1 = a_{k-1}^2\|p^{k-1}\|^2-\frac{2a_{k-1}}{L}\|r^{k-1}\|^2+\frac{1}{L^2}(\|r^{k-1}\|^2+\|r^k\|^2)$. Next, we estimate the  term $T_2$:
    \[     T_2=\langle a_{k-1}p^{k-1},r^{k-1}\rangle-{2}L^{-1}\langle a_{k-1}Bp^{k-1},r^{k-1}\rangle+L^{-2}\langle a_{k-1}Bp^{k-1},Br^{k-1} \rangle.                       \]
    By Property (\rmnum3), we have $\langle a_{k-1}p^{k-1},r^{k-1}\rangle=a_{k-1}\|r^{k-1}\|^2$. Furthermore, applying \eqref{eq-rp} and Property (\rmnum2), we obtain $\langle a_{k-1}Bp^{k-1},r^{k-1}\rangle= \langle r^{k-1}-r^{k},r^{k-1}\rangle=\|r^{k-1}\|^2$ and $\langle a_{k-1}Bp^{k-1},Br^{k-1} \rangle=\langle r^{k-1}-r^k,Br^{k-1}\rangle$. Utilizing Property (\rmnum5), we can infer:
    \[  \langle a_{k-1}Bp^{k-1},Br^{k-1} \rangle = \langle r^{k-1}-r^k,Br^{k-1}\rangle=\|r^{k-1}\|_B^2+a_{k-1}^{-1}\|r^k\|^2.                \]
    Substituting these expressions yields $T_2 =a_{k-1}\|r^{k-1}\|^2-\frac{2}{L}\|r^{k-1}\|^2+\frac{1}{L^2}\|r^{k-1}\|_B^2+\frac{1}{L^2a_{k-1}}\|r^k\|^2$. Finally, let us consider the term $T_3$; we have:
    \begin{align*}
        T_3=\|r^{k-1}\|^2-2L^{-1}\|r^{k-1}\|_B^2+L^{-2}\|Br^{k-1}\|^2. 
    \end{align*}
    Together, this establishes the following representation of $\|\tilde y^{k-1}-\bar y^k\|^2$:
    \begin{align}
        \label{eq:ynormsquaretermform1}
        \|\tilde y^{k-1}-\bar y^k\|^2 & = a_{k-1}^2\|p^{k-1}\|^2+\left[\frac{6}{L^2}-\frac{4a_{k-1}}{L}\right]\|r^{k-1}\|^2 \\ & \hspace{4ex}-\frac{4}{L^3}\|r^{k-1}\|_B^2+\frac{1}{L^4}\|Br^{k-1}\|^2+\left[\frac{1}{L^2}-\frac{2}{L^3a_{k-1}}\right]\|r^k\|^2. \nonumber 
    \end{align}
   We continue with the inner product term $\langle \tilde y^{k-1}-\bar y^k,\bar y^k-y^{(k+1)}\rangle $. By \eqref{eqpy} and \eqref{eqy}, we have:
    \begin{align*}
        \langle \tilde y^{k-1}-\bar y^k,\bar y^k-y^{(k+1)}\rangle& =\langle  (I-L^{-1}B)(a_{k-1}p^{k-1}-{L}^{-1}r^{k-1}),\gamma^kp^k-{L}^{-1}r^{k}\rangle  \\
        & = Q_1 - L^{-1} Q_2,
    \end{align*}
    where $Q_1 = \langle a_{k-1}p^{k-1}-\frac{1}{L}r^{k-1},\gamma_kp^k-\frac{1}{L}r^k \rangle$ and $Q_2 = \langle a_{k-1}Bp^{k-1}-\frac{1}{L}Br^{k-1},\gamma_kp^k-\frac{1}{L}r^k\rangle$. Applying Property (\rmnum6), (\rmnum2) and (\rmnum3), it holds that:
    \[  \langle p^{k-1},p^k\rangle=\frac{\|r^k\|^2\|p^{k-1}\|^2}{\|r^{k-1}\|^2},\quad \langle r^{k-1}, p^k\rangle=\|r^k\|^2,\quad \langle p^{k-1},r^k \rangle=\langle r^k,r^{k-1}\rangle =0,                           \]
     which implies $Q_1=a_{k-1}\gamma_k\langle p^{k-1},p^k\rangle-\frac{\gamma_k}{L}\langle r^{k-1},p^k\rangle-\frac{a_{k-1}}{L}\langle p^{k-1},r^k\rangle+\frac{1}{L^2}\langle r^{k},r^{k-1}\rangle = a_{k-1}\gamma_k\frac{\|r^k\|^2\|p^{k-1}\|^2}{\|r^{k-1}\|^2}-\frac{\gamma_k}{L}\|r^k\|^2$. Similarly, we can expand $Q_2$ as follows:
    \begin{align*}
        Q_2 &=a_{k-1}\gamma_k\langle Bp^{k-1},p^k\rangle-\frac{\gamma_k}{L}\langle Br^{k-1},p^k\rangle-\frac{a_{k-1}}{L}\langle Bp^{k-1},r^k\rangle+\frac{1}{L^2}\langle Br^{k},r^{k-1}\rangle.
    \end{align*}
    Applying Property (\rmnum7), (\rmnum8), (\rmnum{2}), (\rmnum3), (\rmnum5), and \eqref{eq-rp}, we can infer $ \langle Bp^{k-1},p^k\rangle=0$, $\langle Br^{k-1},p^k\rangle=0$, $\langle a_{k-1}Bp^{k-1},r^k\rangle=\langle r^{k-1}-r^k,r^k\rangle=-\|r^k\|^2$, and $\langle Br^k,r^{k-1}\rangle=-\frac{\|r^k\|^2}{a_{k-1}}$, which yields $Q_2=\frac{1}{L}(1-\frac{1}{La_{k-1}})\|r^{k}\|^2.$ Therefore, the inner product term $\langle \tilde y^{k-1}-\bar y^k,\bar y^k-y^{(k+1)}\rangle $ is given by:
      \[  \langle \tilde y^{k-1}-\bar y^k,\bar y^k-y^{(k+1)}\rangle=a_{k-1}\gamma_k\frac{\|r^k\|^2\|p^{k-1}\|^2}{\|r^{k-1}\|^2}-\left[\frac{\gamma_k}{L}+\frac{1}{L^2}-\frac{1}{L^3a_{k-1}}\right]\|r^k\|^2. \]
    Summing \eqref{eq:ynormsquaretermform1} and the previous expression, we obtain:
    \begin{align}
        \nonumber &\|\tilde y^{k-1}-\bar y^k\|^2+2  \langle \tilde y^{k-1}-\bar y^k,\bar y^k-y^{(k+1)}\rangle\\  
        \nonumber & \hspace{5ex}=a_{k-1}^2\|p^{k-1}\|^2+\left[\frac{6}{L^2}-\frac{4a_{k-1}}{L}\right]\|r^{k-1}\|^2-\frac{4}{L^3}\|r^{k-1}\|_B^2+\frac{1}{L^4}\|Br^{k-1}\|^2 \\ & \hspace{9ex}+\left[2a_{k-1}\gamma_k\frac{\|r^k\|^2\|p^{k-1}\|^2}{\|r^{k-1}\|^2}-\frac{2\gamma_k}{L}\|r^k\|^2-\frac{1}{L^2}\|r^k\|^2 \right]. \label{eq:let-it-be-a-number}
    \end{align}
    We continue with two sub-cases.
     
        \textbf{Case 1:} $k=1$. Using the fact $r^0=p^0$, Property~(\rmnum2), and \eqref{eq-rp}, it follows $\|Br^0\|^2=\|Bp^0\|^2=a_0^{-2}\|r^0-r^1\|^2=a_0^{-2}(\|r^0\|^2+\|r^1\|^2)$ and $\langle r^0,Br^0 \rangle=\langle r^0,Bp^0\rangle=a_0^{-1}\langle r^0, r^0-r^1\rangle=a_0^{-1}\|r^0\|^2$. Using these two equalities, we can simplify \eqref{eq:let-it-be-a-number} to: 
       \begin{align*} \|\tilde y^{0}-\bar y^1\|^2+2  \langle \tilde y^{0}-\bar y^1,\bar y^1-y^{(2)}\rangle & \\ & \hspace{-25ex}=a_{0}^2\|r^{0}\|^2+\left[\frac{6}{L^2}-\frac{4a_{0}}{L}\right]\|r^{0}\|^2-\frac{4}{L^3a_0}\|r^{0}\|^2+\frac{1}{L^4a_0^2}(\|r^0\|^2+\|r^{1}\|^2)\\
        & \hspace{-21ex}+\left[2a_{0}\gamma_1\|r^1\|^2-2\gamma_1L^{-1}\|r^1\|^2-L^{-2}\|r^1\|^2 \right] \\ 
        & \hspace{-25ex}=L^{-2}(Q_3\|r^0\|^2+Q_4\|r^1\|^2),
        \end{align*}
        where $Q_3$ and $Q_4$ are defined as $Q_3=a_0^2L^2+6-4a_0L-{4}(a_0L)^{-1}+(a_0L)^{-2}$ and $Q_4=2a_0\gamma_1L^2-2\gamma_1L-1+(a_0L)^{-2}$. By Property (\rmnum{9}), we have $a_0={\|r^0\|^2}/{\langle r^0, Br^0\rangle}\geq \frac{1}{\|B\|}\geq \frac{1}{L}$ and $a_1>{\|r^1\|^2}/{\langle r^1,Br^1 \rangle}\geq\frac{1}{\|B\|}\geq  \frac{1}{L}$. Hence, by the definition of $\gamma_1$, we can infer $ \gamma_1L\geq a_1L>{L}\|B\|^{-1}=\nu\geq 1$ and $a_0L\geq \nu\geq 1$. Therefore, it holds that:
        \begin{align*}
            Q_4&=2\gamma_1L(a_0L-1)-1+(a_0L)^{-2}\geq 2a_0L+(a_0L)^{-2}-3 \geq 2\nu+\nu^{-2}-3\geq 0,
        \end{align*}
        where -- in the last equality -- we used the fact that the function $x\mapsto 2x+x^{-2}$ is monotonically increasing for $x\geq 1$. Concerning $Q_3$, we notice:
        \[  Q_3 = (a_0L+{(a_0L)^{-1}}-2)^2\geq 0.   \]
        Since $x\mapsto x+\frac{1}{x}-2$ is monotonically increasing and nonnegative for $x\in [1,\infty)$, we can further infer $Q_3 = (a_0L+{(a_0L)^{-1}}-2)^2 \geq (\nu+\nu^{-1}-2)^2$, Together, we obtain $\|\bar y^{1}-y^*\|^2+(\nu+\nu^{-1}-2)^2\tfrac{\|r^0\|^2}{L^2}+(2\nu+\nu^{-2}-3)\tfrac{\|r^1\|^2}{L^2}\leq \|\tilde y^{0}-y^*\|^2$. \vspace{0.5ex}
       
        \textbf{Case 2:} $k\geq 2$. We first utilize Property (\rmnum10): $a_{k-1}Br^{k-1}=(1+b_{k-2}^{\prime})r^{k-1}-r^k-b^{\prime}_{k-2}r^{k-2}$. Along with Property (\rmnum2), this allows to calculate $\|r^{k-1}\|_B^2$ and $\|Br^k\|^2$: 
        \begin{equation*}
            \begin{aligned}
                & \langle r^{k-1},Br^{k-1}\rangle =\frac{1}{a_{k-1}}\langle r^{k-1},(1+b_{k-2}^{\prime})r^{k-1}-r^k-b^{\prime}_{k-2}r^{k-2}\rangle=\frac{1+b^\prime_{k-2}}{a_{k-1}}\|r^{k-1}\|^2, \\
               & \|Br^{k-1}\|^2 =\langle Br^{k-1},Br^{k-1}\rangle=\frac{(1+b^\prime_{k-2})^2}{a_{k-1}^2}\|r^{k-1}\|^2+\frac{1}{a_{k-1}^2}\|r^k\|^2+\frac{(b^{\prime}_{k-2})^2}{a_{k-1}^2}\|r^{k-2}\|^2. 
            \end{aligned}
            \end{equation*}
        Therefore, the term \eqref{eq:let-it-be-a-number} can be decomposed as follows: $\|\tilde y^{k-1}-\bar y^k\|^2+2  \langle \tilde y^{k-1}-\bar y^k,\bar y^k-y^{(k+1)}\rangle = Q_5 + Q_6$, where
        \begin{align*}
            Q_5 &=a_{k-1}^2\|p^{k-1}\|^2+\left[\tfrac{6}{L^2}-\tfrac{4a_{k-1}}{L}-\tfrac{4(1+b^{\prime}_{k-2})}{L^3a_{k-1}}+\tfrac{(1+b^{\prime}_{k-2})^2}{L^4a_{k-1}^2} \right]\|r^{k-1}\|^2+\tfrac{(b^{\prime}_{k-2})^2}{a_{k-1}^2L^4}\|r^{k-2}\|^2\\
            Q_6 & = 2a_{k-1}\gamma_k\tfrac{\|r^k\|^2\|p^{k-1}\|^2}{\|r^{k-1}\|^2}-\tfrac{2\gamma_k}{L}\|r^k\|^2-\tfrac{1}{L^2}\|r^k\|^2+\tfrac{1}{L^4a_{k-1}^2}\|r^k\|^2.
        \end{align*}
        We start with bounding $Q_6$. First, by Property (\rmnum4), it holds that:
        \[  2a_{k-1}\gamma_k\frac{\|r^k\|^2\|p^{k-1}\|^2}{\|r^{k-1}\|^2}\geq  2a_{k-1}\gamma_k\frac{\|r^k\|^2\|r^{k-1}\|^2}{\|r^{k-1}\|^2}=2a_{k-1}\gamma_k\|r^k\|^2.            \]
        Thus, we have $Q_6\geq \frac{1}{L^2} [2a_{k-1}\gamma_kL^2-2\gamma_kL-1+(a_{k-1}L)^{-2}]\|r^k\|^2$. The coefficient in the parentheses can be shown to be larger or equal than $2\nu+\nu^{-2}-3$ by using the same strategy as in \textbf{Case 1} for $Q_4$. This yields $Q_6\geq (2\nu+\nu^{-2}-3) L^{-2}{\|r^k\|^2}$. Next, recalling the definition of $b^{\prime}_{k-2}$ (see Property (\rmnum10)), we obtain:
        \[   b^{\prime}_{k-2}=\frac{a_{k-1}}{a_{k-2}}\frac{\|r^{k-1}\|^2}{\|r^{k-2}\|^2} \quad \implies \quad \frac{(b^{\prime}_{k-2})^2}{a_{k-1}^2L^4}\|r^{k-2}\|^2=\frac{\|r^{k-1}\|^2}{a_{k-2}^2L^4\|r^{k-2}\|^2} \|r^{k-1}\|^2.                            \]
        In addition, by Property (\rmnum4), it follows:
        \[  a_{k-1}^2\|p^{k-1}\|^2=a_{k-1}^2\|r^{k-1}\|^4{\sum}_{j=0}^{k-1}\frac{1}{\|r^j\|^2}\geq a_{k-1}^2\|r^{k-1}\|^2+a_{k-1}^2\frac{\|r^{k-1}\|^2}{\|r^{k-2}\|^2}\|r^{k-1}\|^2.                       \]
        Setting $Q_7 = a_{k-1}^2L^2 + 6 - 4a_{k-1}L - 4(a_{k-1}L)^{-1} + (a_{k-1}L)^{-2} = (a_{k-1}L + (a_{k-1}L)^{-1}-2)^2$ and using the previous inequalities, we can lower bound $Q_5$ by:
        \begin{align*} Q_5 & \geq\left[Q_7-\tfrac{4b^{\prime}_{k-2}}{a_{k-1}L}+\tfrac{(1+b^{\prime}_{k-2})^2-1}{a_{k-1}^2L^2} +\left[a_{k-1}^2L^2 + \tfrac{1}{a_{k-2}^2L^2}\right]\tfrac{\|r^{k-1}\|^2}{\|r^{k-2}\|^2} \right]\frac{\|r^{k-1}\|^2}{L^2}. \end{align*}
       Let us denote the term in parentheses by $Q_8$. It suffices to show that $Q_8$ is nonnegative. In particular, it holds that:
        \begin{align*}
        Q_8 & \geq Q_7 +\left[a_{k-1}L-\frac{1}{a_{k-2}L}\right]^2\frac{\|r^{k-1}\|^2}{\|r^{k-2}\|^2}+\frac{2a_{k-1}}{a_{k-2}}\frac{\|r^{k-1}\|^2}{\|r^{k-2}\|^2}-\frac{4b^{\prime}_{k-2}}{a_{k-1}L}+\frac{2b^{\prime}_{k-2}}{a_{k-1}^2L^2} \\
            & \geq Q_7 +2b^{\prime}_{k-2}-\frac{4b^{\prime}_{k-2}}{a_{k-1}L}+\frac{2b^{\prime}_{k-2}}{a_{k-1}^2L^2} \geq \left[\nu+\nu^{-1}-2\right]^2+2b_{k-2}^{\prime}\left[1-(a_{k-1}L)^{-1}\right]^2,
        \end{align*}
        where we again used $a_{k-1}L \geq \nu$. This finally establishes \eqref{eq:a-nice-result}, which concludes the proof of \cref{thm4-9}.  
\end{proof}
The previous result shows that performing two gradient steps on $y^{k-1}$ achieves less progress in terms of the distance to the optimal solution compared to performing one gradient step on $y^k$. In fact, we are able to prove a similar result for $\CG$, which is of independent interest. More precisely, $\CG$ can decrease the distance to the optimal solution no slower than the gradient method with stepsize $\frac1L$. Hence, $y^k$ can provide more progress than $\bar y^{k-1}$. The proof is much easier and is deferred to \cref{app:cg-proof}.

\begin{thm}
    \label{thm4-10}
    Let $\{y^k\}_k$ be generated by $\CG$ applied to the system $B(y-y^*)=0$, where $B \in \R^{n \times n}$ is symmetric, positive definite with $\|B\|\leq L$. Then, we have:
    \[  \|y^{k+1}-y^*\|^2\leq \|\bar y^{k}-y^*\|^2.               \]
\end{thm}

As discussed at the beginning of this section, the sequences $\{x_R^k\}_k$ and $\{y^k\}_k$ are equivalent up to a linear transformation, i.e., it holds that $y^k = B^\frac12 x_R^k$. This allows to transfer our obtained results back to the $\CR$ method. We summarize our observations for $\CR$ in the following corollary.
\begin{cor}
    \label{coro4-9}
    Let $B \in \R^{n \times n}$ be a symmetric, positive definite matrix with $\|B\|\leq L$ and let $x^0 \in \R^n$ be given. Suppose that $\{x^k_R\}_k$ is generated by the $\CR$ method to solve the linear system $B(x-x^0)=b$. Then, we have:
    \[        \vp(\bar x^k_R) \leq \vp(\tilde x^{k-1}_R) \quad \text{and} \quad \vp(x^k_R)\leq \vp(\bar x^{k-1}_R),                         \]
    where $\bar x^k_R$ and $\tilde x^k_R$ are defined in \eqref{defn_bar_x} and $\vp(x) := \frac12(x-x^0)^\top B(x-x^0)-b^\top (x-x^0)$.
\end{cor}

\subsection{Proof of \cref{thm3-18}}
\label{aa_descent}
In this subsection, we combine our obtained results and show that $\AAr$ locally decreases the objective function {no slower} than a gradient descent step with stepsize $\frac{1}{L}$ (up to a certain higher-order error term). 

Throughout this section, we will work with the following choices $B=\nabla^2f(x^0)$, $b=-\nabla f(x^0)$, and $A=B+E_m$, where $E_m$ is defined in \eqref{eq:lets-define-A}.

\begin{proof}[Proof of \cref{thm3-18}] Clearly, \eqref{eq:super-duper} holds for $k = 0$. Furthermore, we only need to verify \eqref{eq:super-duper} for one cycle of {$\AAr$} as all assumptions and results will also hold for subsequent cycles, since all the subsequent iterations would also belong to $U$ by \cref{prop-local}. Let $U=S_{\epsilon}$ be the neighborhood defined in \cref{prop-local}. Then, for all $k\in \n$, we have:
 \[      \|h(x^{k+1})\|\leq (1-(2\kap)^{-1})\|h(x^k)\|.                 \]
\cref{prop-HXeq} establishes $\kappa(H_k^\top H_k)\leq M_H^2$ for some $M_H>0$ and by \cref{lemma3-5}, we can infer $\|\hat g^k-x^0\|=\mathcal O(\|b\|)$. Due to $x^{k+1}=\hat g^k$, this just means $\|x^{k+1}-x^0\|=\mathcal O(\|b\|)$. Notice that this estimate holds for every $k=0,1,\dots,m$ and therefore, it follows $M_x^m=\mathcal O(\|b\|)$. Furthermore, using \eqref{bound_Em}, we obtain $\|E_m\|=\mathcal O(\|b\|)$. Reducing $\epsilon$ if necessary, we may assume that $\|A-B\|=\|E_m\|<\mu$, which ensures the invertibility of $A$ as shown in the proof of \cref{lemma4-7}. Now, let $\{x^k_G\}_k$ and $\{x^k_R\}_k$ be generated by $\GMRES$ and $\CR$ applied to the linear systems $A(x-x^0)=b$ and $B(x-x^0)=b$ with $x^0_G=x^0_R=x^0$, respectively. By \cref{prop4-4}, we have $\bar x^k_G=x^{k+1}$ for all $k=0,\dots,m$ and $\kappa((\bar X_G^k)^\top \bar X_G^k) \leq M^2$ for all $k \in [m]$. Moreover, since the perturbed gradient mapping $\bar g$ is exact at each $x^k$, $k=0,\dots,m$, it holds that 
 $$g(x^k)=\bar g(x^k)=\bar g(\bar x^{k-1}_G)=\tilde x^{k-1}_G \quad \forall~k=1,\dots,m. $$
 In addition, we have $\|\bar x^k_G-x^0\| = \|x^{k+1}-x^0\|=\mathcal O(\|b\|)$. Reducing $\epsilon$ --- if necessary --- we may assume $\epsilon\leq \epsilon_\sharp$, where $\epsilon_\sharp$ was introduced in the proof of \cref{thm4-6}. Thus, all conditions in \cref{thm4-6} are satisfied and it follows
 \begin{align}
    \label{bound_gr}
    \|\bar x^k_G-\bar x^k_R\|=\mathcal O(\|b\|^2) \quad \text{and} \quad \|\tilde x^k_G-\tilde x^k_R\|=\mathcal O(\|b\|^2) \quad \forall~k=0,\dots,m.  
 \end{align}
Moreover, since $g$ is a contraction on $\B_r(\xopt)$ and due to $\|x^k-x^0\|=\mathcal O(\|b\|)$, we have $\|g(x^k)-x^0\|\leq \|g(x^k)-g(x^0)\|+\|g(x^0)-x^0\|=\mathcal O(\|b\|)$. Reusing the notation from \cref{coro4-9}, the Lipschitz continuity of the Hessian $\nabla^2 f$ then implies
\begin{equation}
    \label{ftovp}
    \begin{aligned}
        f(x^{k+1})&=f(x^0)+\vp(x^{k+1})+\mathcal O(\|x^{k+1}-x^0\|^3)=f(x^0)+\vp(\bar x^k_G)+\mathcal O(\|b\|^3), \\
        f(g(x^k))&=f(x^0)+\vp(g(x^k))+\mathcal O(\|g(x^k)-x^0\|^3)=f(x^0)+\vp(\tilde x^{k-1}_G)+\mathcal O(\|b\|^3),   
    \end{aligned}
\end{equation}
see, e.g., \cite[Lemma 4.1.1]{nesterov2018lectures}. Since the mapping $\vp$ is quadratic, we can further write:
    \begin{align*}
        \vp(\bar x^k_G)&=\vp(\bar x^k_R)+\nabla \vp(\bar x^k_R)^\top(\bar x^k_G-\bar x^k_R)+\tfrac{1}{2}(\bar x^k_G-\bar x^k_R)^\top B(\bar x^k_G-\bar x^k_R), \\
        \vp(\tilde x^{k-1}_G)&=\vp(\tilde x^{k-1}_R)+\nabla \vp(\tilde x^{k-1}_R)^\top(\tilde x^{k-1}_G-\tilde x^{k-1}_R)+\tfrac{1}{2}(\tilde x^{k-1}_G-\tilde x^{k-1}_R)^\top B(\tilde x^{k-1}_G-\tilde x^{k-1}_R). 
    \end{align*}
Next, applying \cref{lemma4-7} for the case $A=B$, it holds that:
\[    \|\nabla \vp(\bar x^k_R)\| = \|B(\bar x^k_R-x^0)-b\| \leq \|B(x^k_R-x^0)-b\|\leq    \|B(x^0-x^0)-b\|=\|b\|,                                 \]
where we used \cref{cr_property} in the last step. Similarly, we can show $\|\nabla \vp(\tilde x^{k-1}_R)\|\leq\|b\|$. Thus, combining \eqref{bound_gr} and the representations of $\vp(\bar x^k_G)$ and $\vp(\tilde x^{k-1}_G)$, we obtain
\begin{equation*}
    \begin{aligned}
        &  |\vp(\bar x^k_G)-\vp(\bar x^k_R)|\leq \| \nabla \vp(\bar x^k_R)\|\|\bar x^k_G-\bar x^k_R\|+\tfrac{L}{2}\|\bar x^k_G-\bar x^k_R\|^2=\mathcal O(\|b\|^3),  \\
        & |\vp(\tilde x^{k-1}_G)-\vp(\tilde x^{k-1}_R)|\leq \| \nabla \vp(\tilde x^{k-1}_R)\|\|\tilde x^{k-1}_G-\tilde x^{k-1}_R\|+\tfrac{L}{2}\|\tilde x^{k-1}_G-\tilde x^{k-1}_R\|^2=\mathcal O(\|b\|^3), 
    \end{aligned}
\end{equation*}
Using these estimates in \eqref{ftovp}, we can infer
\[        f(x^{k+1})=f(x^0)+\vp(\bar x^k_R)+\mathcal O(\|b\|^3), \quad f(g(x^k))=f(x^0)+\vp(\tilde x^{k-1}_R)+\mathcal O(\|b\|^3). \]
The conclusion then follows immediately from \cref{coro4-9}.
\end{proof}


\section{A Function Value-Based Globalization for $\AAr$}
\label{globalization}
Based on the local descent properties established in the last section, we now propose a globalization mechanism for $\AAr$. We prove global convergence and provide simple global-to-local transition results for the globalized $\AAr$ algorithm. To the best of our knowledge, this is the first function value-based globalization of $\AAr$ that achieves both global and local convergence. Previously, only heuristic strategies seem to be available, see \cite{peng2018anderson,ouyang2020anderson,scieur2017nonlinear}. 

The full procedure is presented in \cref{algo4}. Our core idea is to check whether the $\AAn$ step $x_{\AAn}^k$ satisfies a sufficient decrease condition
\begin{equation} \label{eq:descent-condition}  f(x_{\AAn}^k) \leq f(x^k)-\gamma\|\nabla f(x^k)\|^2+\min\{c_1\|\nabla f(x^{k-\hat m})\|^\nu,c_2\|\nabla f(x^{k-\hat m})\|^2,c_3\}, \end{equation}
where $\gamma$, $\nu$, and $c_i$, $i = 1,2,3$, are given parameters. We accept the $\AAn$ step as new iterate $x^{k+1} = x_{\AAn}^k$ if condition \eqref{eq:descent-condition} holds. Otherwise, a gradient step $x^{k+1} = g(x^k)$ is performed (which ensures decrease of the objective function values). We summarize several basic convergence properties of \cref{algo4} in the following theorem. 

\begin{algorithm}[t]
    \caption{A Globalized $\AAn$ Scheme with Restarting}
    \label{algo4}
    \begin{algorithmic}[1] 
        \STATE Choose an initial point $x^0 \in \Rn$, the memory parameter $m$, and constants $\gamma,\nu,c_1,c_2,c_3>0$. Set $\hat m=0$. 
        \FOR{$k=0,1,\dots$}
        \STATE Set $\hat m = \mathrm{mod}(k,m+1)$.
        \IF{$\hat m=0$ }
        \STATE Set $x^{k+1}=g(x^k)$.
        \ELSE
        \STATE Calculate the coefficient $\alpha^k$ based on the sequence $\{h(x^k),\dots,h(x^{k-\hat m})\}$ by solving \eqref{AAsubp} and set $x^{k}_{\AAn}=g^{k-\hat m}+G_k\alpha^k$.
        \IF{$f(x^{k}_{\AAn})> f(x^k)-\gamma\|\nabla f(x^k)\|^2+\min\{c_1\|\nabla f(x^{k-\hat m})\|^\nu,c_2\|\nabla f(x^{k-\hat m})\|^2,c_3\}$}
        \STATE Set $x^{k+1}=g(x^k)$.
        \ELSE 
        \STATE Set $x^{k+1}=x^{k}_{\AAn}$.
        \ENDIF
        \ENDIF
        \ENDFOR
    \end{algorithmic}
\end{algorithm}

\begin{thm}
    \label{thm4-1}
    Suppose that \ref{A1} holds and let $f$ be bounded from below. Let the sequence $\{x^k\}_k$ be generated by \cref{algo4} with $\gamma, c_1, c_3 > 0$, $c_2 < \frac{1}{2mL}$, and $\nu \in (2,3)$. Then, we have
    \[    {\sum}_{k=0}^\infty \|\nabla f(x^{k})\|^2<\infty \quad \text{and} \quad \lim_{k\to\infty}\|\nabla f(x^{k})\|=0.              \]
    In addition, if $\gamma < \frac{1}{2L}$ and if the conditions \ref{A2}--\ref{A5} are satisfied with $\xopt$ being an accumulation point of $\{x^{k(m+1)}\}_k$, then we have $x^k\to \xopt$ and all $\AAn$ steps will be eventually accepted, i.e., \cref{algo4} locally turns into \cref{algo1}.
\end{thm}
\begin{proof}
    Notice that the $k$-th cycle starts at iteration $(k-1)(m+1)$ and ends at iteration $k(m+1)$ (with $x^{k(m+1)}$ serving as initial point for the next cycle). In order to keep the notation simple, we concentrate on the first cycle. Since the first iteration within each cycle is a gradient descent step, i.e., $x^1 = g(x^0)$, we can deduce $f(x^1)\leq f(x^0)-\frac{1}{2L}\|\nabla f(x^0)\|^2$. For all $k = 1,\dots,m$, the iterate $x^{k+1}$ either results from a gradient descent step or an $\AAn$ step satisfying the acceptance criterion:
    \[  f(x^{k+1})\leq f(x^k)-\gamma\|\nabla f(x^k)\|^2+c_2\|\nabla f(x^0)\|^2. \]
    Hence, each update $k = 1,\dots, m$ satisfies $f(x^{k+1}) \leq f(x^k) - \min\{\frac{1}{2L},\gamma\}\|\nabla f(x^k)\|^2 + c_2\|\nabla f(x^0)\|^2$.
    Summing these estimates from $1$ to $m$, we obtain
    \[   f(x^{m+1})\leq f(x^0) - \min\left\{\tfrac{1}{2L},\gamma\right\}{\sum}_{k=1}^m\|\nabla f(x^k)\|^2 -\left[\tfrac{1}{2L}-mc_2\right] \|\nabla f(x^0)\|^2.               \]
    Defining $\sigma := \min\{\frac{1}{2L},\gamma,\frac{1}{2L}-mc_2\} > 0$, this result holds for every cycle of \cref{algo4}, i.e., we have
    \[  f(x^{(k+1)(m+1)})\leq f(x^{k(m+1)})-\sigma {\sum}_{i = k(m+1)}^{k(m+1)+m} \|\nabla f(x^{i})\|^2 \quad \forall~k \in \n.    \]
    Summing this inequality for all $k\in\n$ and noticing that $f$ is bounded from below, it follows $\sum_{i=0}^\infty \|\nabla f(x^{i})\|^2<\infty$ which readily implies $\|\nabla f(x^{i})\|\to0$. Next, let $\xopt$ be an accumulation point of $\{x^{k(m+1)}\}_k$ satisfying \ref{A2}--\ref{A5}. By \cref{thm3-18}, there is a neighborhood $U$ of $\xopt$ such that if $y^0\in U$, then the sequence $\{y^k\}_k$ generated by \cref{algo1} satisfies
    \[      f(y^{k+1})\leq f(g(y^k))+\mathcal O(\|\nabla f(y^{k-\hat m})\|^3)\leq f(y^k)-\tfrac{1}{2L}\|\nabla f(y^k)\|^2+\mathcal O(\|\nabla f(y^{k-\hat m})\|^3).                  \]
    Thus, by shrinking $U$ if necessary and using $\gamma < \frac{1}{2L}$, we can assume 
    \begin{equation} \label{eq:yeah}  f(y^{k+1})\leq f(y^k)-\gamma \|\nabla f(y^k)\|^2+c_1\|\nabla f(y^{k-\hat m})\|^{\nu} \quad \forall~k.                  \end{equation}
    Since $\xopt$ is an accumulation point of $\{x^{k(m+1)}\}_k$, there exists $s$ with $x^{s(m+1)}\in U$. We now set $y^k := x^{k+s(m+1)}$. Due to $y^0 \in U$, $\|\nabla f(x^i)\| \to 0$, and since the conditions \ref{A1}--\ref{A5} are satisfied, we can inductively infer that every $\AAn$ step fulfills \eqref{eq:yeah} and is accepted as new iterate, i.e., we have $y^{k+1} = x^{s(m+1)+k+1} = x^{s(m+1)+k}_{\AAn}$, $k \geq 1$. Convergence of $\{x^k\}_k$ then follows from \cref{prop-local} and \ref{A2}.
\end{proof}


\section{Numerical Experiments}
\label{sec:exp}
In this section, we conduct preliminary numerical experiments to illustrate the performance and convergence behavior of $\AAr$ and to empirically verify the descent properties of Algorithm~\ref{algo4}\footnote{Code available under \url{https://github.com/yangliu-op/AndersonAcceleration}}.

\begin{figure}[t]
	\setlength{\abovecaptionskip}{-3pt plus 3pt minus 0pt}
	\setlength{\belowcaptionskip}{-10pt plus 3pt minus 0pt}
	\centering
	\includegraphics[width=13.0cm]{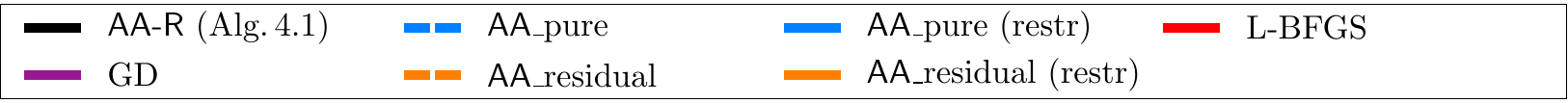}
	\hspace*{-1.2ex}
	\begin{tikzpicture}[scale=1]
	\node[right] at (0.0,0) {\includegraphics[width=3.35cm,trim=15 50 60 45,clip]{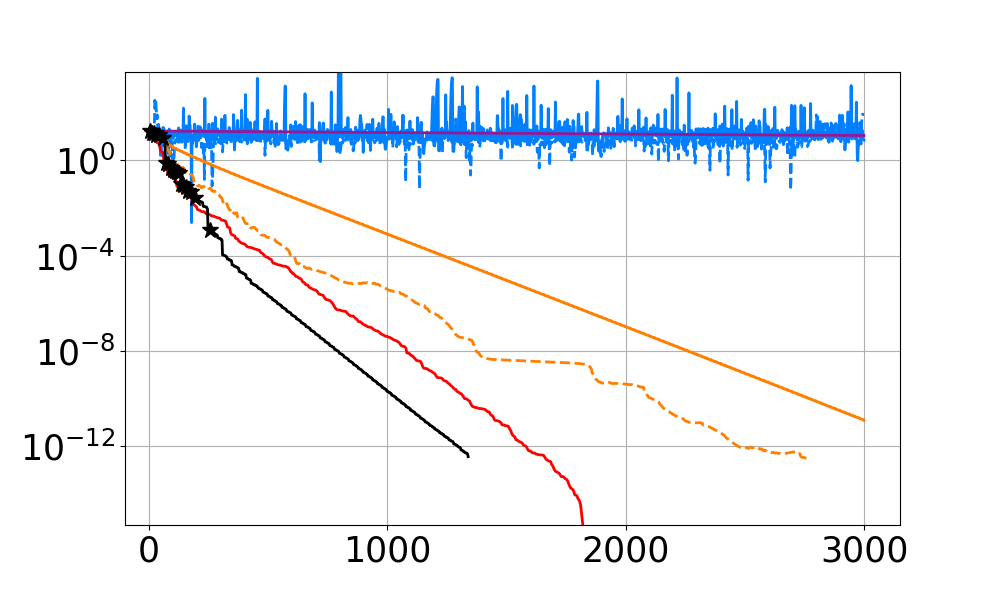}};
	\node[right] at (3.45,0) {\includegraphics[width=3cm,trim=85 50 60 45,clip]{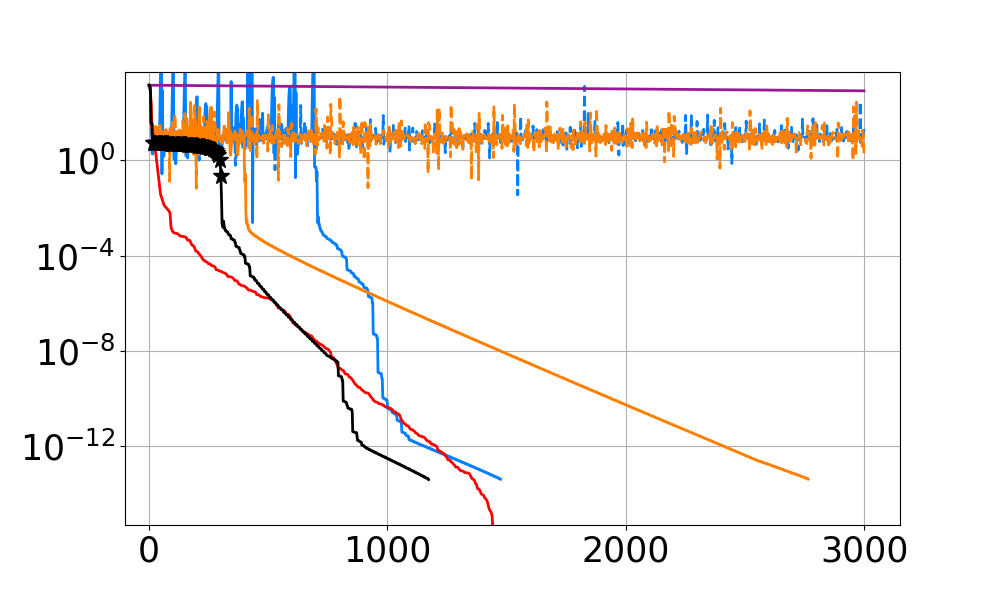}};
	\node[right] at (6.55,0) {\includegraphics[width=3cm,trim=85 50 60 45,clip]{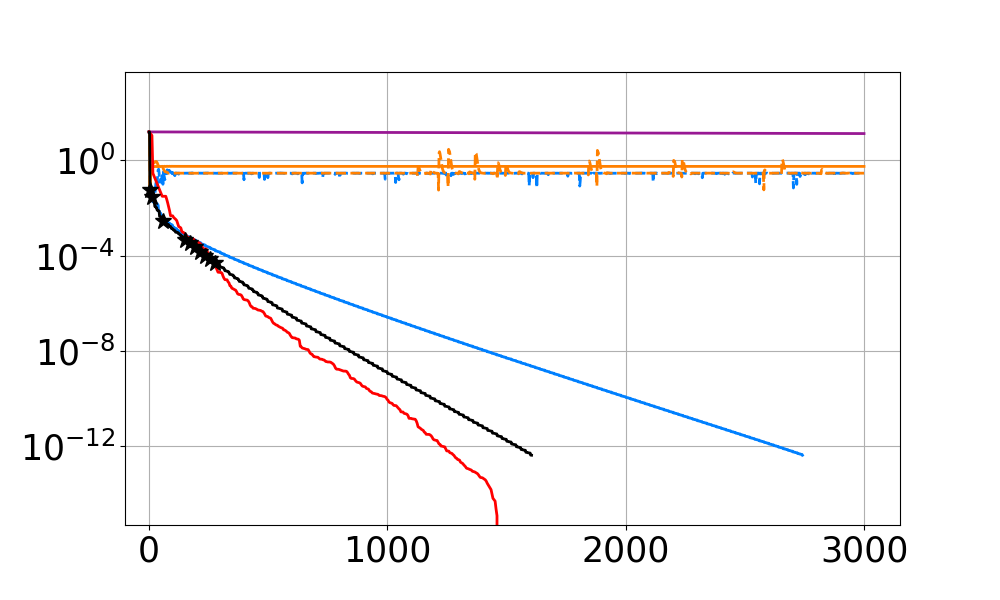}};
	\node[right] at (9.65,0) {\includegraphics[width=3cm,trim=85 50 60 45,clip]{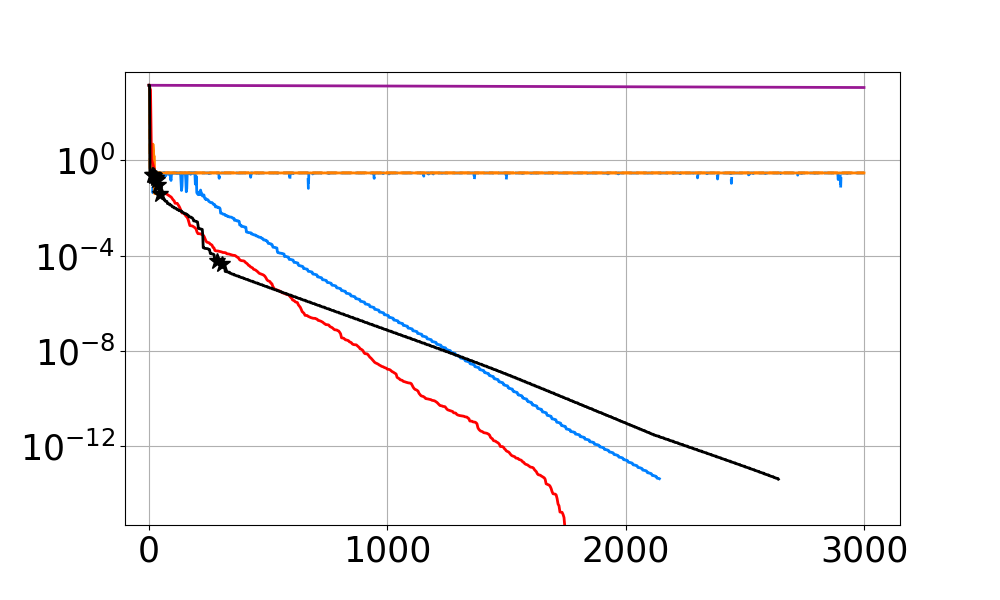}};
	\node[right] at (0.0,-1.9) {\includegraphics[width=3.35cm,trim=15 50 60 45,clip]{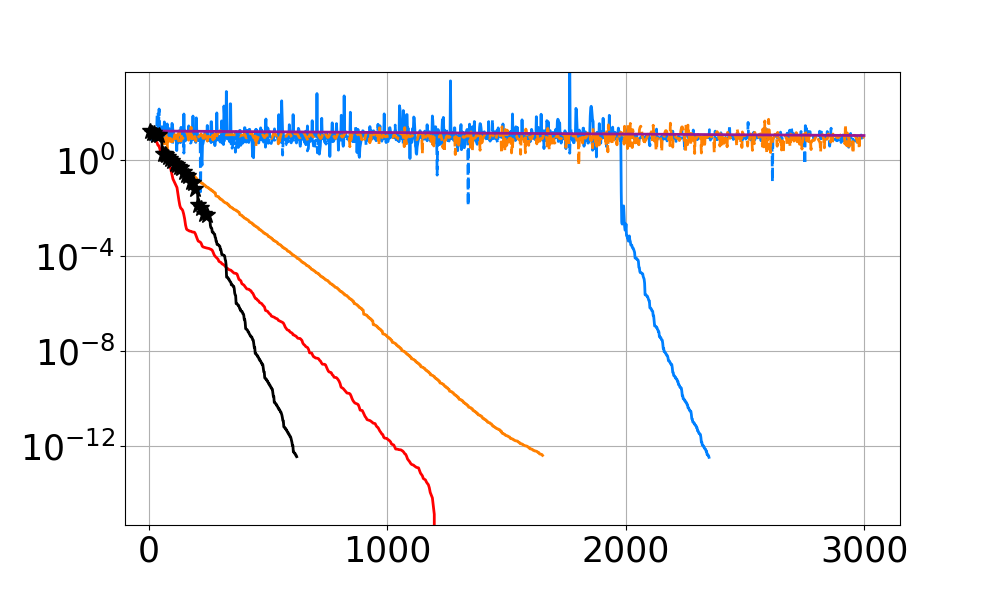}};
	\node[right] at (3.45,-1.9) {\includegraphics[width=3cm,trim=85 50 60 45,clip]{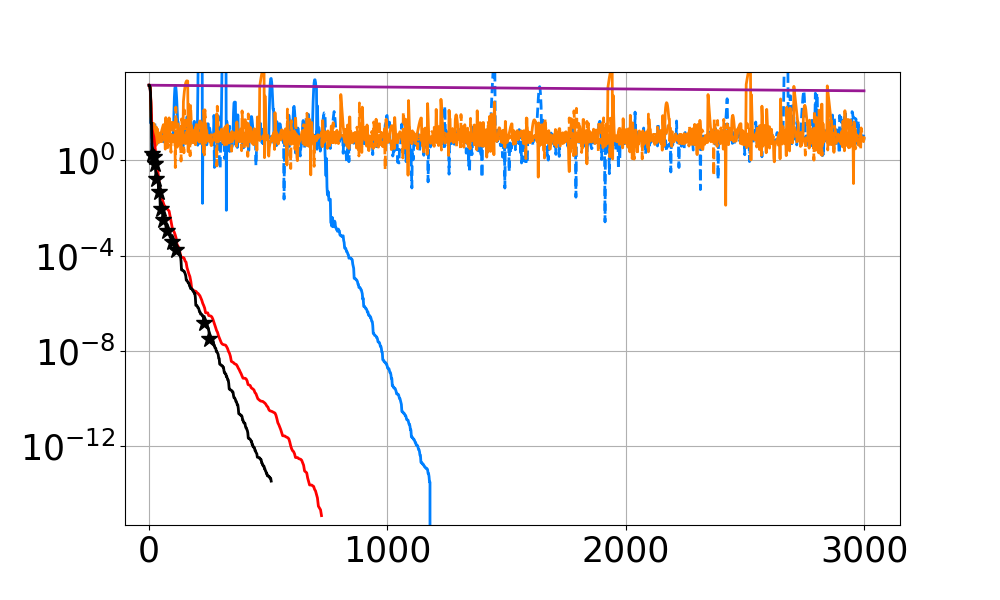}};
	\node[right] at (6.55,-1.9) {\includegraphics[width=3cm,trim=85 50 60 45,clip]{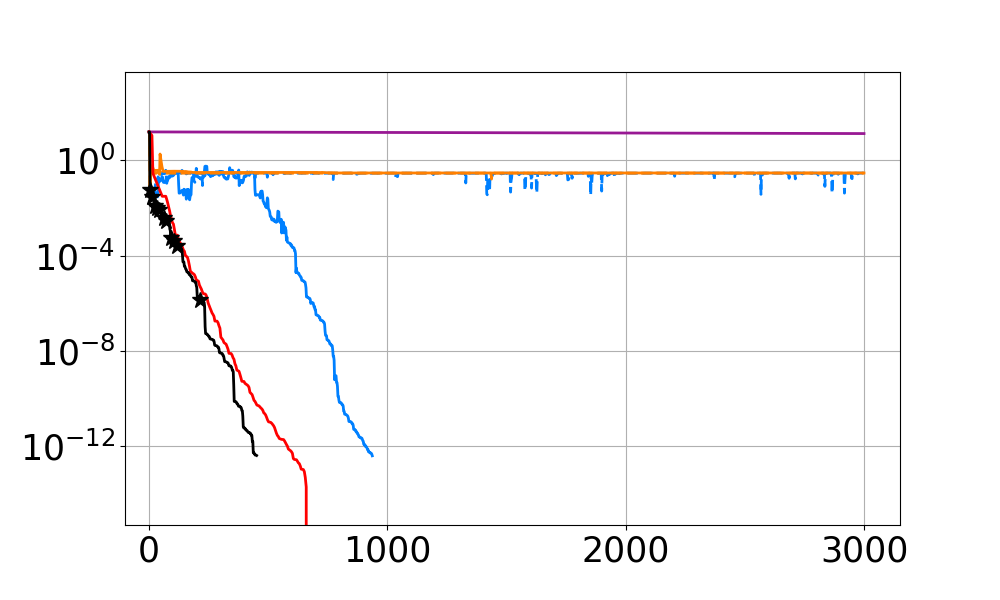}};
	\node[right] at (9.65,-1.9) {\includegraphics[width=3cm,trim=85 50 60 45,clip]{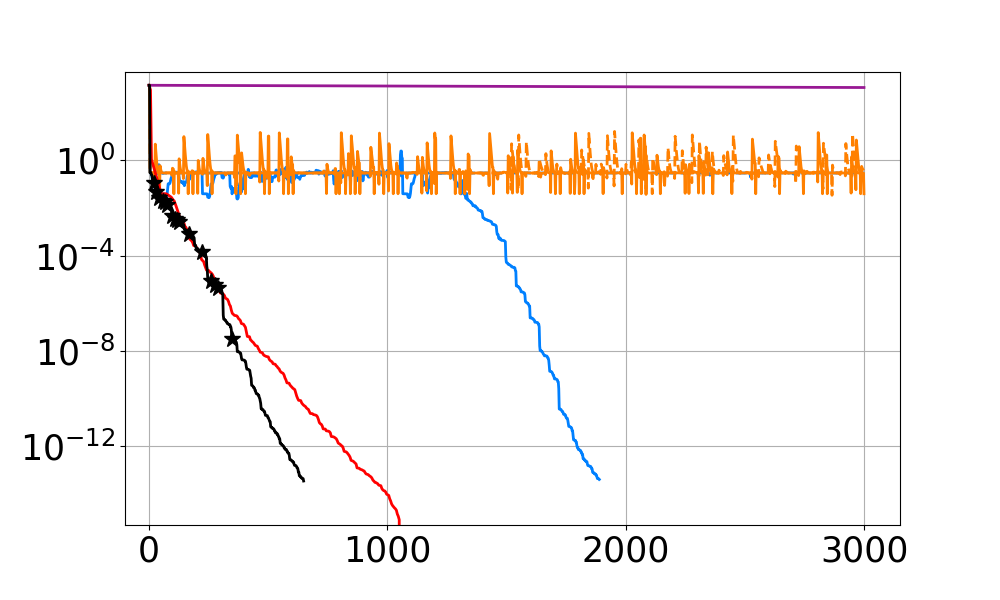}};
	\node[right] at (0.0,-3.8) {\includegraphics[width=3.35cm,trim=15 50 60 45,clip]{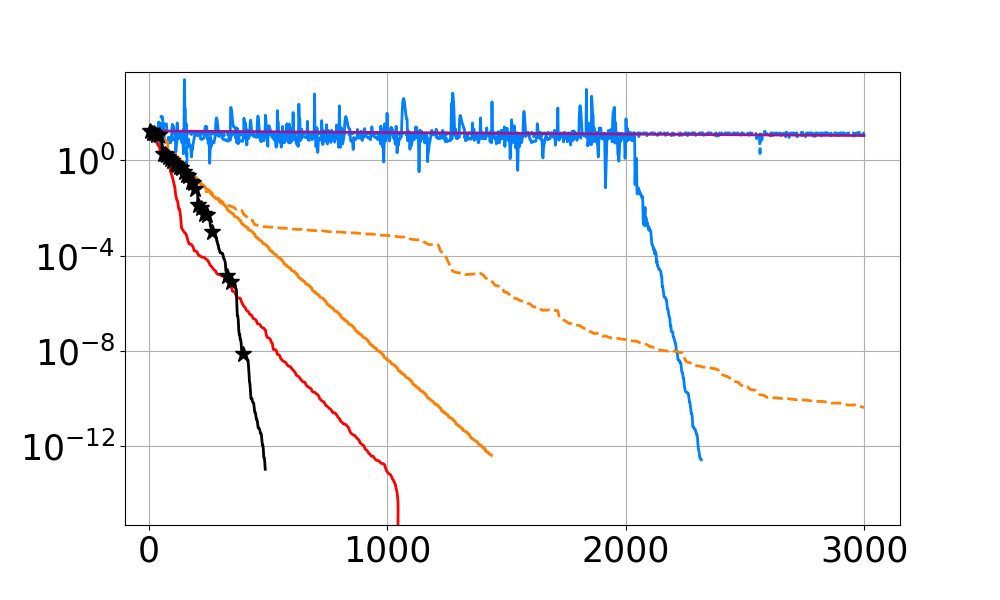}};
	\node[right] at (3.45,-3.8) {\includegraphics[width=3cm,trim=85 50 60 45,clip]{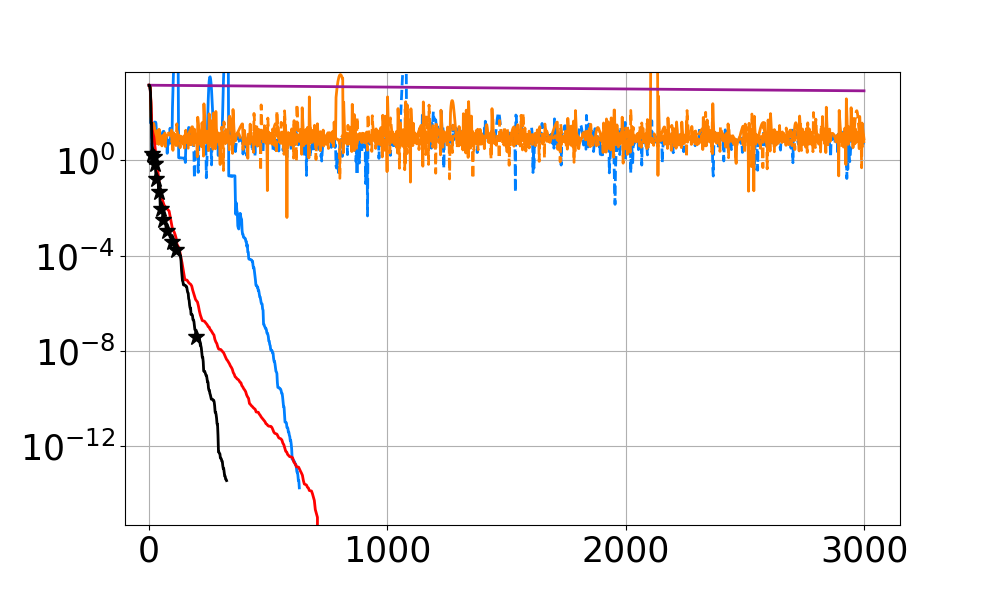}};
	\node[right] at (6.55,-3.8) {\includegraphics[width=3cm,trim=85 50 60 45,clip]{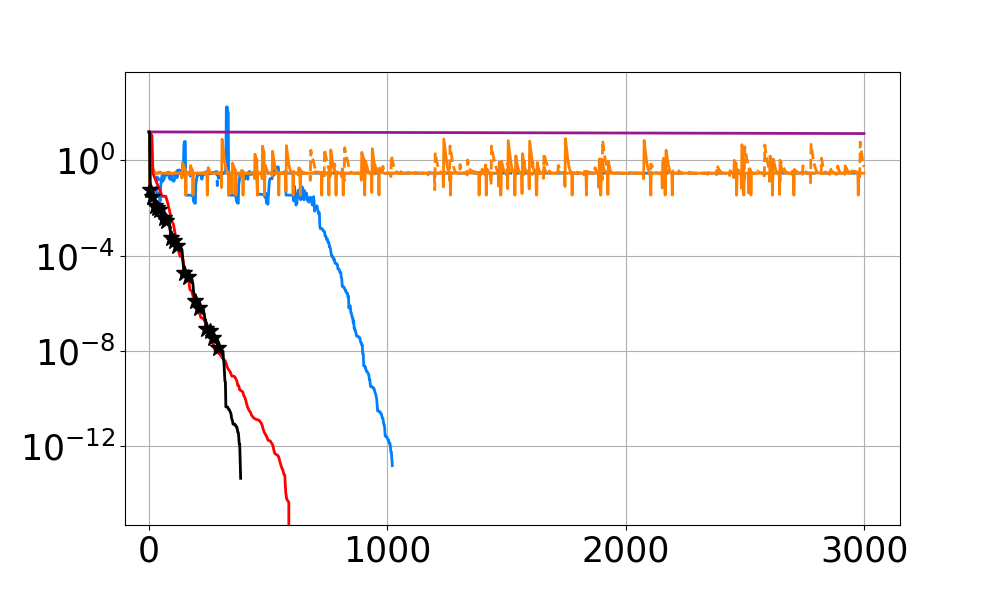}};
	\node[right] at (9.65,-3.8) {\includegraphics[width=3cm,trim=85 50 60 45,clip]{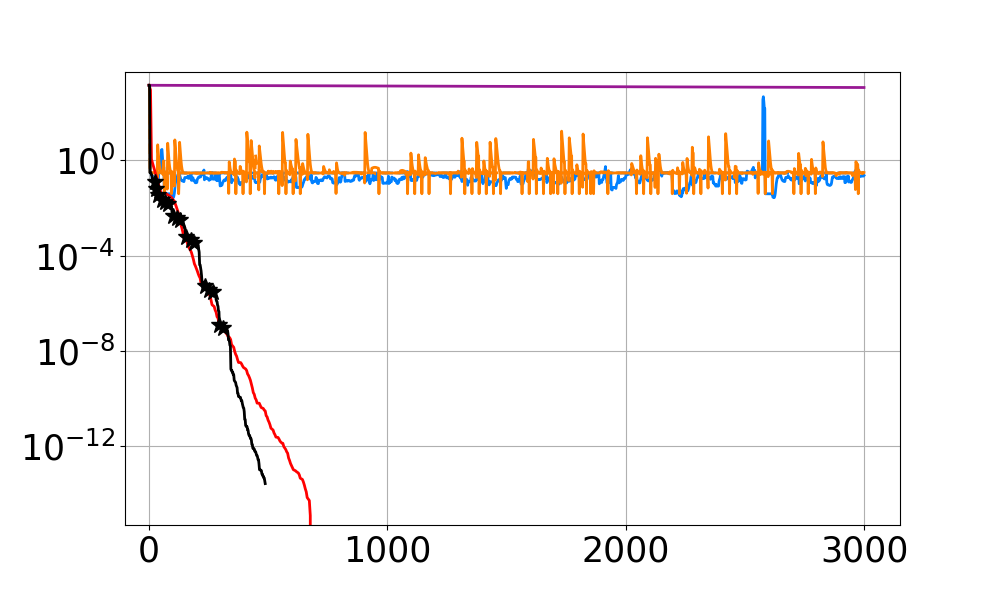}};
	\node[right] at (0.0,-5.7) {\includegraphics[width=3.35cm,trim=15 50 60 45,clip]{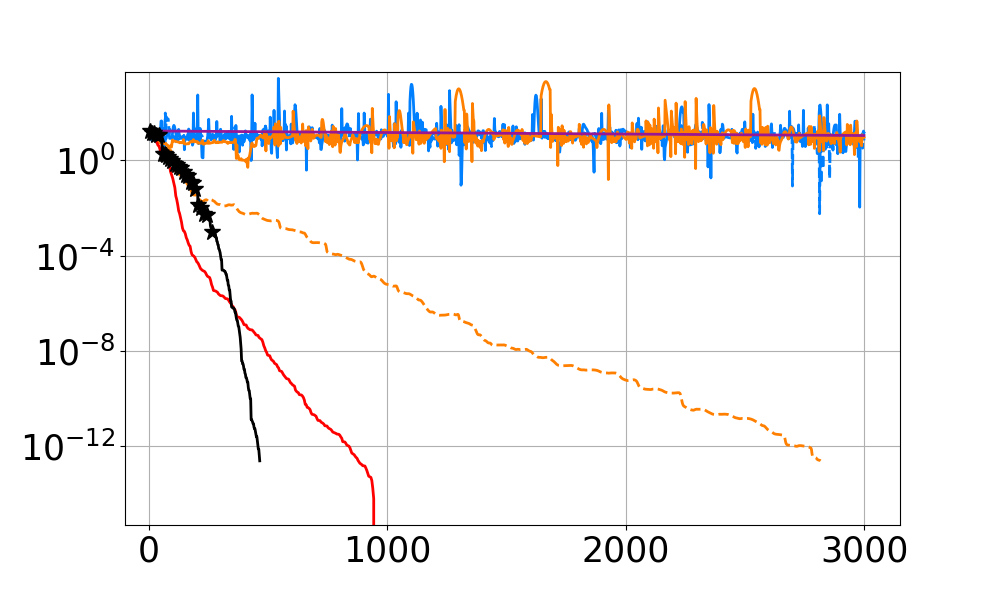}};
	\node[right] at (3.45,-5.7) {\includegraphics[width=3cm,trim=85 50 60 45,clip]{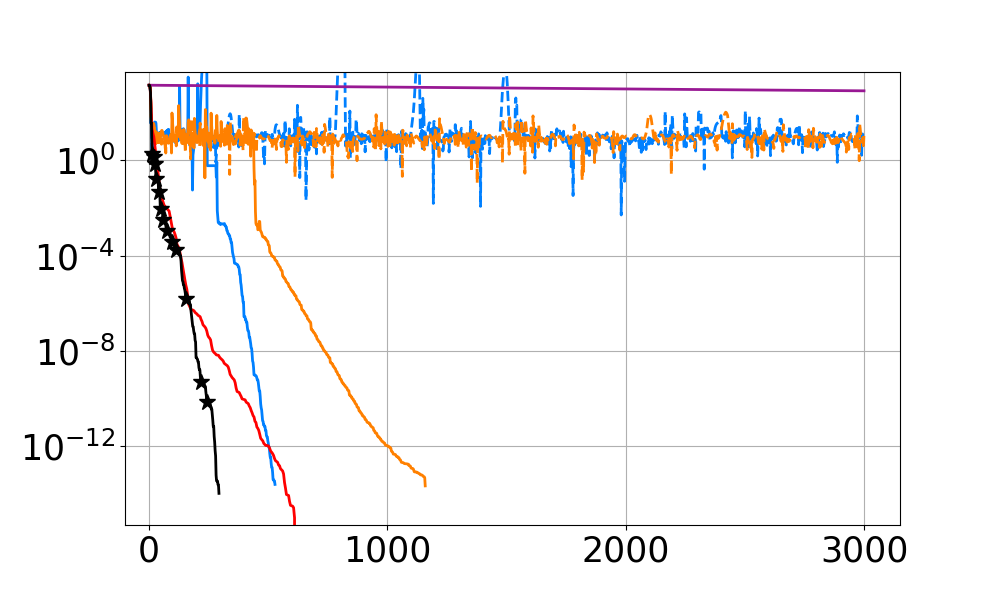}};
	\node[right] at (6.55,-5.7) {\includegraphics[width=3cm,trim=85 50 60 45,clip]{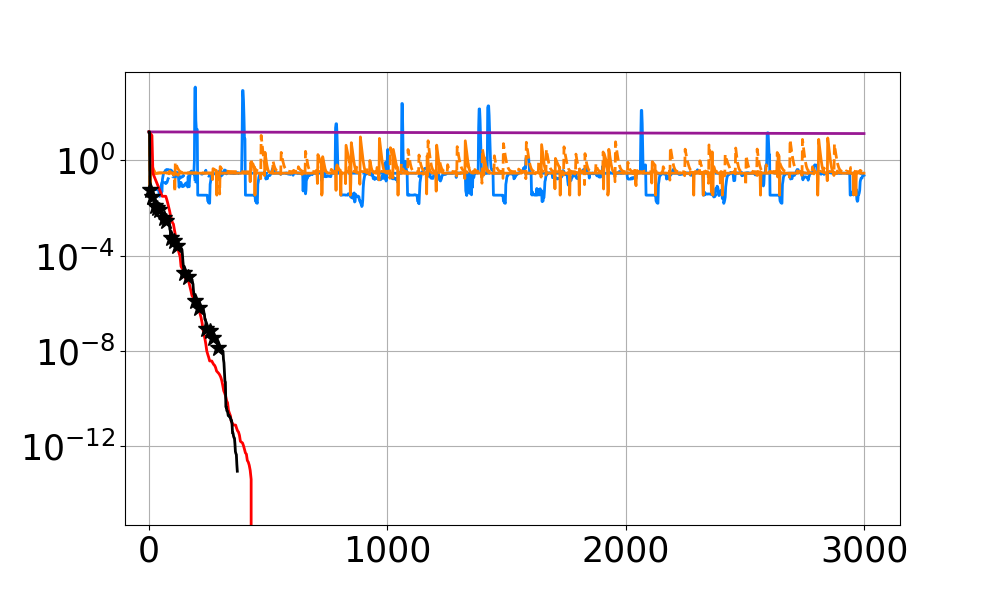}};
	\node[right] at (9.65,-5.7) {\includegraphics[width=3cm,trim=85 50 60 45,clip]{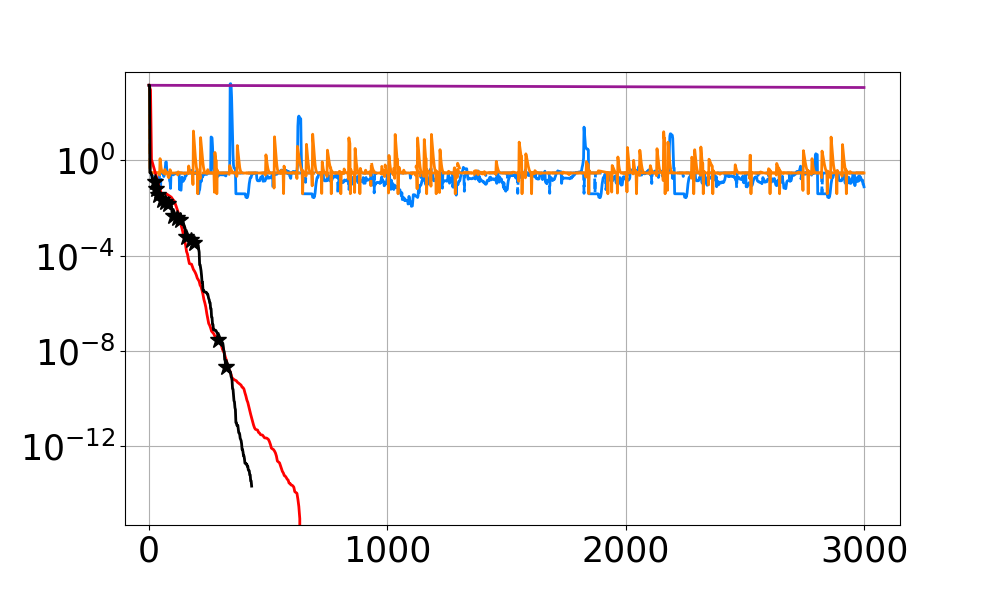}};
	\node[right] at (0.0,-7.7) {\includegraphics[width=3.35cm,trim=15 20 60 45,clip]{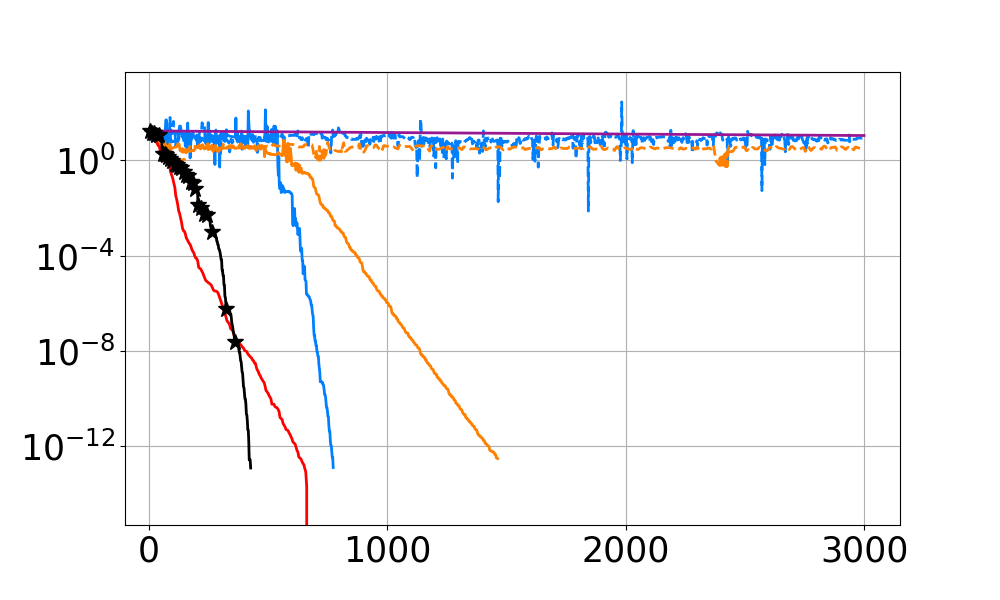}};
	\node[right] at (3.45,-7.7) {\includegraphics[width=3cm,trim=85 20 60 45,clip]{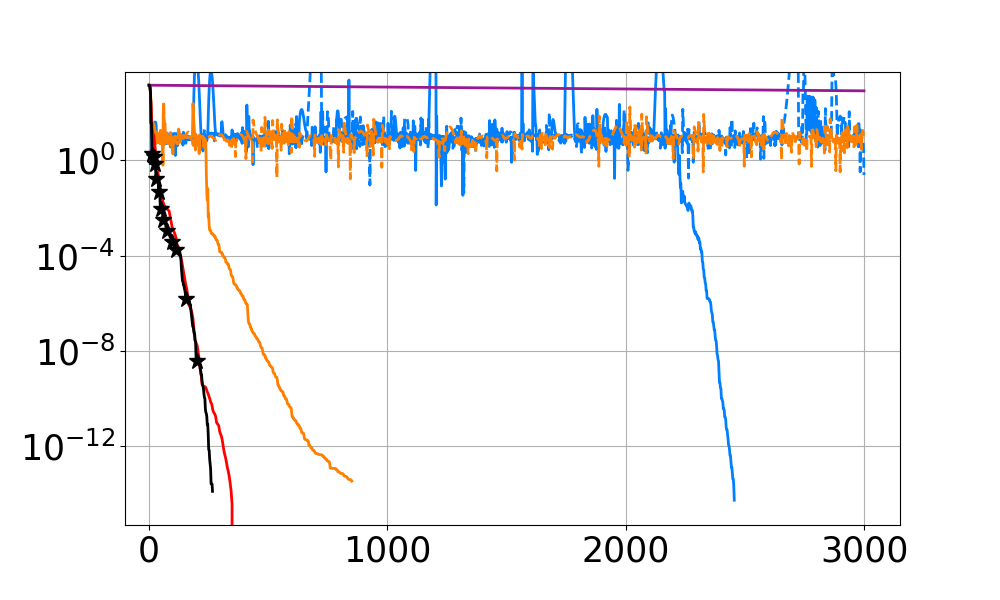}};
	\node[right] at (6.55,-7.7) {\includegraphics[width=3cm,trim=85 20 60 45,clip]{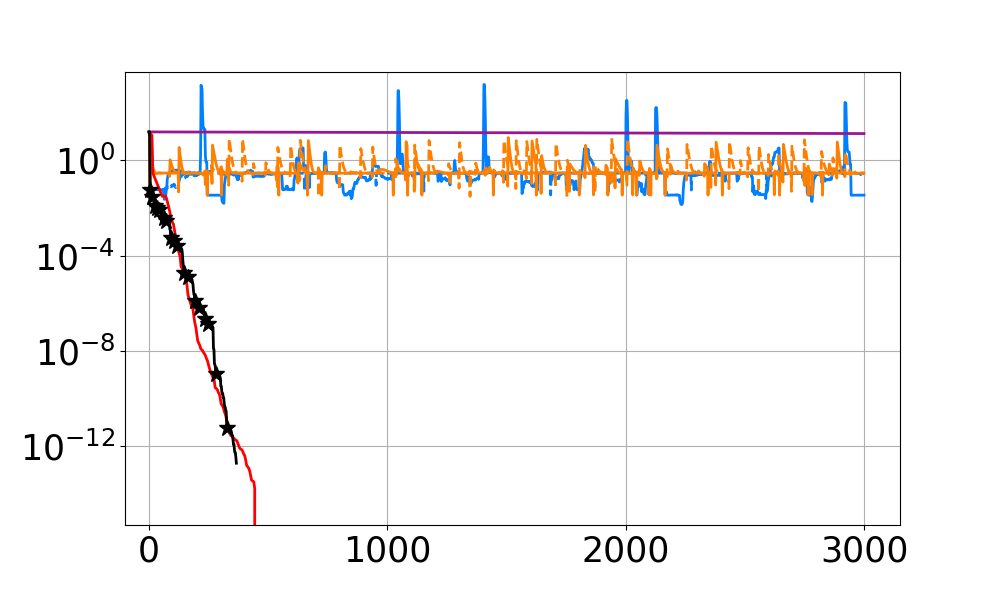}};
	\node[right] at (9.65,-7.7) {\includegraphics[width=3cm,trim=85 20 60 45,clip]{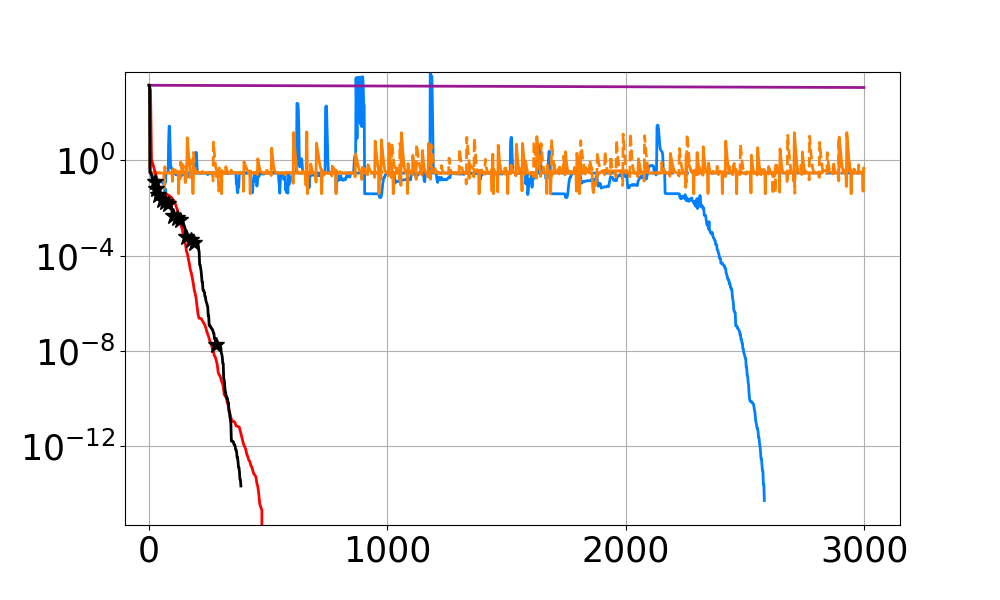}};
	\node[right] at (0.3,-9.0) {{\footnotesize(a)~ST \& \texttt{CIFAR10}}};
	\node[right] at (3.75,-9.0) {{\footnotesize(b)~ST \& \texttt{STL10}}};
	\node[right] at (6.85,-9.0) {{\footnotesize(c)~NLS \& \texttt{CIFAR10}}};
	\node[right] at (9.95,-9.0) {{\footnotesize(d)~NLS \& \texttt{STL10}}};
	\node at (12.97,0) {\rotatebox{-90}{{\tiny $m=5$}}};
	\node at (12.97,-1.9) {\rotatebox{-90}{{\tiny $m=10$}}};
	\node at (12.97,-3.8) {\rotatebox{-90}{{\tiny $m=15$}}};
	\node at (12.97,-5.7) {\rotatebox{-90}{{\tiny $m=20$}}};
	\node at (12.97,-7.7) {\rotatebox{-90}{{\tiny $m=30$}}};
\end{tikzpicture}
\vspace{-1ex}
	\caption{Relative error $ (f(x^k) - f^{*}) / \max \{f^{*}, 1\} $ vs.\ {Oracle calls} for the student's $t$ (ST) and nonlinear least-squares (NLS) problem and the datasets \texttt{CIFAR10} and \texttt{STL10}. The plots in each column are generated using the identical initial point $x^0 \sim  \mathcal{N}^{d}(0,1) $. In each row, the different $\AAn$ methods and \texttt{L-BFGS} are executed using the same memory parameter $ m \in \{5,10,15,20,30\}$. The $x$-axes of each plot have the same scaling $0$\,--\,$3,000$ (as shown in the bottom row). \label{fig:f}}
\end{figure}
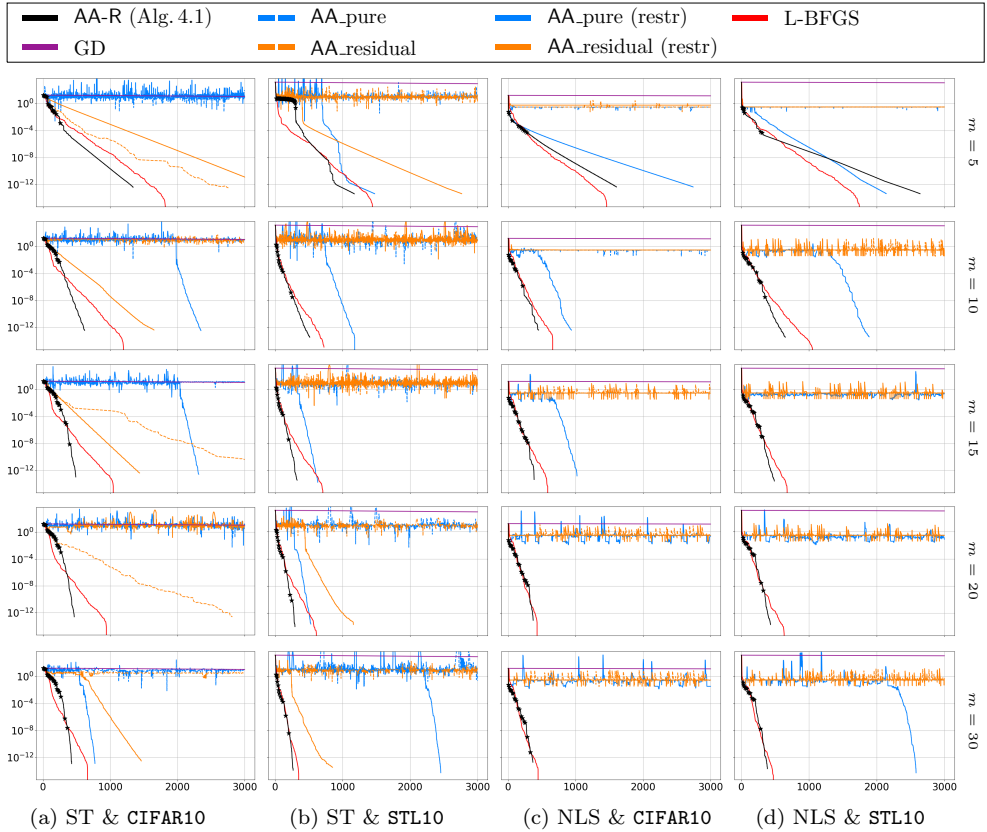

\begin{figure}[t]
	\setlength{\abovecaptionskip}{-3pt plus 3pt minus 0pt}
	\setlength{\belowcaptionskip}{-10pt plus 3pt minus 0pt}
	\centering
	\includegraphics[width=13.0cm]{plot/legend_03.pdf}
		\hspace*{-1.2ex}
	\begin{tikzpicture}[scale=1]
	\node[right] at (0.0,0) {\includegraphics[width=3.35cm,trim=15 50 60 45,clip]{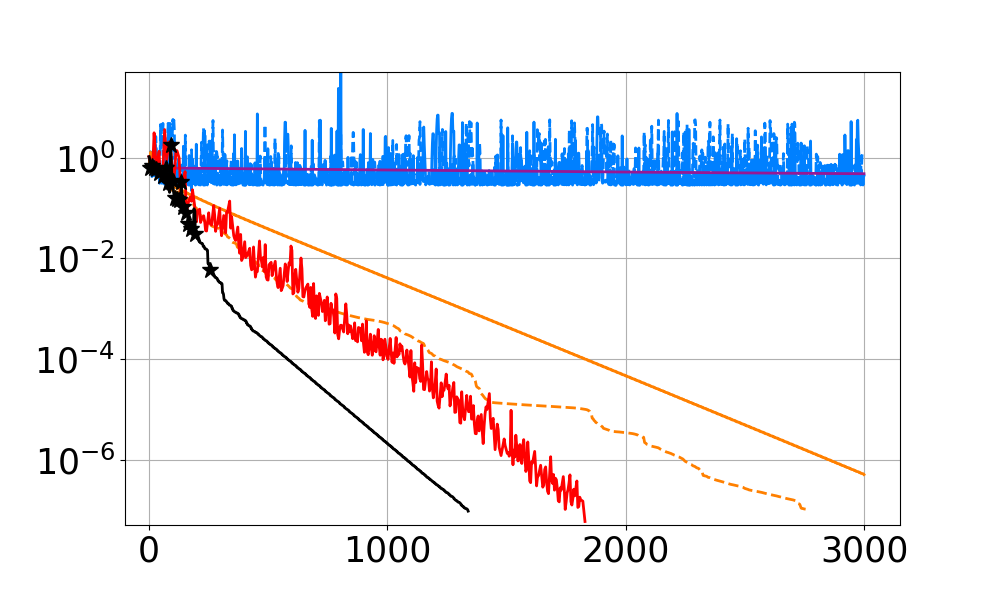}};
	\node[right] at (3.45,0) {\includegraphics[width=3cm,trim=85 50 60 45,clip]{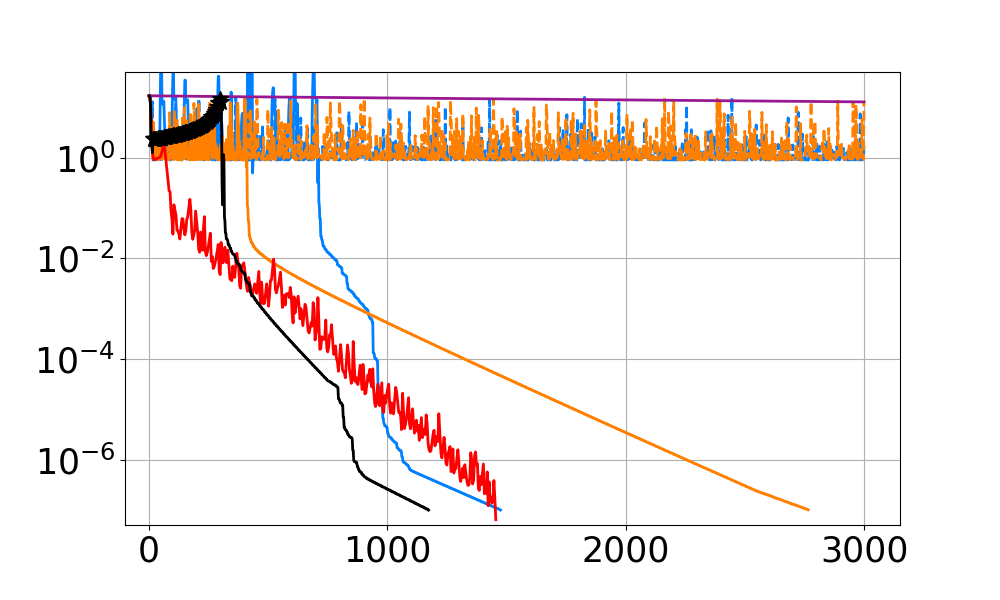}};
	\node[right] at (6.55,0) {\includegraphics[width=3cm,trim=85 50 60 45,clip]{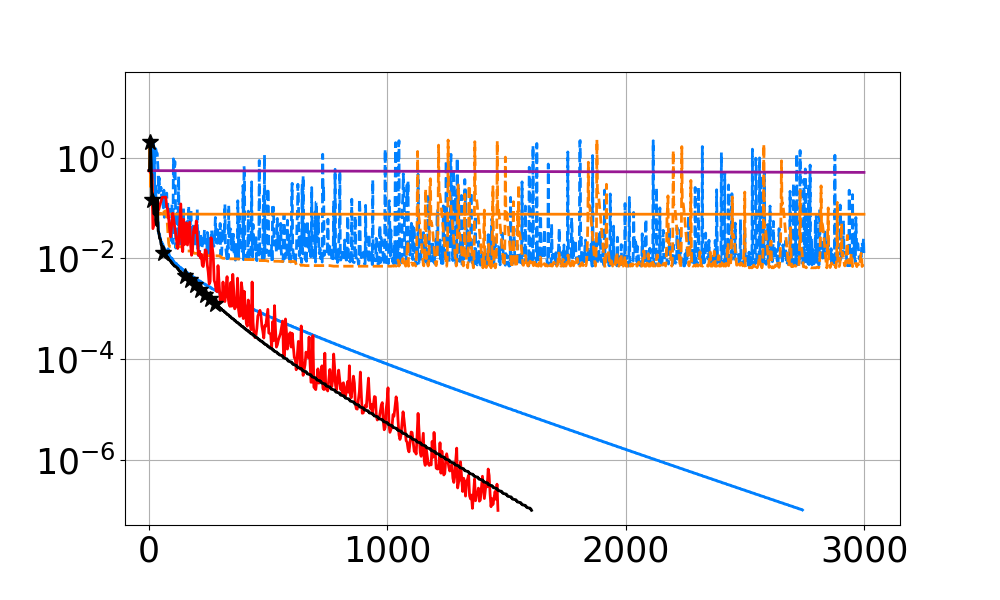}};
	\node[right] at (9.65,0) {\includegraphics[width=3cm,trim=85 50 60 45,clip]{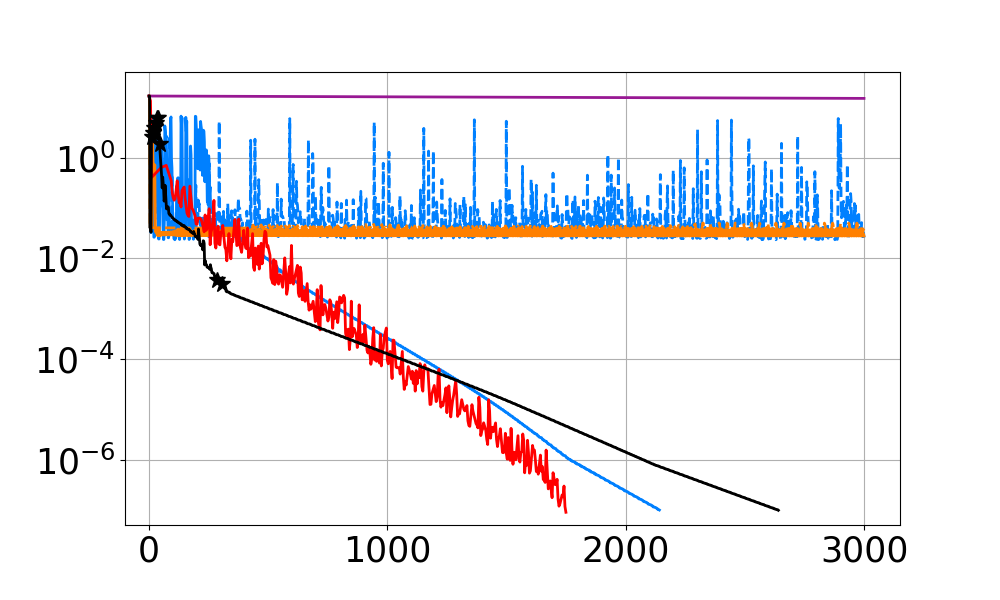}};
	\node[right] at (0.0,-1.9) {\includegraphics[width=3.35cm,trim=15 50 60 45,clip]{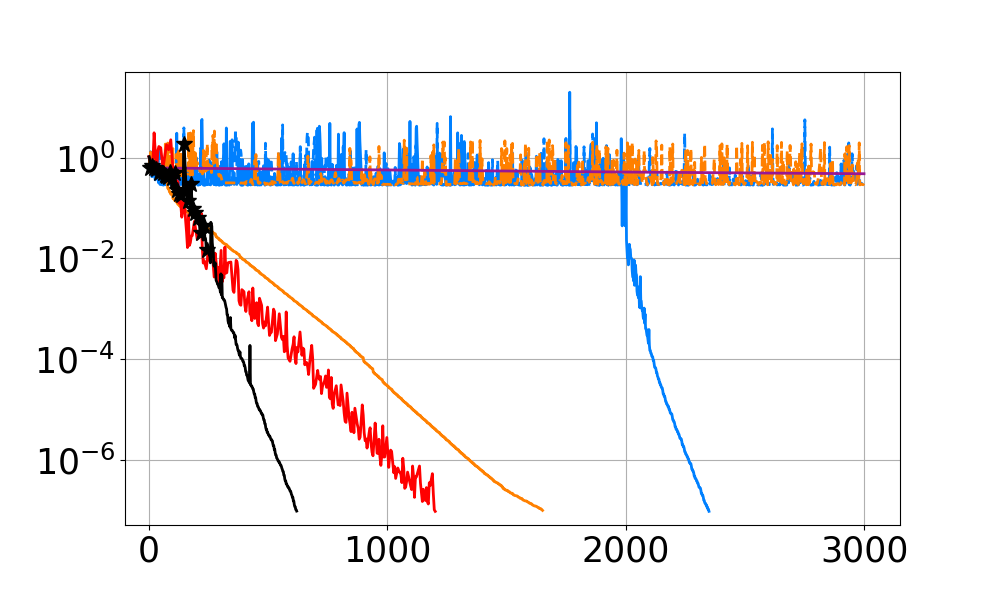}};
	\node[right] at (3.45,-1.9) {\includegraphics[width=3cm,trim=85 50 60 45,clip]{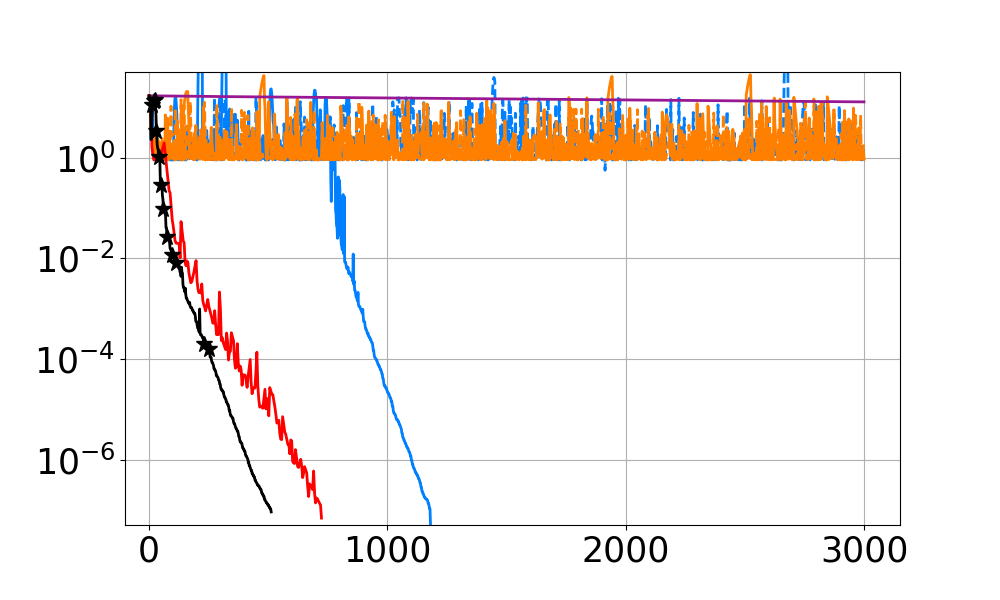}};
	\node[right] at (6.55,-1.9) {\includegraphics[width=3cm,trim=85 50 60 45,clip]{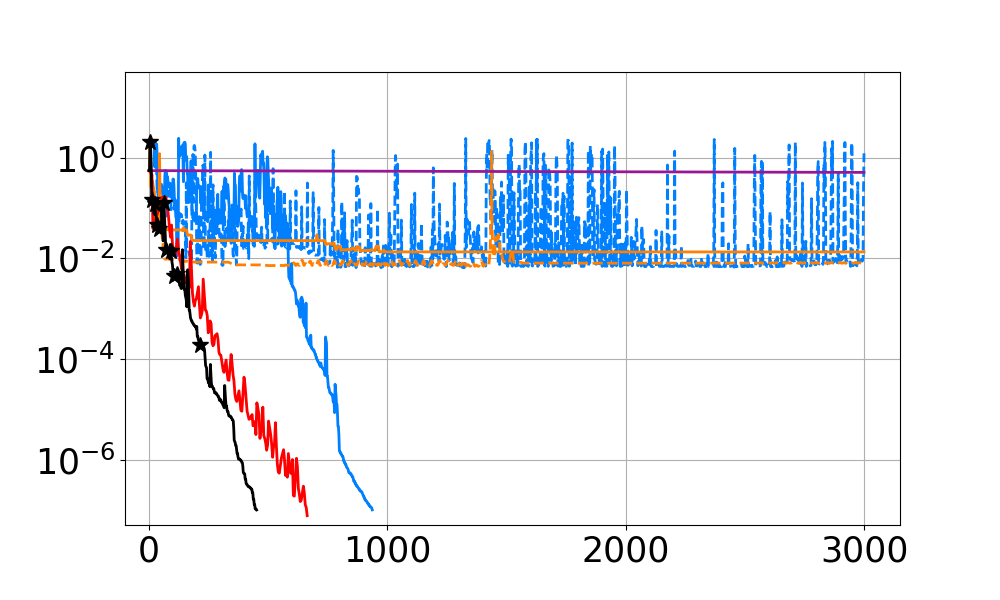}};
	\node[right] at (9.65,-1.9) {\includegraphics[width=3cm,trim=85 50 60 45,clip]{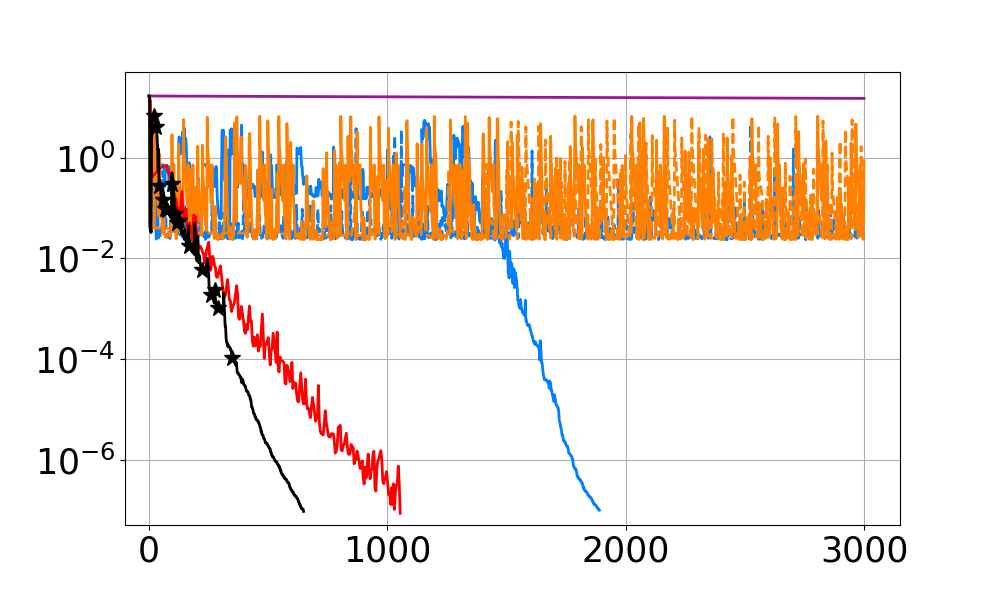}};
	\node[right] at (0.0,-3.8) {\includegraphics[width=3.35cm,trim=15 50 60 45,clip]{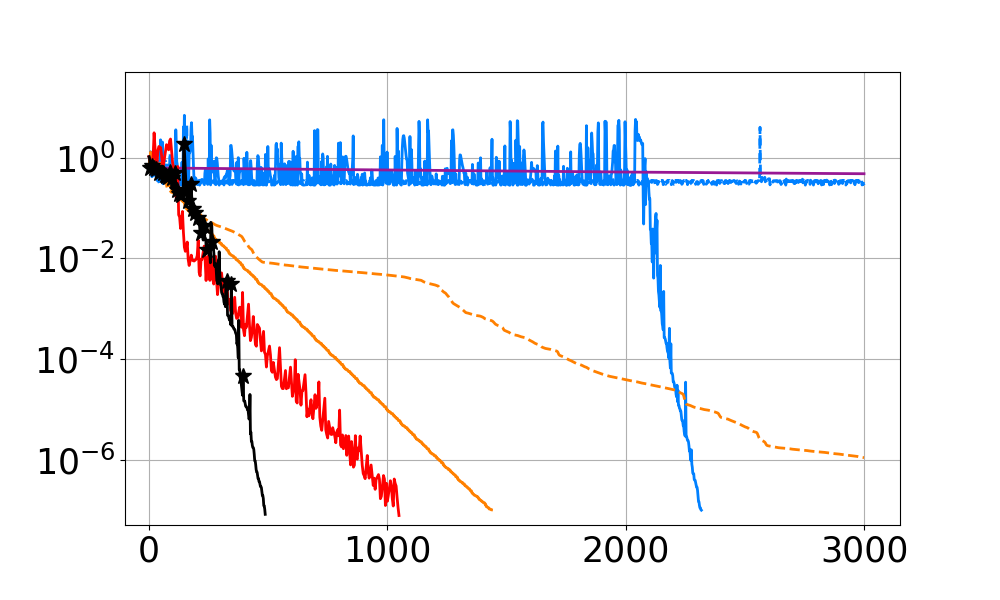}};
	\node[right] at (3.45,-3.8) {\includegraphics[width=3cm,trim=85 50 60 45,clip]{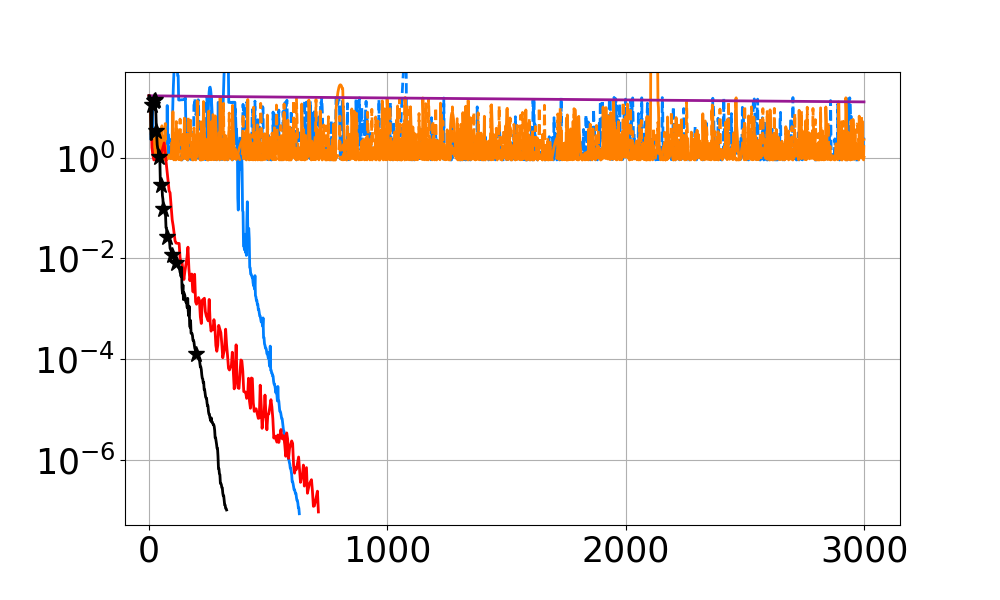}};
	\node[right] at (6.55,-3.8) {\includegraphics[width=3cm,trim=85 50 60 45,clip]{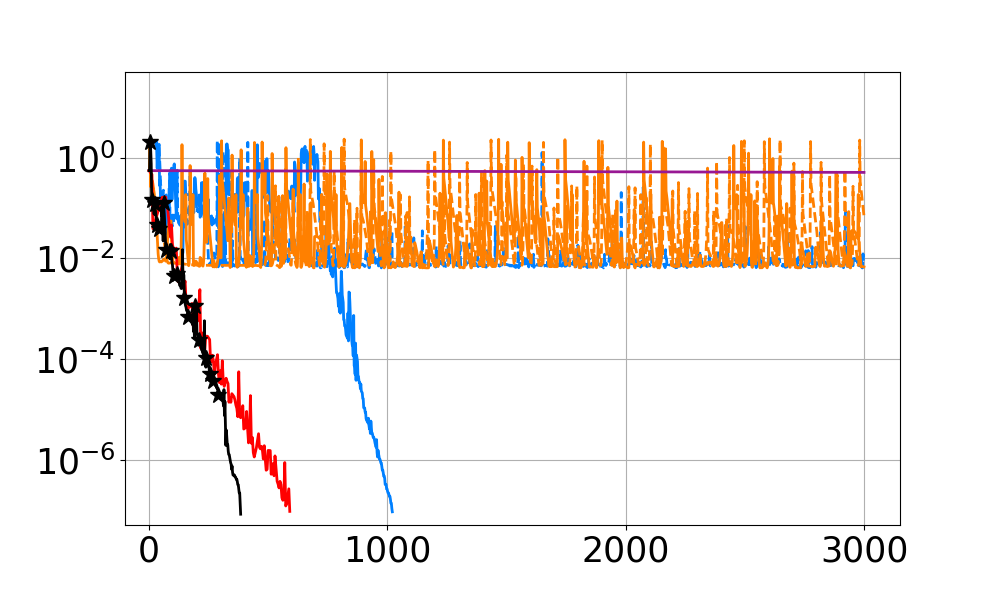}};
	\node[right] at (9.65,-3.8) {\includegraphics[width=3cm,trim=85 50 60 45,clip]{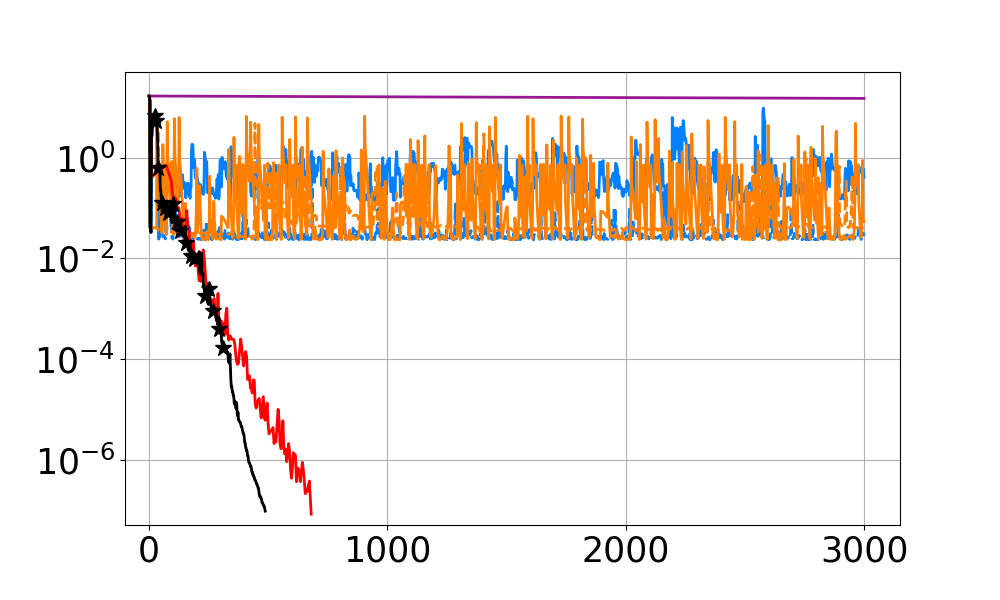}};
	\node[right] at (0.0,-5.7) {\includegraphics[width=3.35cm,trim=15 50 60 45,clip]{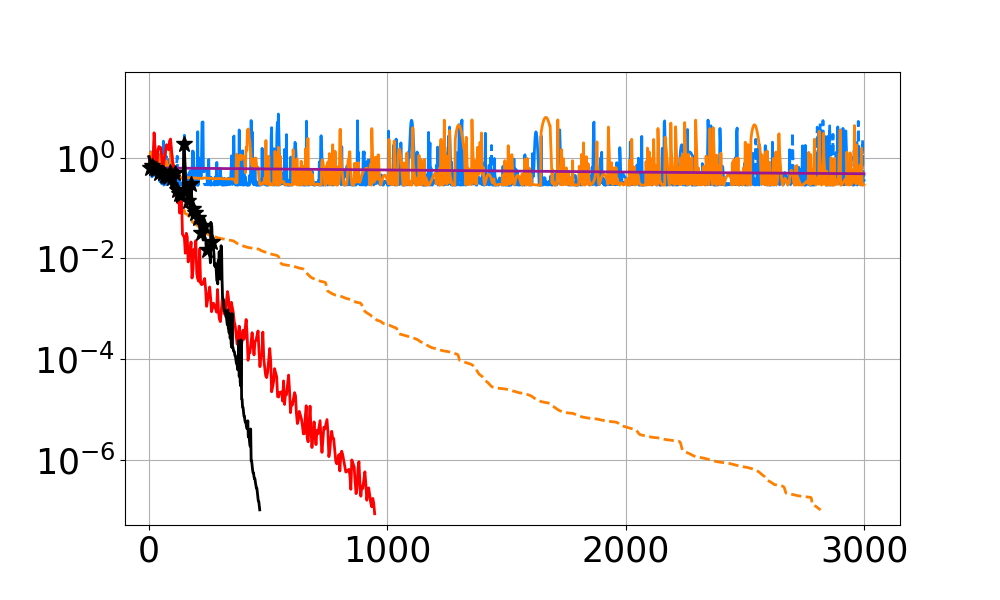}};
	\node[right] at (3.45,-5.7) {\includegraphics[width=3cm,trim=85 50 60 45,clip]{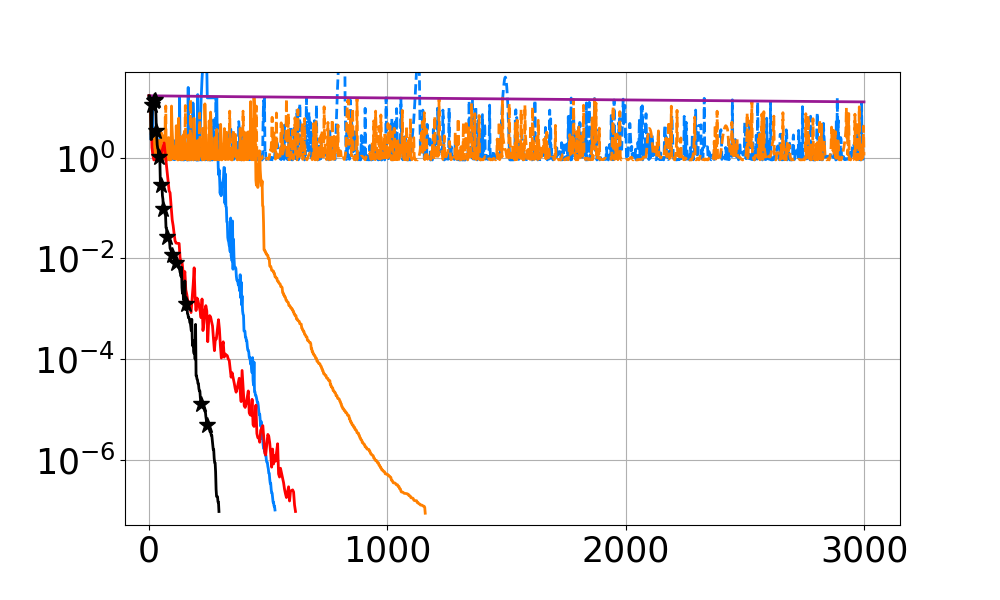}};
	\node[right] at (6.55,-5.7) {\includegraphics[width=3cm,trim=85 50 60 45,clip]{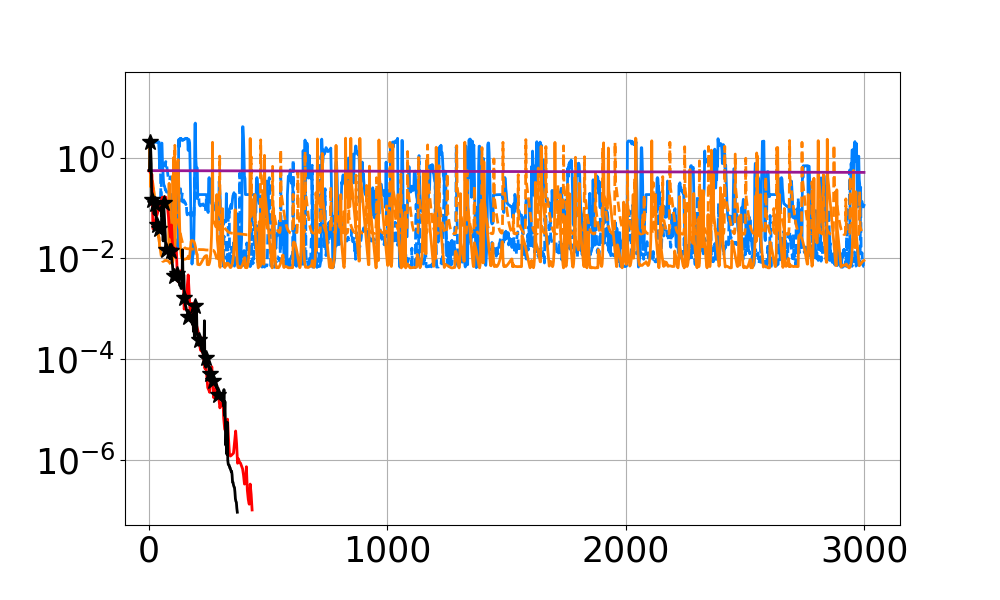}};
	\node[right] at (9.65,-5.7) {\includegraphics[width=3cm,trim=85 50 60 45,clip]{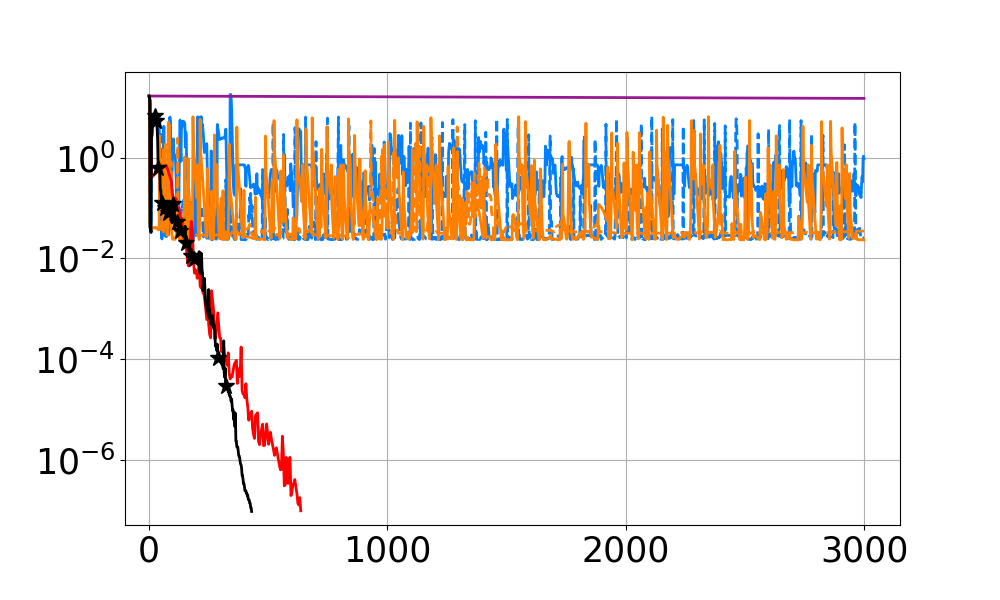}};
	\node[right] at (0.0,-7.7) {\includegraphics[width=3.35cm,trim=15 20 60 45,clip]{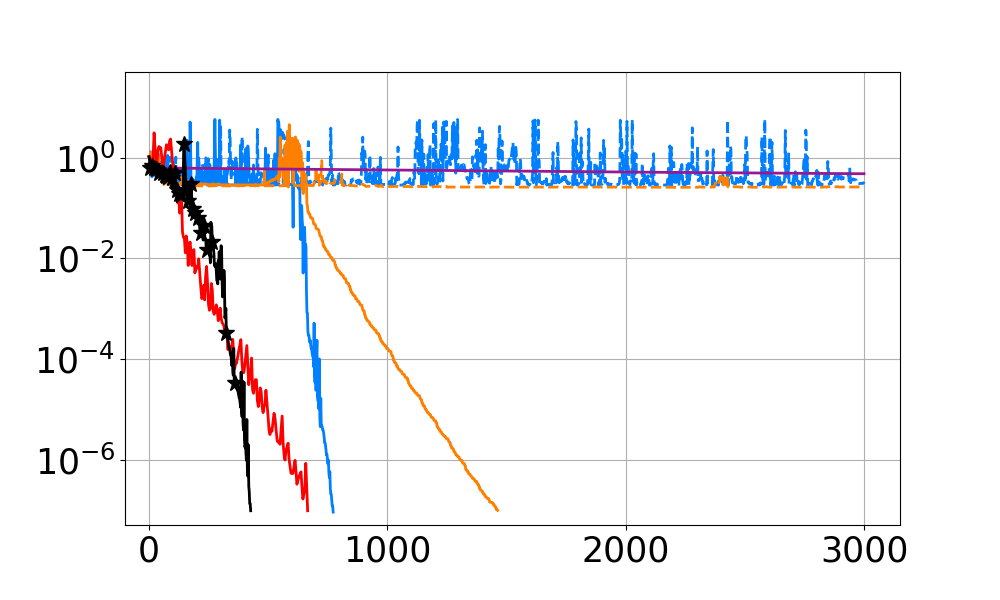}};
	\node[right] at (3.45,-7.7) {\includegraphics[width=3cm,trim=85 20 60 45,clip]{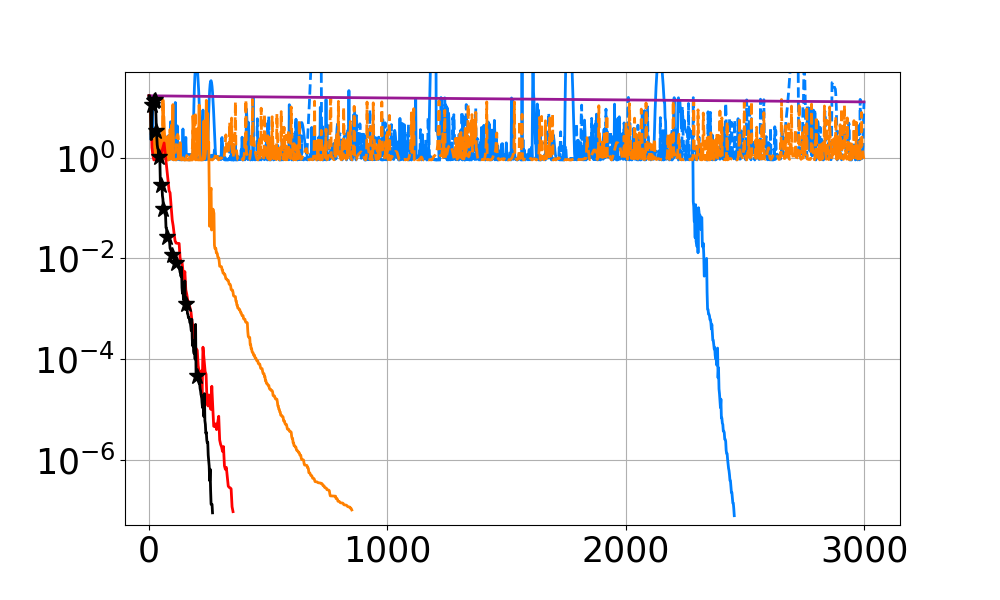}};
	\node[right] at (6.55,-7.7) {\includegraphics[width=3cm,trim=85 20 60 45,clip]{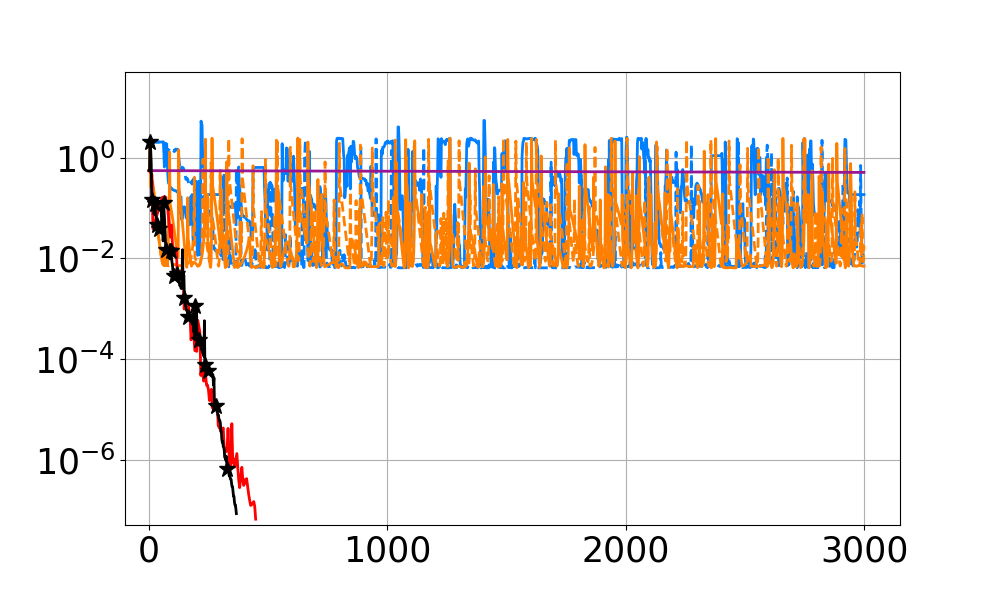}};
	\node[right] at (9.65,-7.7) {\includegraphics[width=3cm,trim=85 20 60 45,clip]{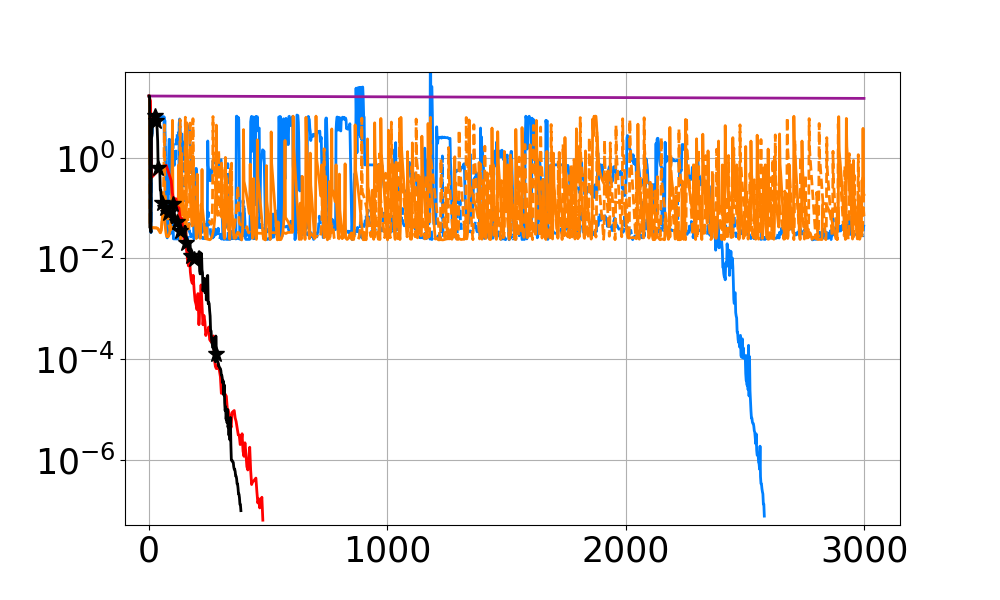}};
	\node[right] at (0.3,-9.0) {{\footnotesize(a)~ST \& \texttt{CIFAR10}}};
	\node[right] at (3.75,-9.0) {{\footnotesize(b)~ST \& \texttt{STL10}}};
	\node[right] at (6.85,-9.0) {{\footnotesize(c)~NLS \& \texttt{CIFAR10}}};
	\node[right] at (9.95,-9.0) {{\footnotesize(d)~NLS \& \texttt{STL10}}};
	\node at (12.97,0) {\rotatebox{-90}{{\tiny $m=5$}}};
	\node at (12.97,-1.9) {\rotatebox{-90}{{\tiny $m=10$}}};
	\node at (12.97,-3.8) {\rotatebox{-90}{{\tiny $m=15$}}};
	\node at (12.97,-5.7) {\rotatebox{-90}{{\tiny $m=20$}}};
	\node at (12.97,-7.7) {\rotatebox{-90}{{\tiny $m=30$}}};
\end{tikzpicture}
\vspace{-1ex}
	\caption{$ \|\nabla f(x^k)\| $ vs.\ {Oracle calls} for the student's $t$ (ST) and nonlinear least-squares (NLS) problem and the datasets \texttt{CIFAR10} and \texttt{STL10}. The plots in each column are generated using the identical initial point $x^0 \sim  \mathcal{N}^{d}(0,1) $. In each row, the different $\AAn$ methods and \texttt{L-BFGS} are executed using the same memory parameter $ m \in \{5,10,15,20,30\}$. The $x$-axes of each plot have the same scaling $0$\,--\,$3,000$ (as shown in the bottom row). \label{fig:gnorm}}
\end{figure}

\subsection{Nonconvex Classification}
We consider two nonconvex classification problems, namely a nonlinear least-squares problem and a student's-$t$ problem. A detailed introduction of the tested problems is deferred to the subsequent paragraphs. We will compare \cref{algo4} with four different methods:
\begin{itemize}
	\item[(1)] The gradient descent method (\texttt{GD}) with fixed step size $\frac{1}{L}$. This is the original Picard iteration \eqref{picard-iteration} and serves as a baseline.
	\item[(2)] Pure $\AAn$ with and without restarting~\cite{walker2011anderson}. Pure $\AAn$ does not use any globalization strategy, i.e., in each step, we perform an $\AAn$ iteration.
	\item[(3)] $\AAn$ with residual-based globalization. Our implementation is based on \texttt{A2DR} \cite[Algorithm 3]{fu2020anderson} and we consider two variants with and without restarting. \texttt{A2DR} uses a residual-based acceptance mechanism and we apply the method with the default parameters as suggested in \cite{fu2020anderson}.
	\item[(4)] \texttt{L-BFGS}. We implement \texttt{L-BFGS} with Wolfe conditions as in \cite[Algorithm 7.5]{jorge2006numerical}. The line-search parameter is set to $10^{-4}$ and the parameter in Wolfe's condition is set to $0.9$ as suggested in \cite{jorge2006numerical}. The maximum number of line-search iterations is set to $1,000$. 
\end{itemize}

We note that the line search procedure in \texttt{L-BFGS} can contain many function and gradient evaluations per iteration. Therefore, it is improper to give comparisons solely based on the number of iterations. In our plots, the $x$-axis is typically set as the number of oracle calls, which appears to be a more appropriate and fair measure when comparing $\AAn$ algorithms and \texttt{L-BFGS}. Specifically, the computation of each function value is counted as one oracle call and every gradient evaluation contributes as an additional oracle call. The $y$-axis is set as $(f(x^k) - f^{\star})/\max \{f^{\star}, 1\}$ or $\| \nabla f(x^k) \|$, respectively, where $ f^{\star} $ denotes the best objective function value among all algorithms over the maximum oracle calls. In the figures, we will add a special mark ``$\star$'' once the current $\AAn$ step is rejected and a gradient step is performed in \cref{algo4}.

We continue with the description of the utilized training datasets and several universal implementational details. We use the \texttt{CIFAR10} dataset \cite{krizhevsky2009learning}, which contains $60,000$ images with $32 \times 32$ colored pixels and the \texttt{STL10} dataset \cite{coates2011analysis}, which consists of $5,000$ colored images of size $96 \times 96$. Given that both datasets contain $ 10 $ classes, we split the data into even and odd classes to allow binary classification. We use $\{u_i, v_i\}_{i=1}^n$ to denote the training samples, where $u_i \in \mathbb{R}^d$ represents the training image and $v_i \in \{0, 1\}$ is the associated label. We set $U=\{u_1,u_2,\ldots,u_n\}^\top \in \mathbb{R}^{n \times d}$. We terminate the algorithms once $ \| \nabla f(x^k) \| \leq 10^{-7} $ or the number of oracle calls exceeds $ 3,000 $. The memory parameter $m$ is chosen from $ m \in \{5,10,15,20,30\}$ for all $\AAn$-based methods and \texttt{L-BFGS}. The regularization parameter $ \lambda $ in \eqref{eq:ST} and  \eqref{eq:NLS} is set to $10^{-2}$ for \texttt{CIFAR10} and to $10^{-1}$ for \texttt{STL10}. The initial points $ x^0 \sim \mathcal{N}^{d}(0,1) $ are generated following a normal distribution. Finally, in \cref{algo4}, we {utilize the default parameters: $ \gamma = \frac{0.01}{2L} $, $ c_1 = c_3 = 1 $, $ c_2 = \frac{0.99}{2mL} $, and $ \nu = 2.1$.  Let us briefly motivate this default choice. In order to promote acceptance of $\mathsf{AA}$ steps (and to ensure potential acceleration), the descent condition \eqref{eq:descent-condition} should not be too strict. This can be achieved when $\gamma$ is small and when the $\min$-term in \eqref{eq:descent-condition} is large. Hence, we set $\gamma$ fairly small, $c_2$ close to the theoretical threshold, and $\nu$ close to $2$. Furthermore, we have found that the simple choice $c_1 = c_3 = 1$ works well enforcing sufficient progress during the first iterations. An additional ablation study for $c_1$, $c_2$, $c_3$, and $\gamma$ is discussed in \cref{sec:sub-abla}}. We use the \texttt{LSQR} method \cite{paige1982lsqr} to solve the $\AAn$ subproblem \eqref{AAsubp} and to compute $\alpha^k = -H_k^{\dagger} h^{k-\hat m} =-(H_k^\top H_k)^{-1}H_k^\top h^{k-\hat m} $. (Here, $H_k^{\dagger}$ represents the Moore-Penrose pseudo-inverse of $H_k$). The termination condition of \texttt{LSQR} is set to $ \| H_k^\top (H_k \alpha + h^{k-\hat m}) \| < 10^{-16} $.

Next, we present the classification models used in our numerical comparison:
\begin{itemize}
\item \textbf{Student's-t Loss with $\ell_2$-Regularization (ST).} We consider the following classification problem with student's-t loss, \cite{aravkin2011robust,aravkin2012robust},
\end{itemize}
\begin{align}
\label{eq:ST}   
f(x) = \frac{1}{n} {\sum}_{i=1}^{n}  \log\left(1 + (u_i^\top x - v_i)^2 / \mu \right) + \frac{\lambda}{2}  \|x\|^2.
\end{align}
\begin{itemize}
\item[]The Lipschitz constant of $\nabla f$ is given by  $ L = \frac{2}{\mu n} \| U \|^2 + \lambda $ and we set $\mu = 20$. 
\item \textbf{Nonlinear Least-Squares Problem with $\ell_2$-Regularization (NLS).} As a second example, we consider a nonlinear least-squares problem, \cite{xuNonconvexEmpirical2017},
\end{itemize}
\begin{align}
\label{eq:NLS}  
f(x) = \frac{1}{n} {\sum}_{i=1}^{n}  (\psi(u_i^\top x) - v_i )^2 + \frac{\lambda}{2}  \|x\|^2,
\end{align}
\begin{itemize}
\item[]where $ \psi(z) = 1/(1+e^{-z}) $ is the sigmoid function. The Lipschitz constant of $\nabla f$ is given by $ L = \frac{1}{6n} \| U \|^2 + \lambda $.
\end{itemize}

The initial points for all algorithms and $ m \in \{ 5, 10, 15, 20, 30\}$ are identical for each tested dataset and classification model. Figures~\ref{fig:f} and \ref{fig:gnorm} illustrate that $\AAr$ with function value-based globalization (\ref{algo4}) is a competitive solver. Specifically, \cref{algo4} requires the least amount of oracle calls to satisfy the stopping criterion when $ m \in \{10, 15, 20, 30\}$. However, in the low memory case $ m=5 $, \cref{algo4} and the restarting strategy seem less effective (especially for the nonlinear least-squares problem). The plots in Figures \ref{fig:f}--\ref{fig:gnorm} generally underline the potential of function value- and descent-based globalization mechanisms for $\AAn$ schemes. Rejections predominantly occur in the early stage of the iterative process to ensure global convergence and progress of \cref{algo4}. In addition, transition to a pure $\AAr$ phase with accelerated convergence is maintained --- as indicated by our theoretical results. 

As the applications tested in this section are nonconvex, we have recorded the smallest eigenvalues of the respective Hessians in the last iterations of $\AAr$ for each of the problems and datasets. We have observed that that these eigenvalues are all approximately equal to $\lambda > 0$ and hence, assumption \ref{A2} is locally satisfied.  

\subsection{Descent Properties}
In Figure~\ref{fig:descent}, we plot the measure
\[ \rho_k := {\max \{f(x_{\AAn}^{k}) - f(g(x^k)), 0\}}/{\| \nabla f(x^{k-\hat{m}}) \|^3}\] 
versus the number of iterations $k$ to further visualize and verify the descent properties derived in \cref{thm3-18}. If the $\AAn$ step achieves descent, $ f(x_{\AAn}^{k}) \leq f(g(x^k)) $, then we have $\rho_k = 0$ and we locally expect $\rho_k \approx 0$ for all $k$ sufficiently large. The special marks ``$\star$'' in Figure~\ref{fig:descent} again indicate that an $\AAn$ step did not pass the descent condition \eqref{eq:descent-condition}. Figure~\ref{fig:descent} illustrates that $\rho_k$ indeed stays zero eventually and that no $\AAn$ steps are rejected locally. This observation is slightly less pronounced on \texttt{CIFAR10} when $m = 30$.

\begin{figure}[t]
	\setlength{\abovecaptionskip}{-3pt plus 3pt minus 0pt}
	\setlength{\belowcaptionskip}{-10pt plus 3pt minus 0pt}
	\centering
	\hspace*{-1.2ex}
	\begin{tikzpicture}[scale=1]
	\node[right] at (0.0,0) {\includegraphics[width=3.35cm,trim=15 20 60 45,clip]{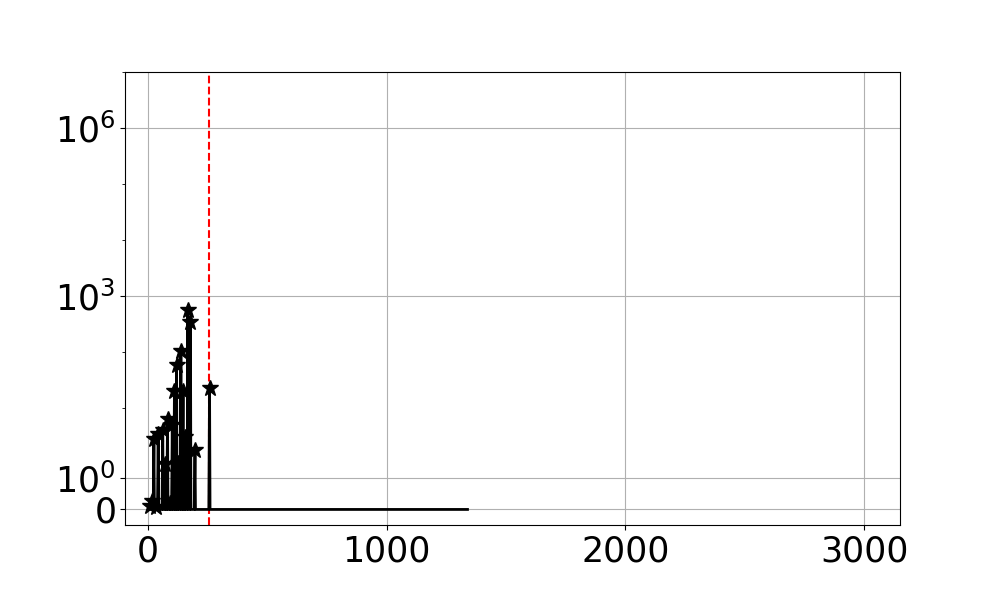}};
	\node[right] at (3.45,0) {\includegraphics[width=3cm,trim=85 20 60 45,clip]{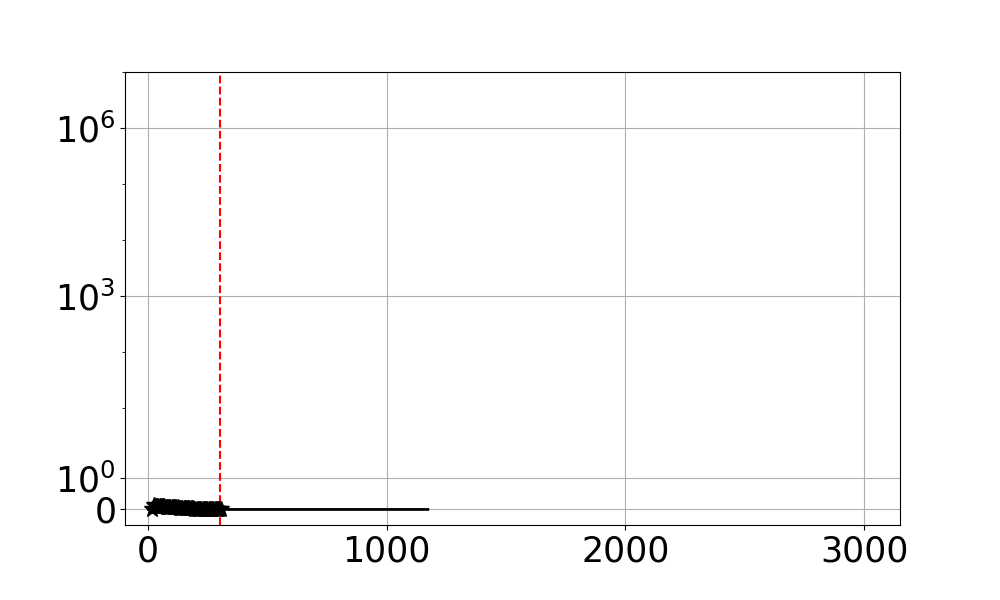}};
	\node[right] at (6.55,0) {\includegraphics[width=3cm,trim=85 20 60 45,clip]{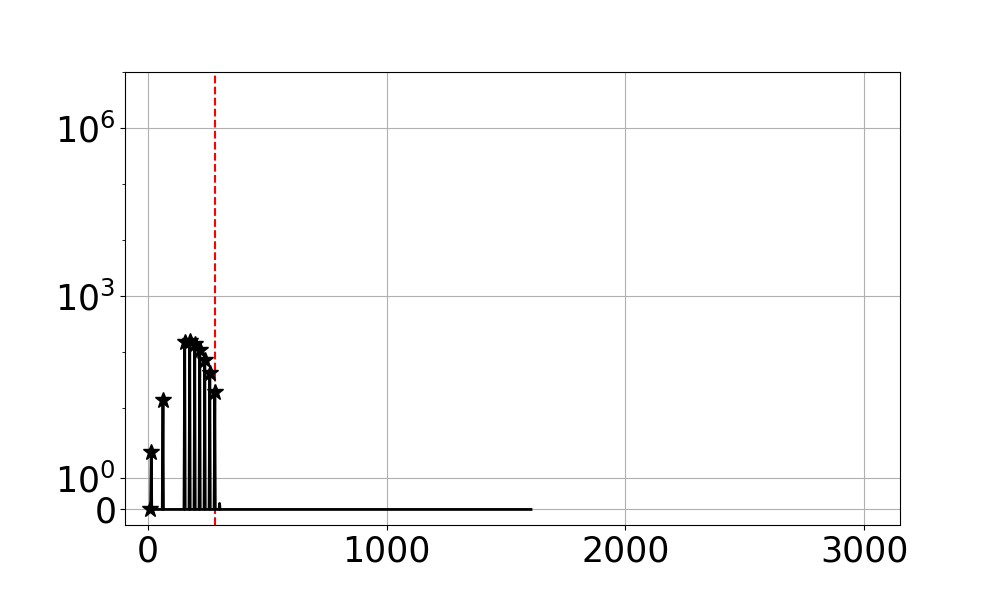}};
	\node[right] at (9.65,0) {\includegraphics[width=3cm,trim=85 20 60 45,clip]{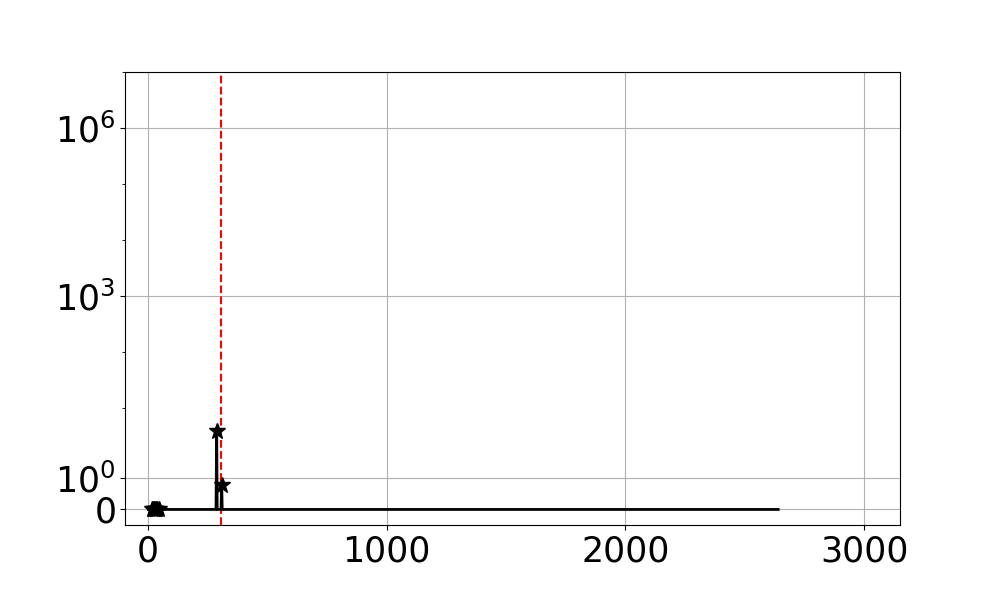}};
	\draw[black,very thin,densely dashed] (0.18,-1.05) -- (12.8,-1.05);
	\node[right] at (0.0,-2.0) {\includegraphics[width=3.35cm,trim=15 50 60 45,clip]{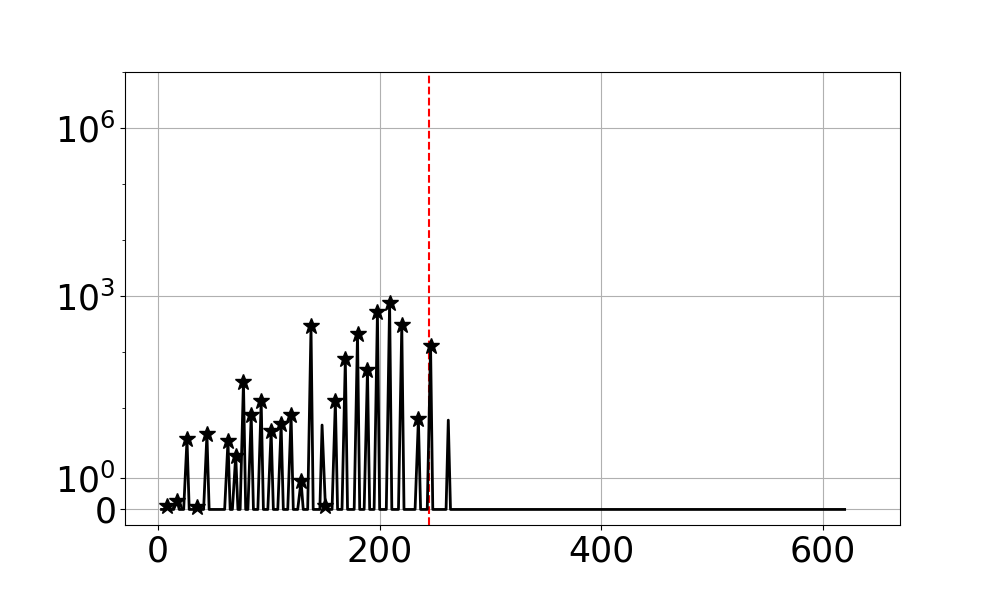}};
	\node[right] at (3.45,-2.0) {\includegraphics[width=3cm,trim=85 50 60 45,clip]{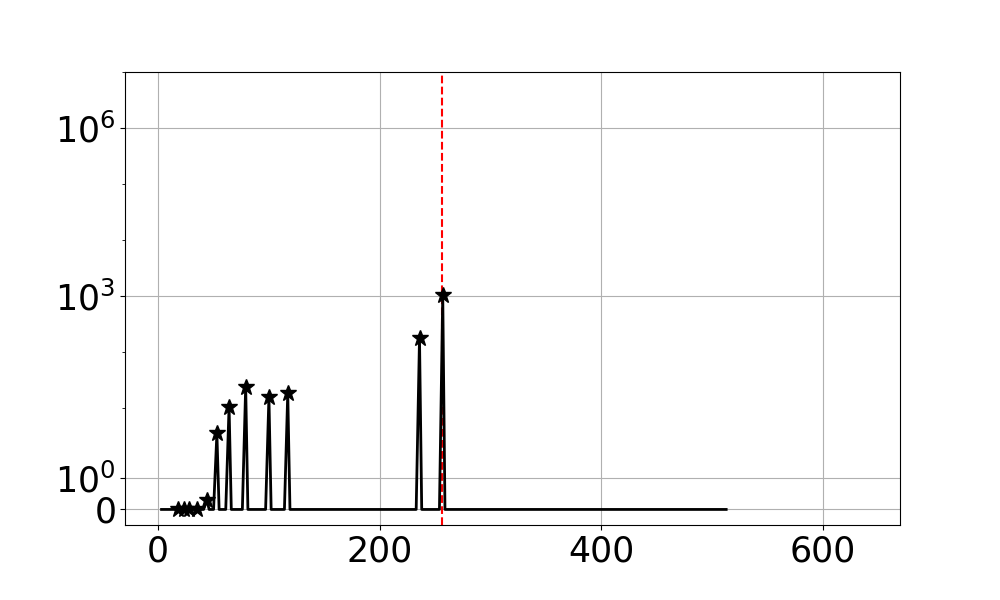}};
	\node[right] at (6.55,-2.0) {\includegraphics[width=3cm,trim=85 50 60 45,clip]{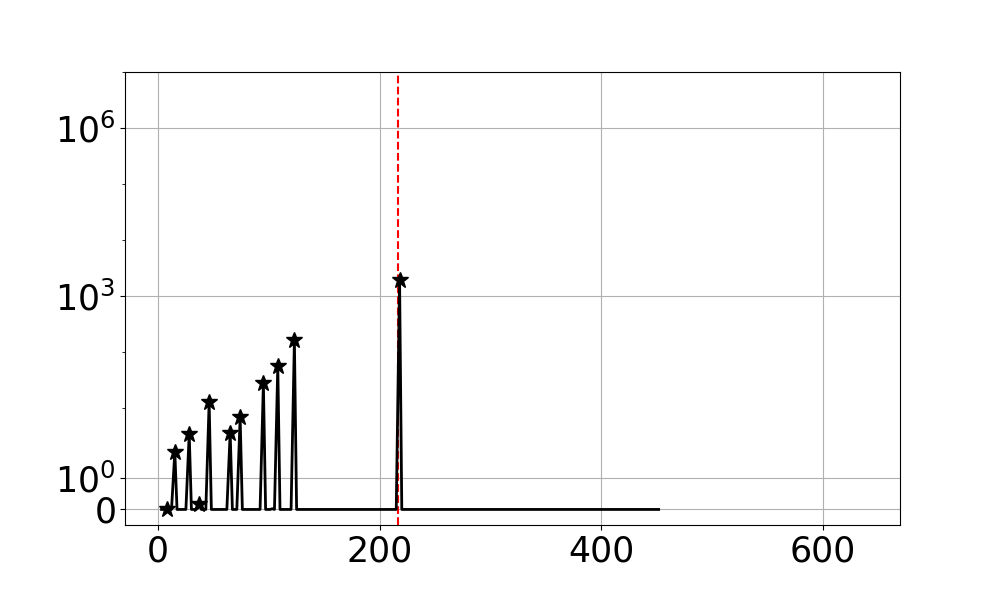}};
	\node[right] at (9.65,-2.0) {\includegraphics[width=3cm,trim=85 50 60 45,clip]{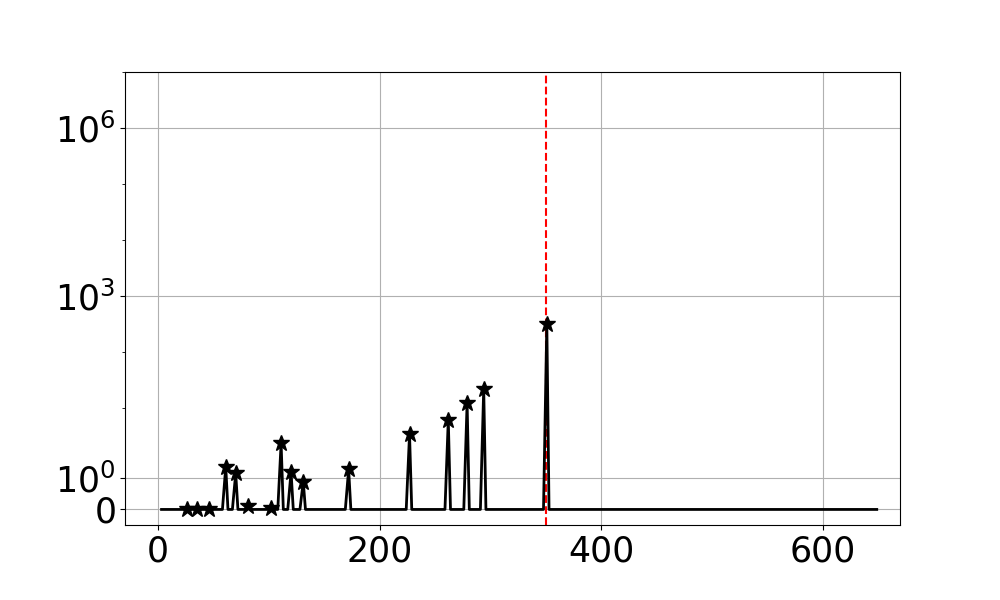}};
	\node[right] at (0.0,-3.9) {\includegraphics[width=3.35cm,trim=15 50 60 45,clip]{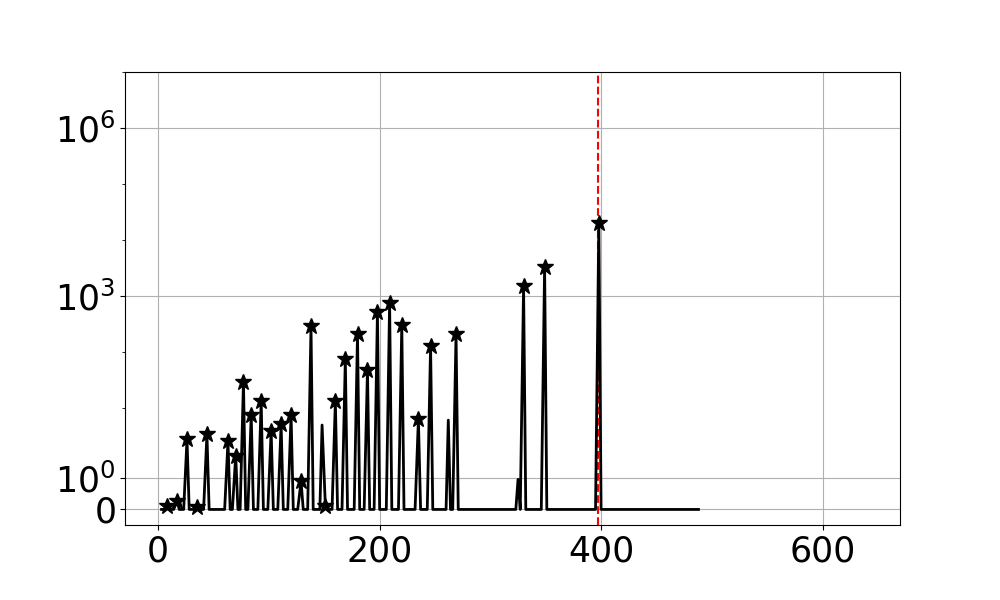}};
	\node[right] at (3.45,-3.9) {\includegraphics[width=3cm,trim=85 50 60 45,clip]{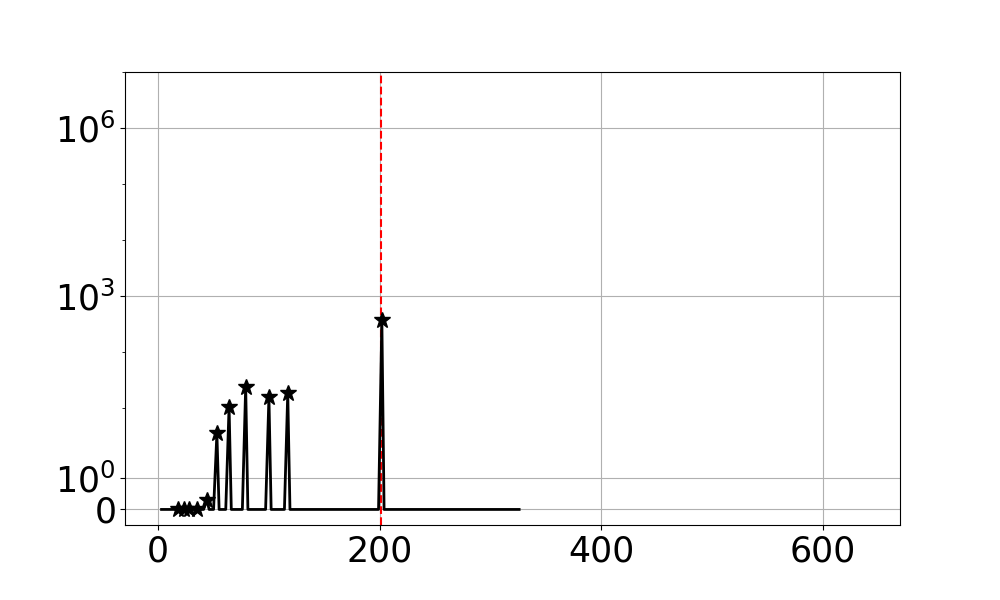}};
	\node[right] at (6.55,-3.9) {\includegraphics[width=3cm,trim=85 50 60 45,clip]{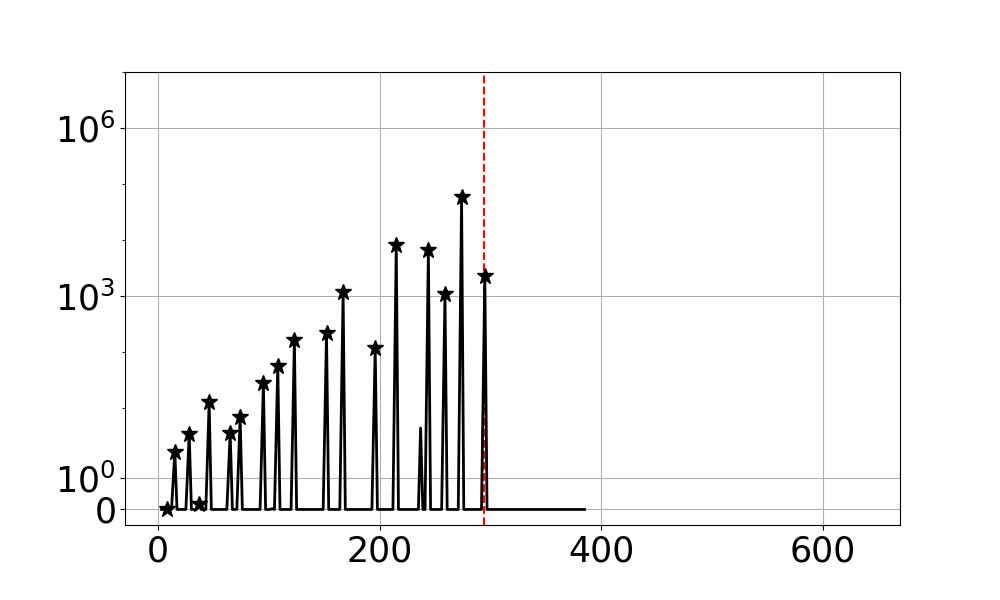}};
	\node[right] at (9.65,-3.9) {\includegraphics[width=3cm,trim=85 50 60 45,clip]{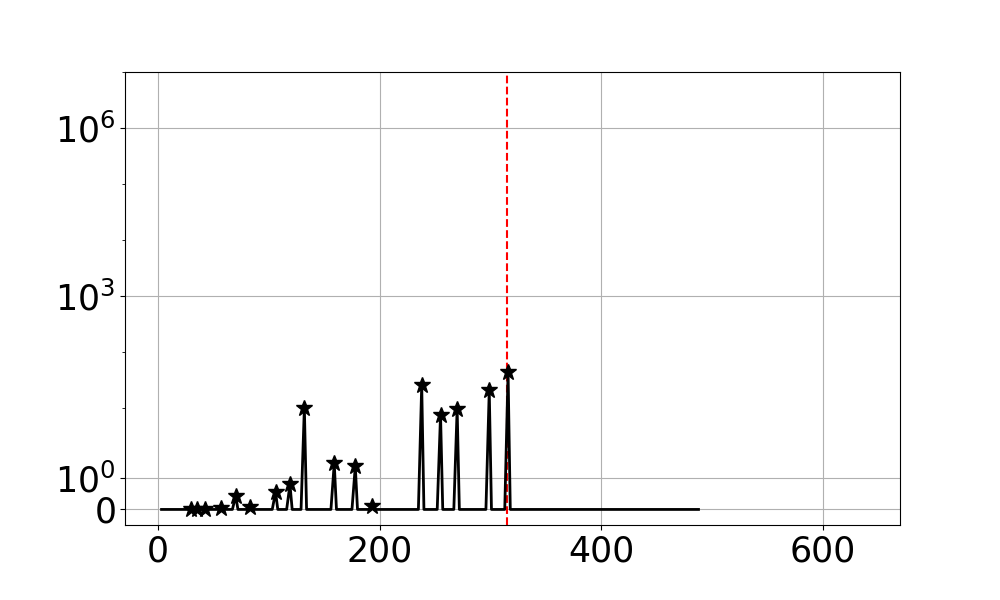}};
	\node[right] at (0.0,-5.8) {\includegraphics[width=3.35cm,trim=15 50 60 45,clip]{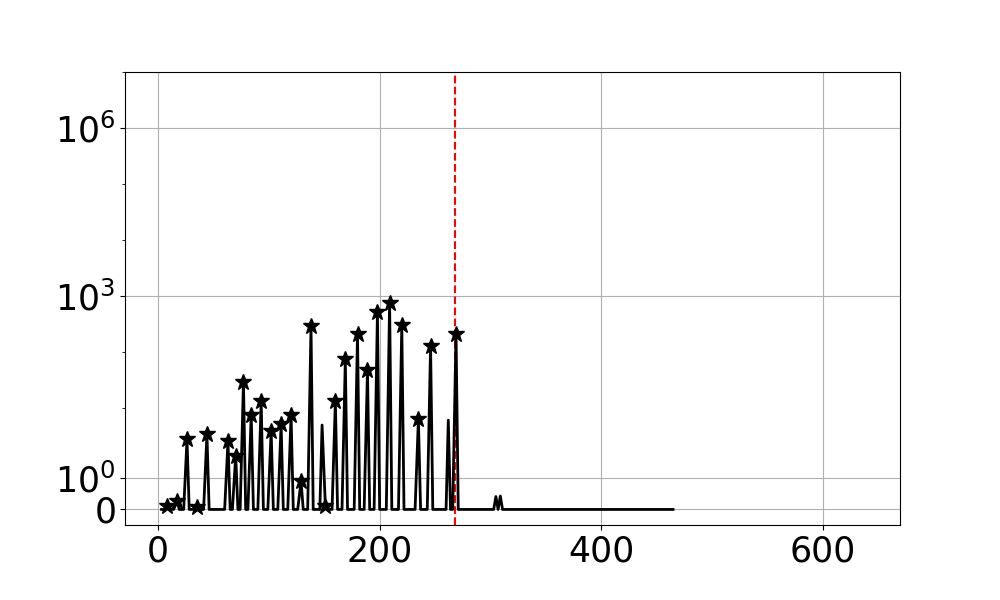}};
	\node[right] at (3.45,-5.8) {\includegraphics[width=3cm,trim=85 50 60 45,clip]{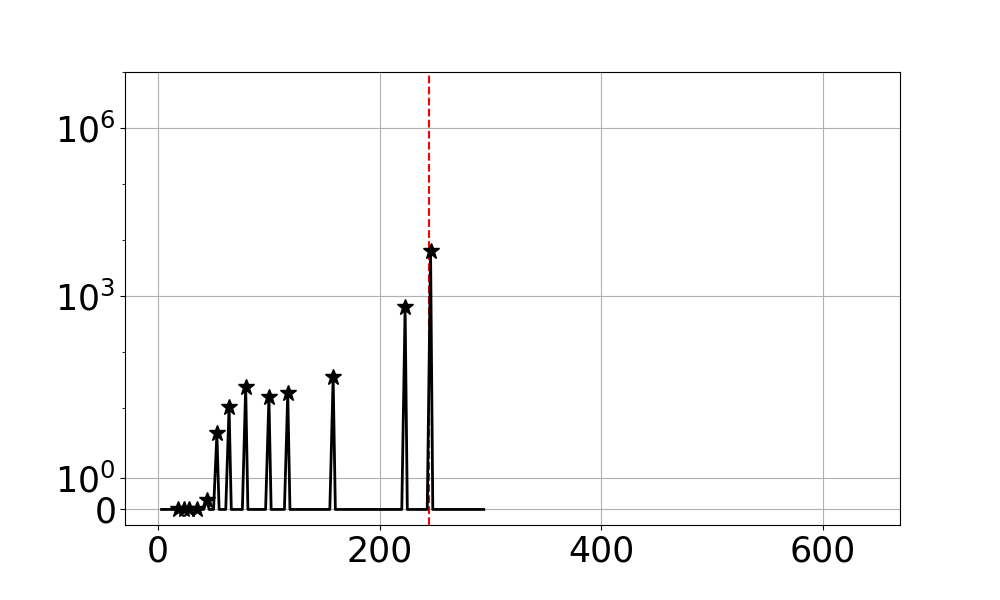}};
	\node[right] at (6.55,-5.8) {\includegraphics[width=3cm,trim=85 50 60 45,clip]{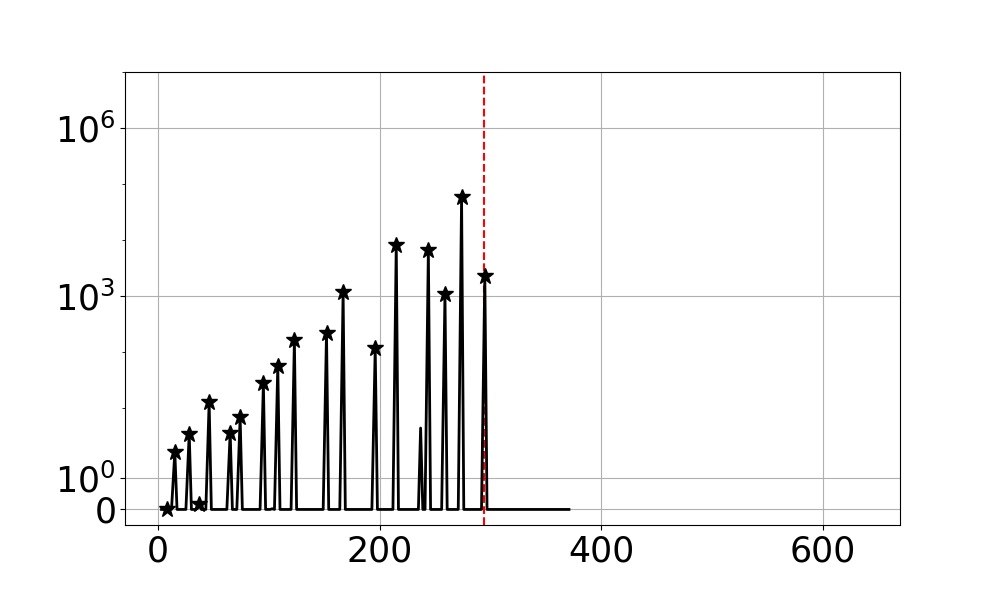}};
	\node[right] at (9.65,-5.8) {\includegraphics[width=3cm,trim=85 50 60 45,clip]{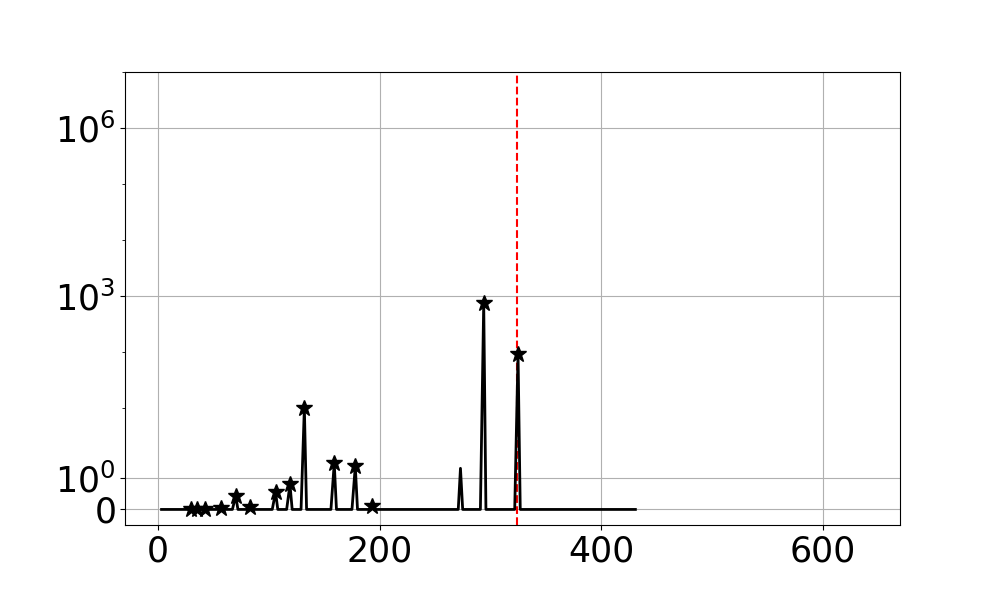}};
	\node[right] at (0.0,-7.8) {\includegraphics[width=3.35cm,trim=15 20 60 45,clip]{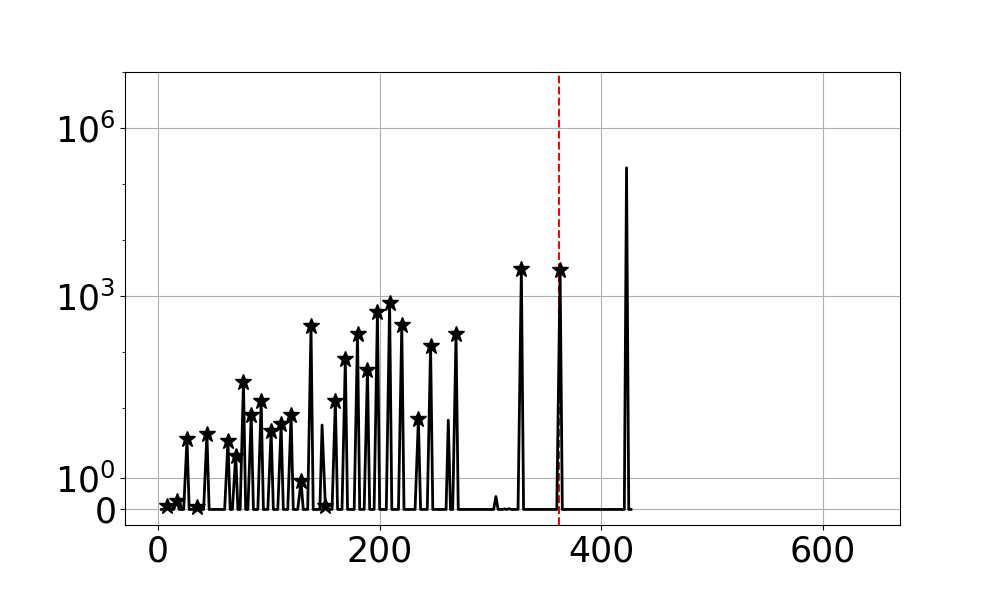}};
	\node[right] at (3.45,-7.8) {\includegraphics[width=3cm,trim=85 20 60 45,clip]{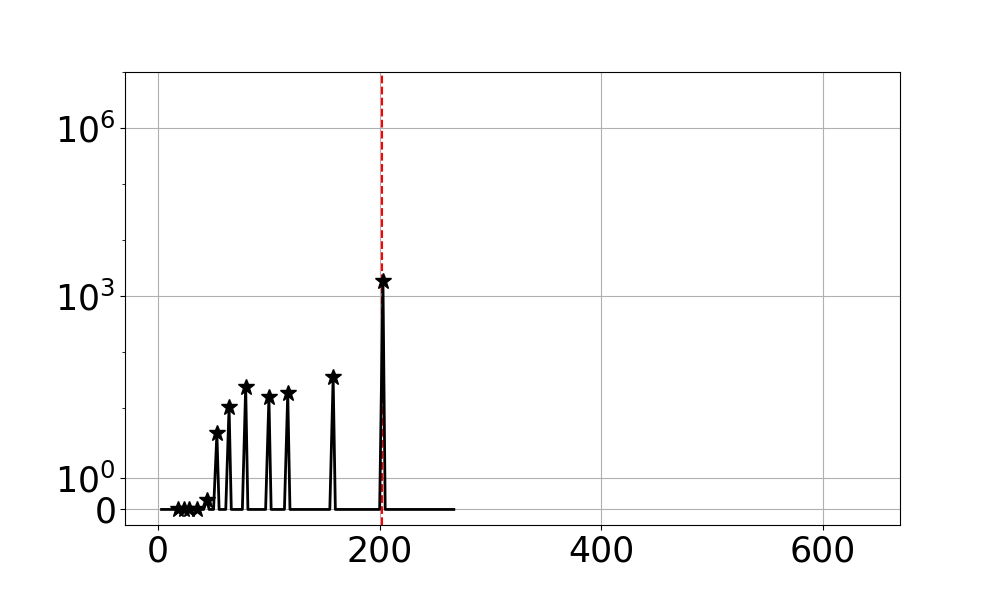}};
	\node[right] at (6.55,-7.8) {\includegraphics[width=3cm,trim=85 20 60 45,clip]{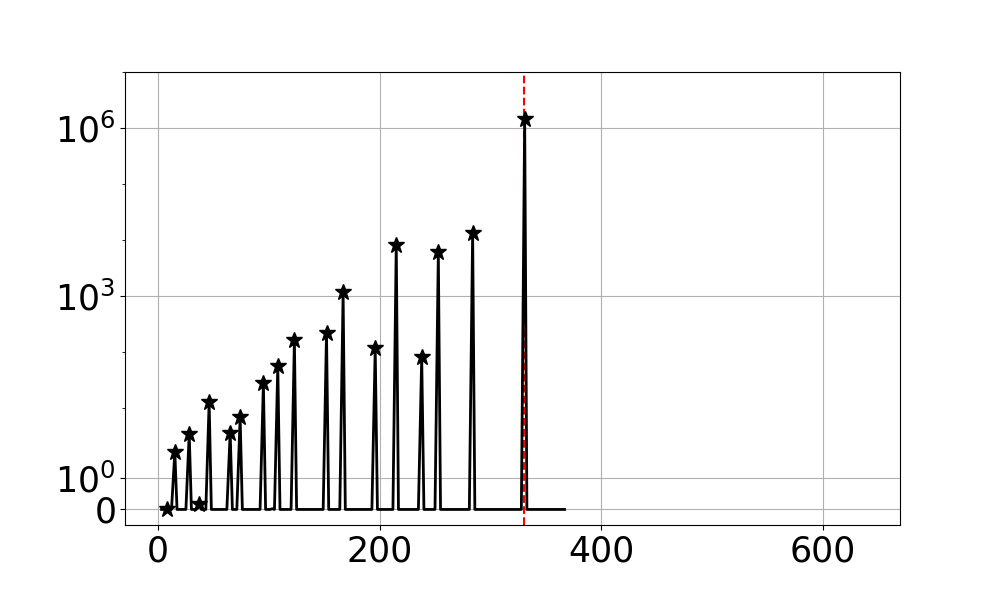}};
	\node[right] at (9.65,-7.8) {\includegraphics[width=3cm,trim=85 20 60 45,clip]{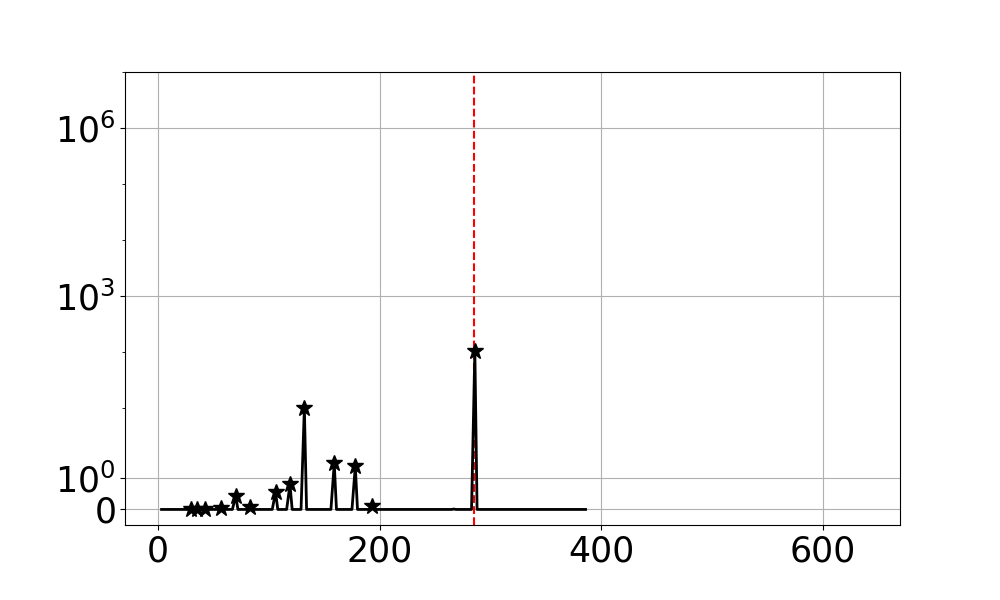}};
	\node[right] at (0.3,-9.1) {{\footnotesize(a)~ST \& \texttt{CIFAR10}}};
	\node[right] at (3.75,-9.1) {{\footnotesize(b)~ST \& \texttt{STL10}}};
	\node[right] at (6.85,-9.1) {{\footnotesize(c)~NLS \& \texttt{CIFAR10}}};
	\node[right] at (9.95,-9.1) {{\footnotesize(d)~NLS \& \texttt{STL10}}};
	\node at (12.97,0) {\rotatebox{-90}{{\tiny $m=5$}}};
	\node at (12.97,-2.0) {\rotatebox{-90}{{\tiny $m=10$}}};
	\node at (12.97,-3.9) {\rotatebox{-90}{{\tiny $m=15$}}};
	\node at (12.97,-5.8) {\rotatebox{-90}{{\tiny $m=20$}}};
	\node at (12.97,-7.8) {\rotatebox{-90}{{\tiny $m=30$}}};
\end{tikzpicture}
\caption{Plot of $\rho_k = \max \{f(x^{k}_{\AAn}) - f(g(x^k)), 0\} / \| \nabla f(x^{k-\hat{m}}) \|^3 $ vs.\ {Oracle calls} for \cref{algo4}. The marks ``$\star$'' indicate rejected $\AAn$ steps. After the last rejected $\AAn$ step (red dashed vertical line), $\rho_k$ mostly stays $ 0 $, which verifies and illustrates \cref{thm3-18}. The $x$-axes of each plot in the rows $m \in \{10,15,20,30\}$ have the same scaling $0$\,--\,$700$. For $m = 5$, the scaling is $0$\,--\,$3,000$. \label{fig:descent}}
\end{figure}
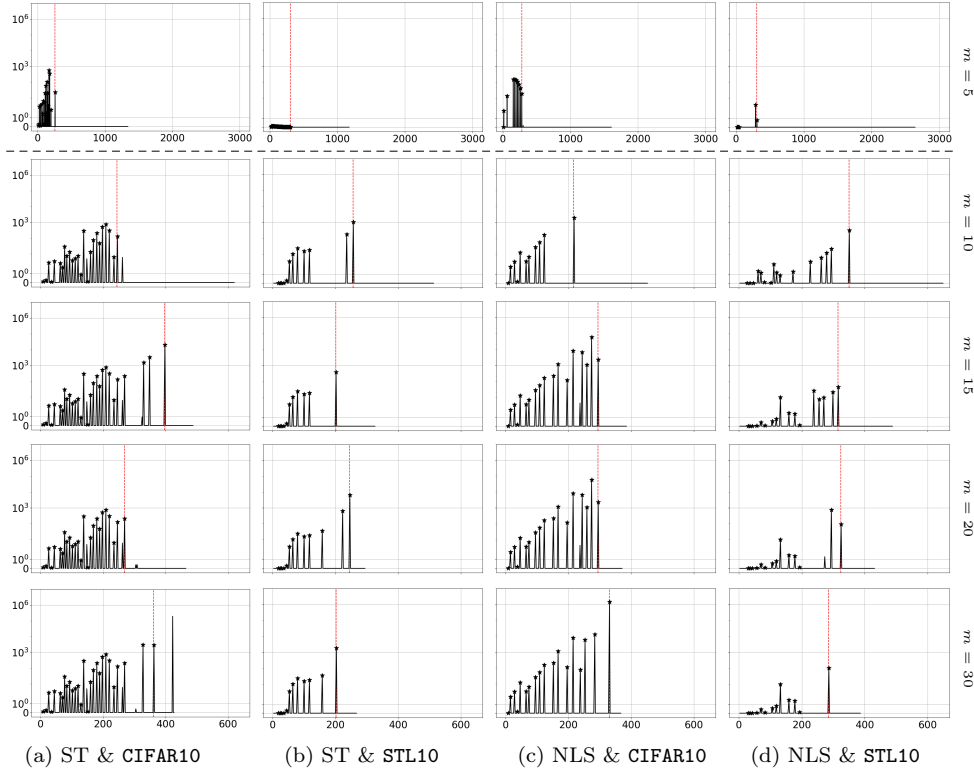

\subsection{Ablation Study} \label{sec:sub-abla}

\begin{figure}[t]
	\setlength{\abovecaptionskip}{-3pt plus 3pt minus 0pt}
	\setlength{\belowcaptionskip}{-10pt plus 3pt minus 0pt}
	\centering
	\includegraphics[width=13.0cm]{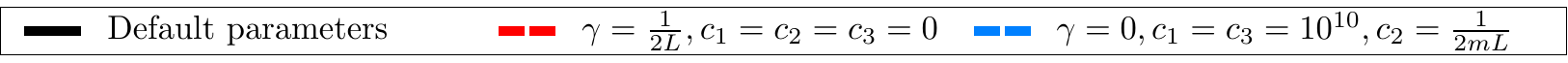}
	\hspace*{-1.2ex}
	\begin{tikzpicture}[scale=1]
	\node[right] at (0.0,0) {\includegraphics[width=3.35cm,trim=15 20 60 45,clip]{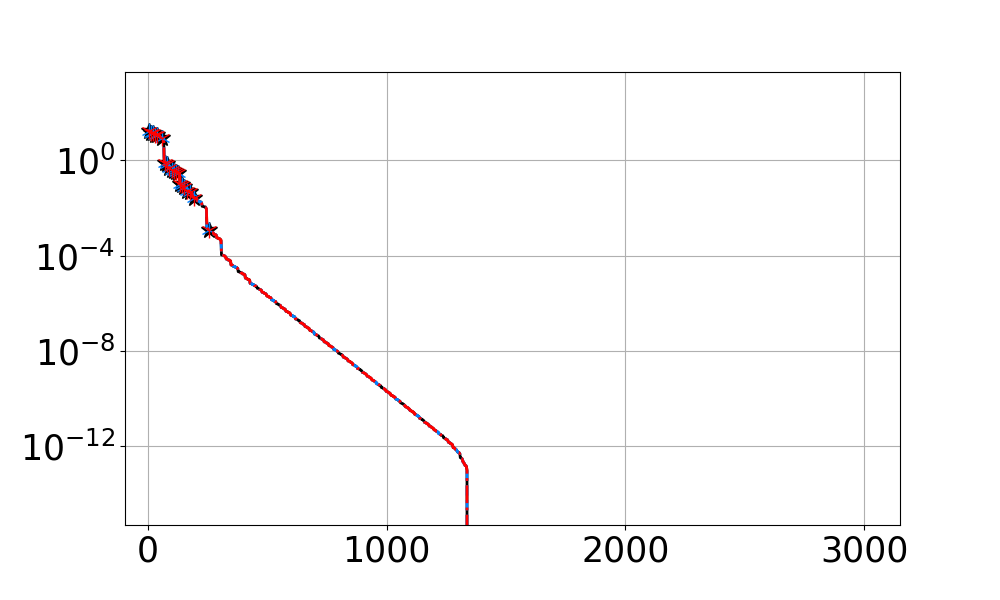}};
	\node[right] at (3.45,0) {\includegraphics[width=3cm,trim=85 20 60 45,clip]{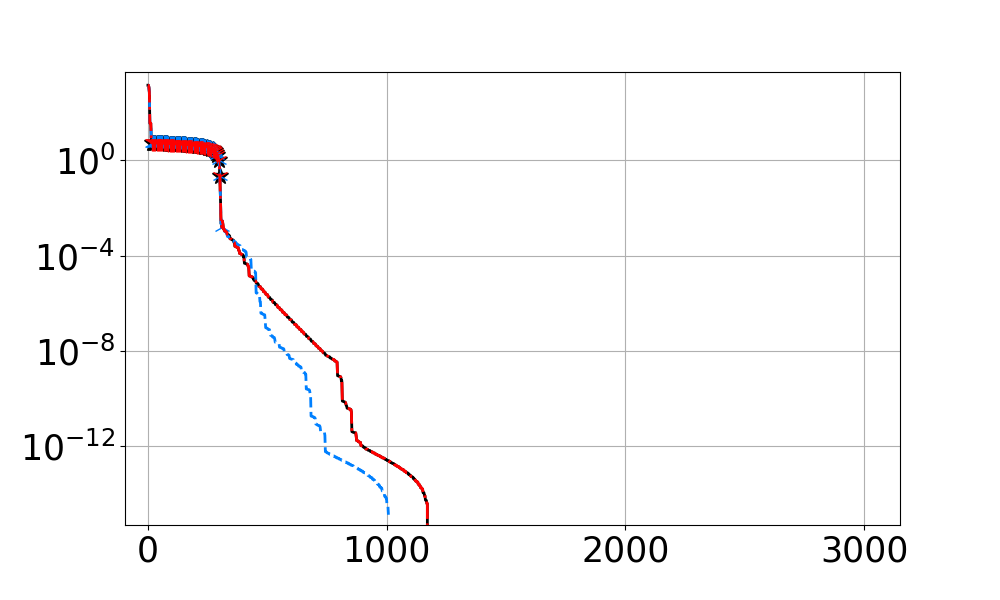}};
	\node[right] at (6.55,0) {\includegraphics[width=3cm,trim=85 20 60 45,clip]{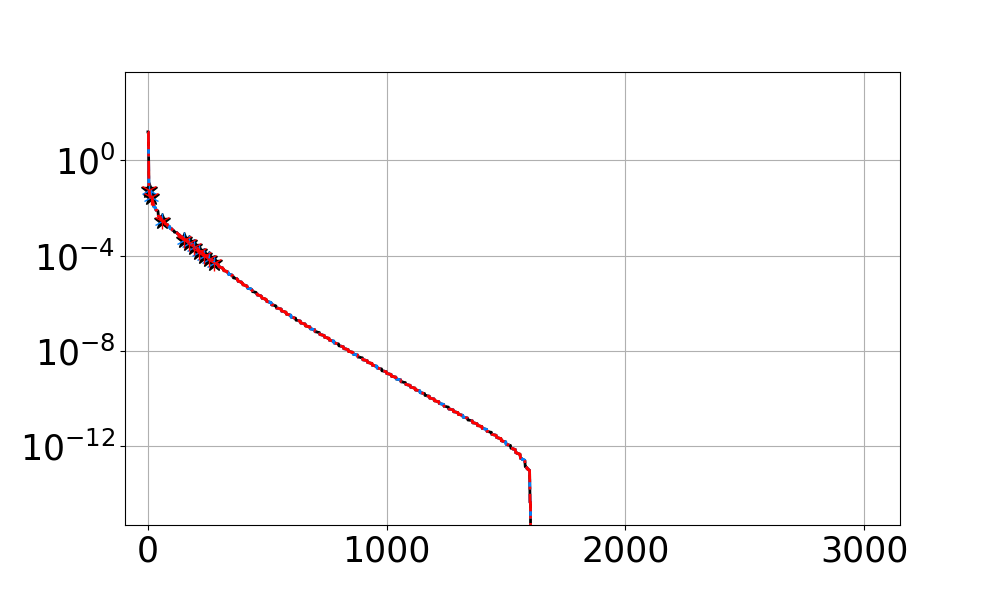}};
	\node[right] at (9.65,0) {\includegraphics[width=3cm,trim=85 20 60 45,clip]{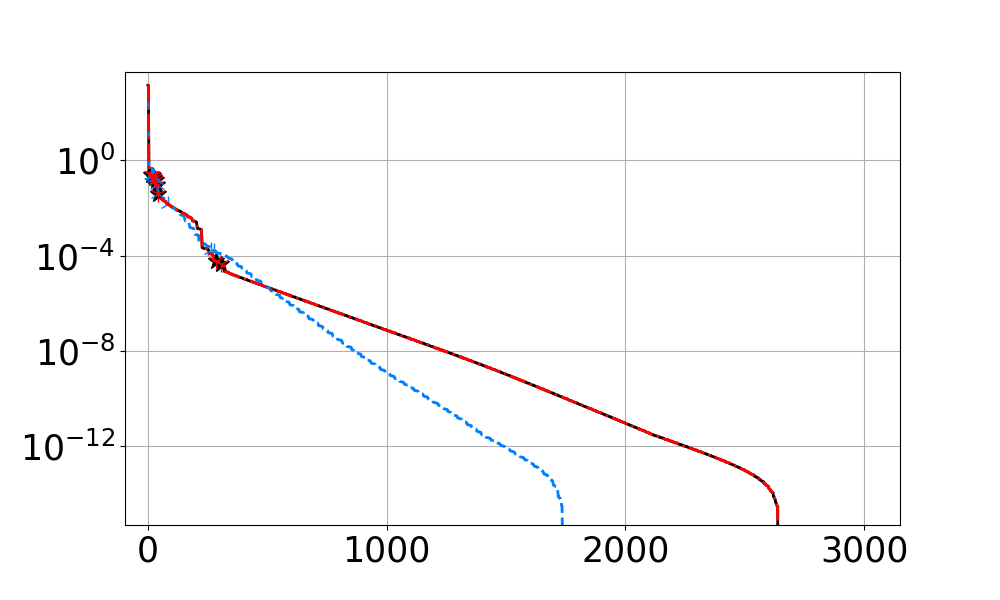}};
	\draw[black,very thin,densely dashed] (0.18,-1.05) -- (12.8,-1.05);
	\node[right] at (0.0,-2.0) {\includegraphics[width=3.35cm,trim=15 50 60 45,clip]{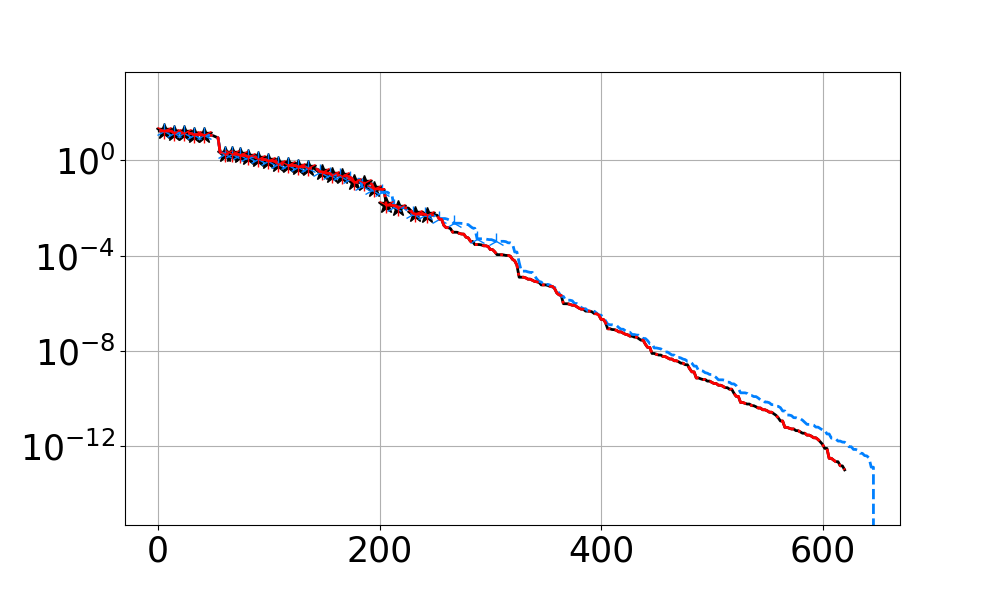}};
	\node[right] at (3.45,-2.0) {\includegraphics[width=3cm,trim=85 50 60 45,clip]{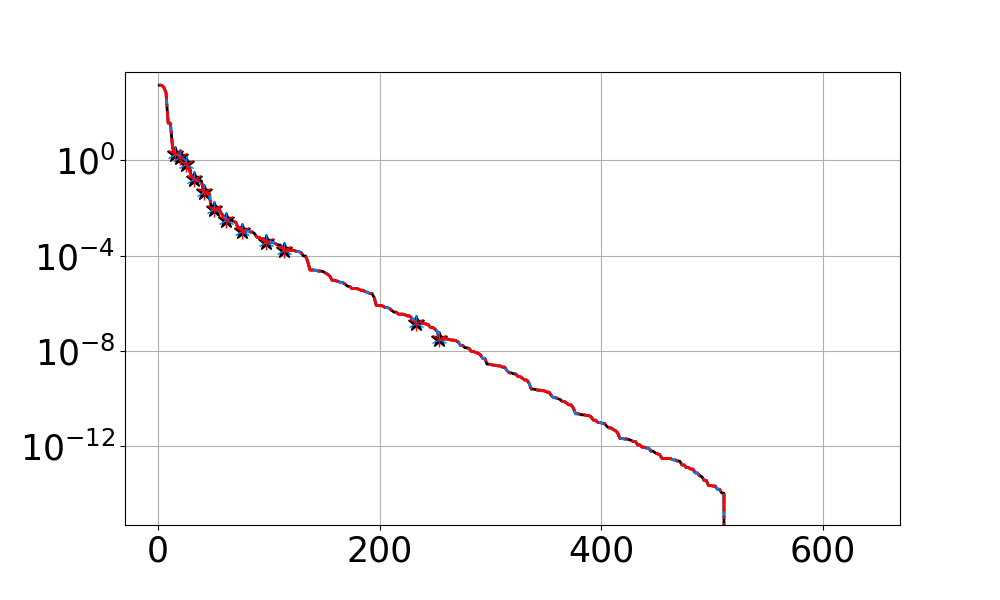}};
	\node[right] at (6.55,-2.0) {\includegraphics[width=3cm,trim=85 50 60 45,clip]{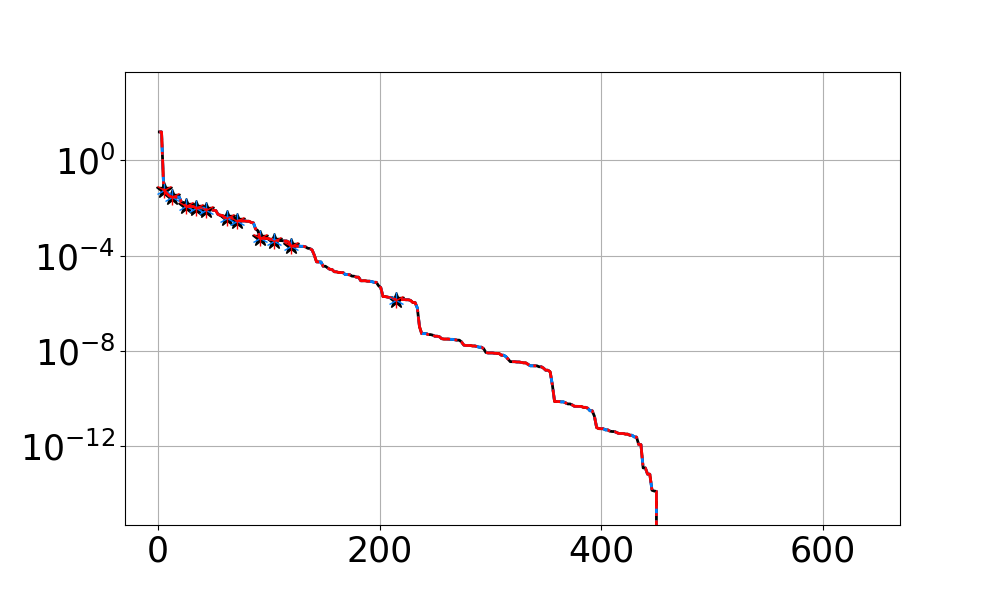}};
	\node[right] at (9.65,-2.0) {\includegraphics[width=3cm,trim=85 50 60 45,clip]{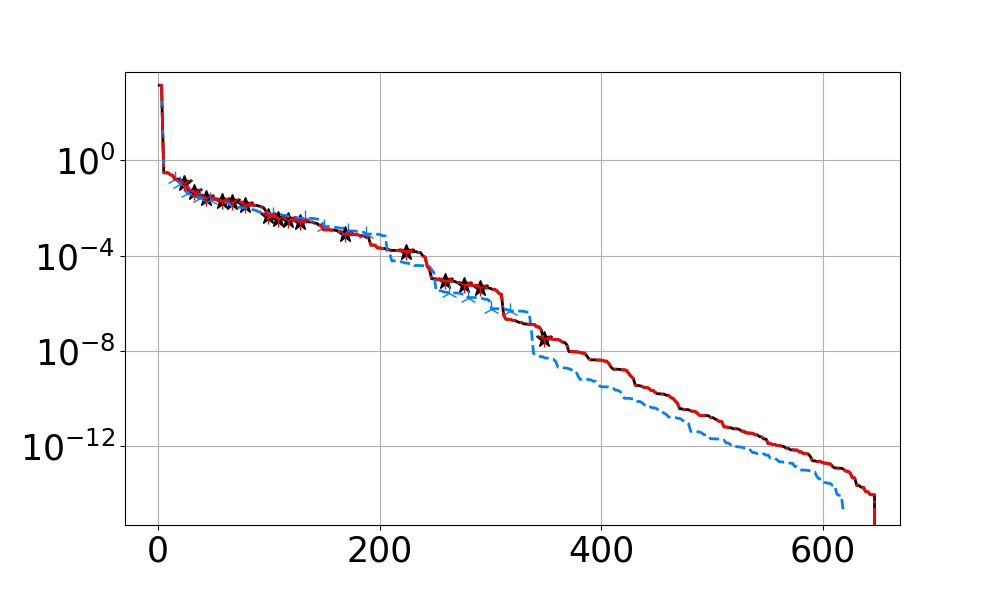}};
	\node[right] at (0.0,-3.9) {\includegraphics[width=3.35cm,trim=15 50 60 45,clip]{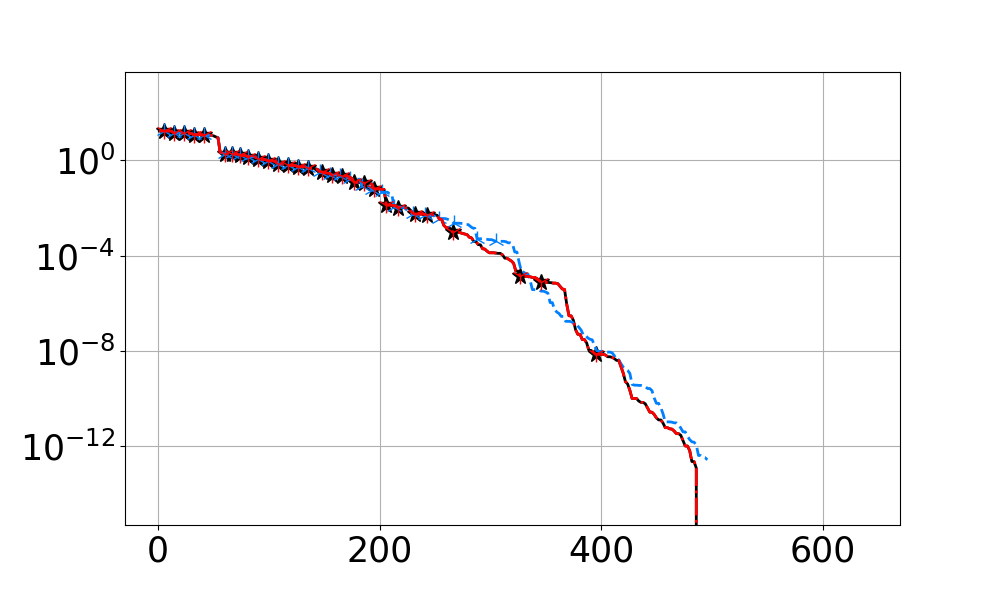}};
	\node[right] at (3.45,-3.9) {\includegraphics[width=3cm,trim=85 50 60 45,clip]{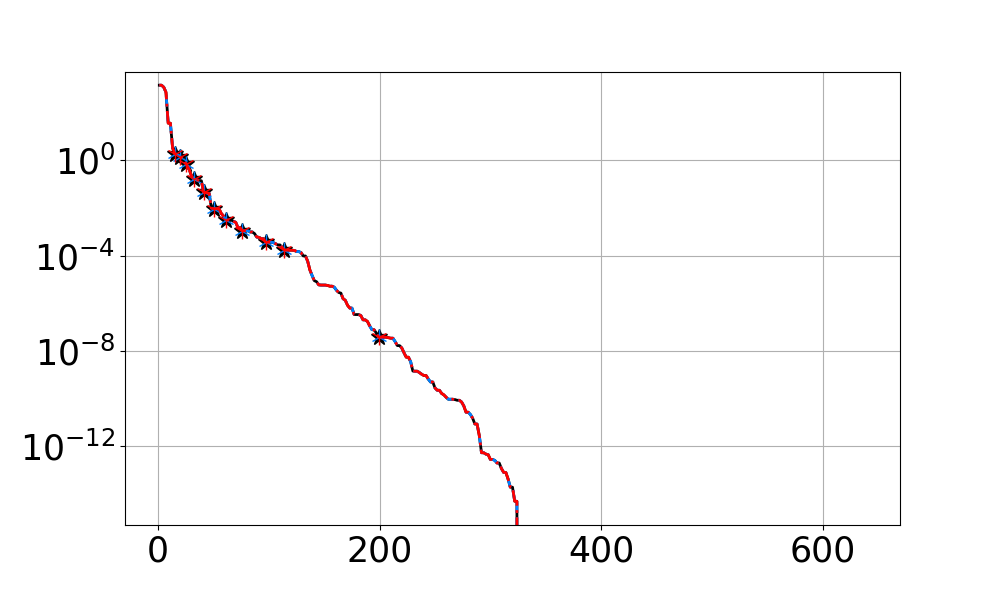}};
	\node[right] at (6.55,-3.9) {\includegraphics[width=3cm,trim=85 50 60 45,clip]{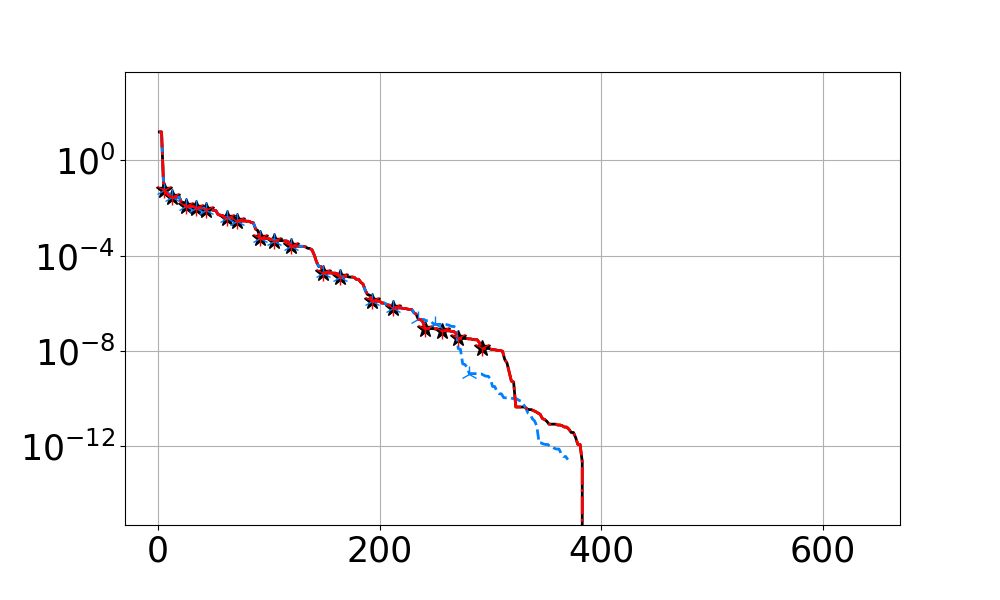}};
	\node[right] at (9.65,-3.9) {\includegraphics[width=3cm,trim=85 50 60 45,clip]{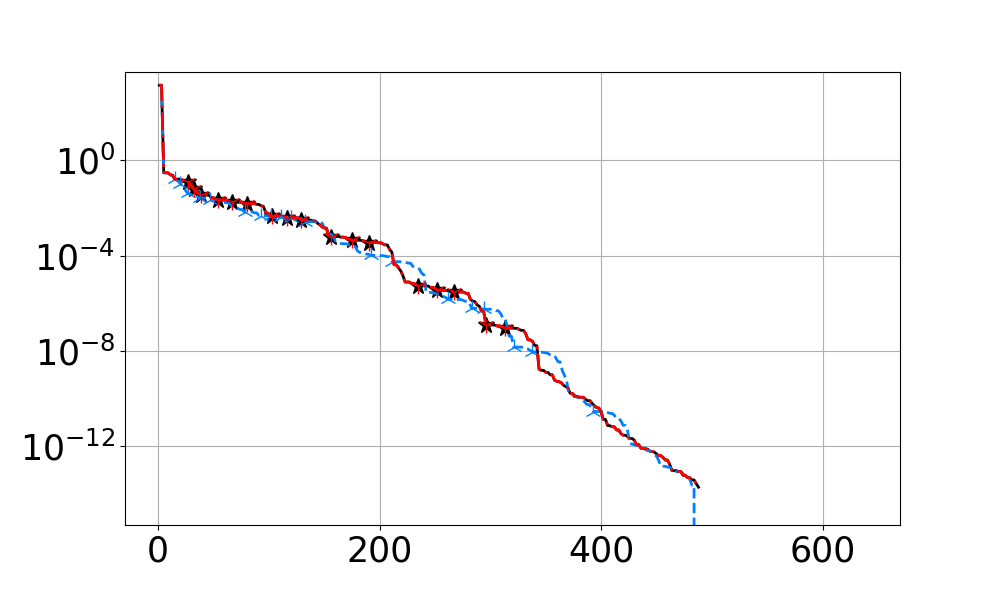}};
	\node[right] at (0.0,-5.8) {\includegraphics[width=3.35cm,trim=15 50 60 45,clip]{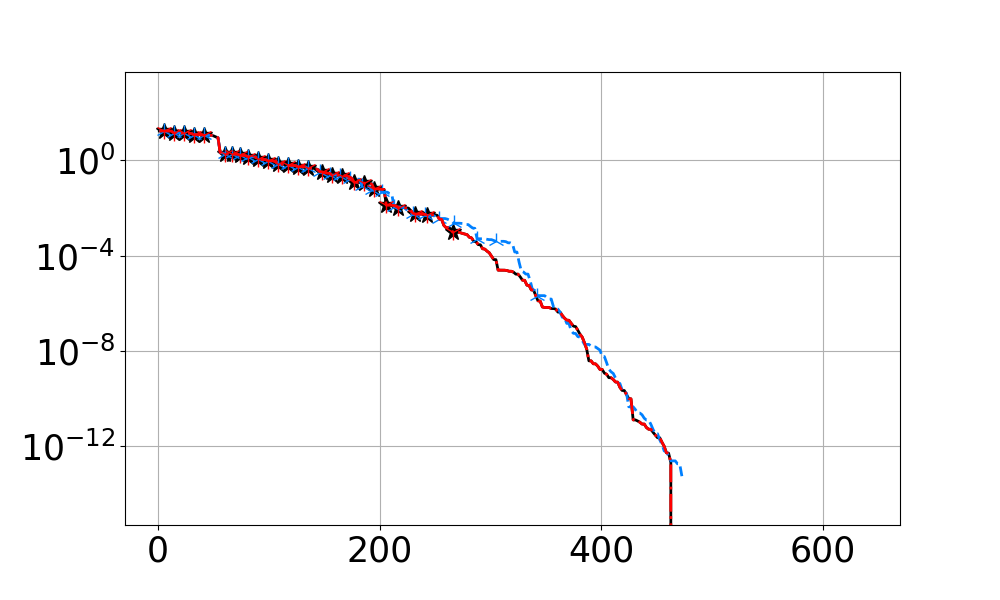}};
	\node[right] at (3.45,-5.8) {\includegraphics[width=3cm,trim=85 50 60 45,clip]{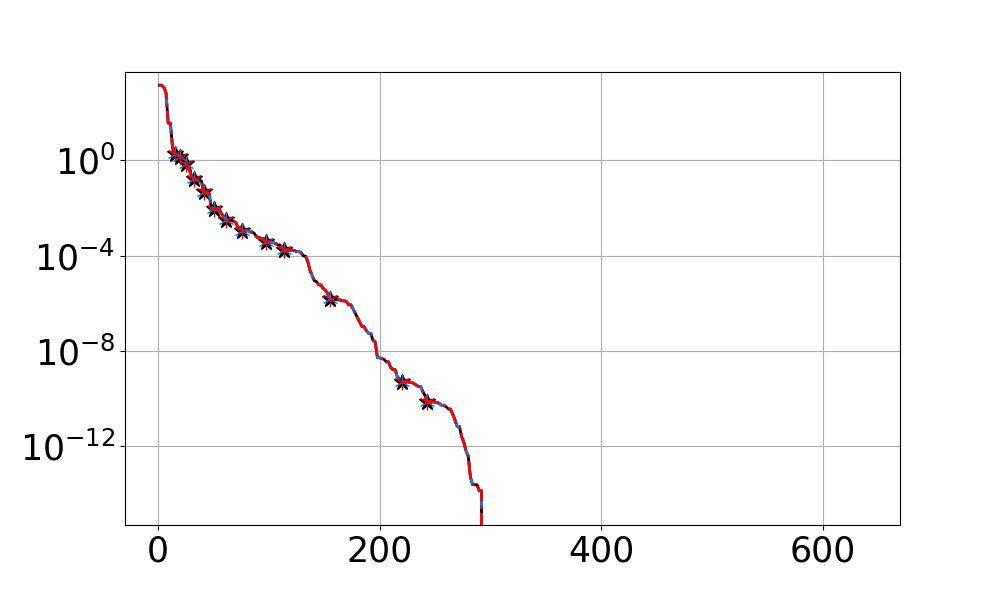}};
	\node[right] at (6.55,-5.8) {\includegraphics[width=3cm,trim=85 50 60 45,clip]{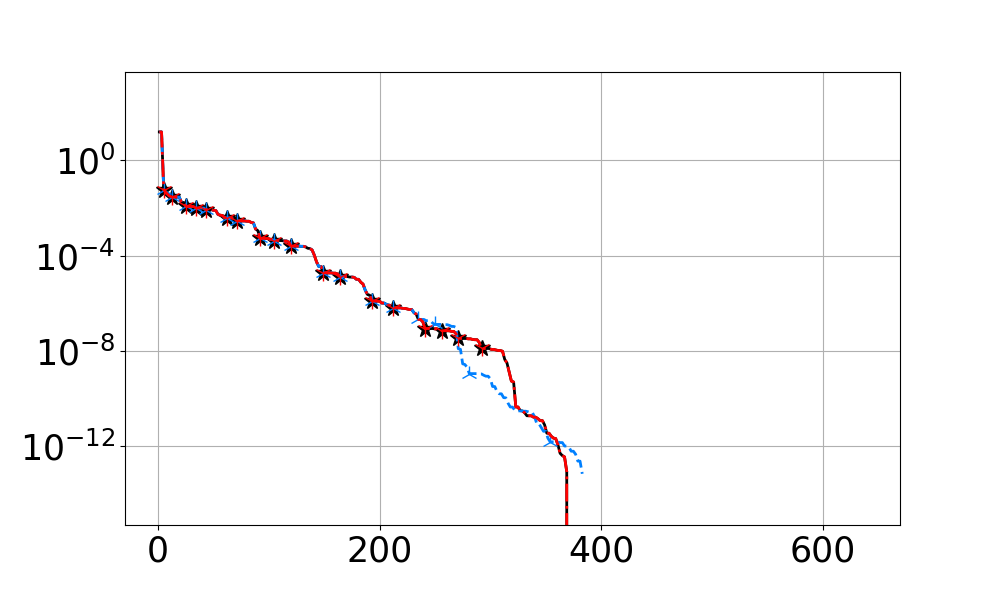}};
	\node[right] at (9.65,-5.8) {\includegraphics[width=3cm,trim=85 50 60 45,clip]{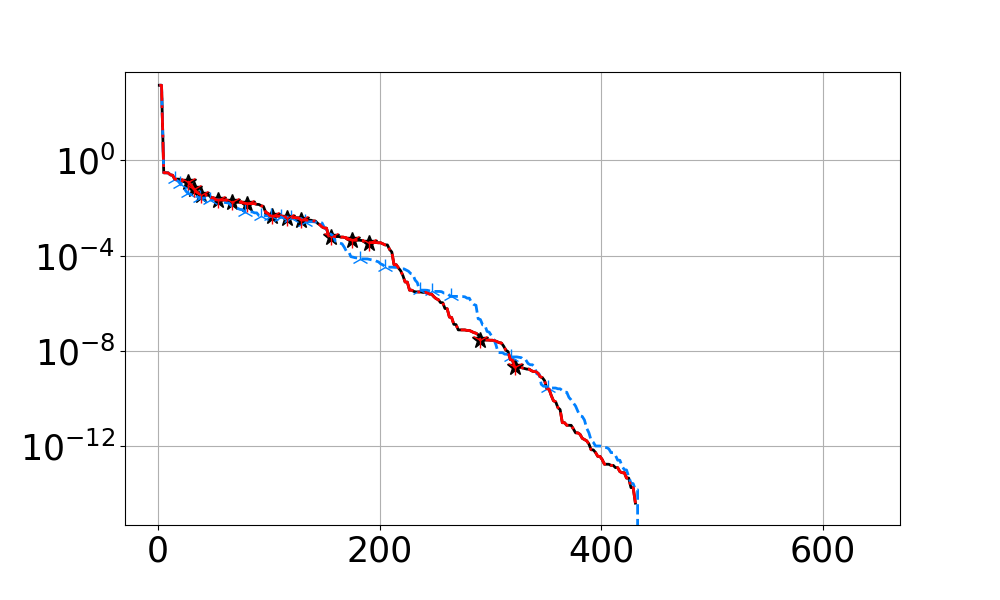}};
	\node[right] at (0.0,-7.8) {\includegraphics[width=3.35cm,trim=15 20 60 45,clip]{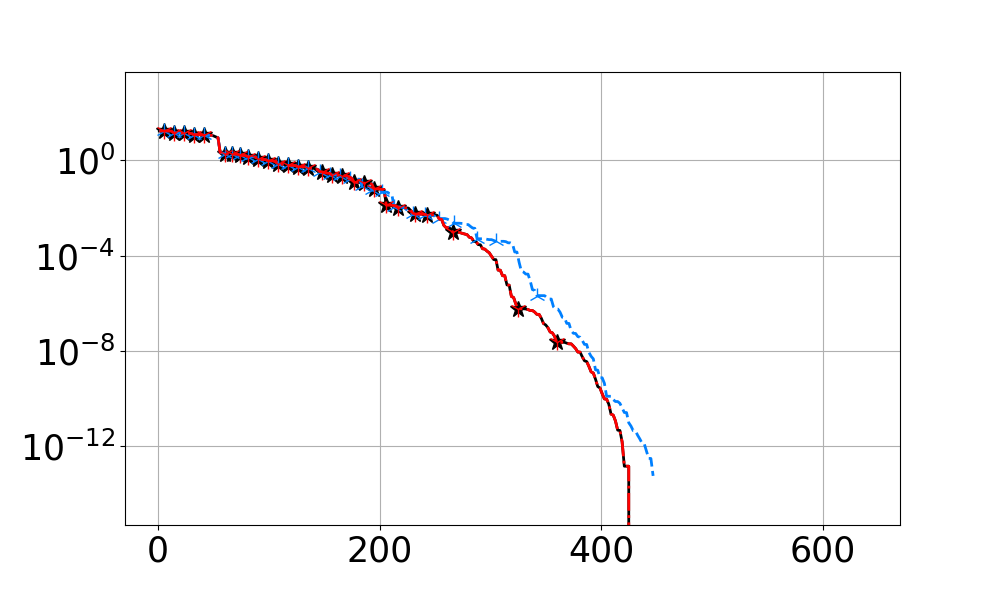}};
	\node[right] at (3.45,-7.8) {\includegraphics[width=3cm,trim=85 20 60 45,clip]{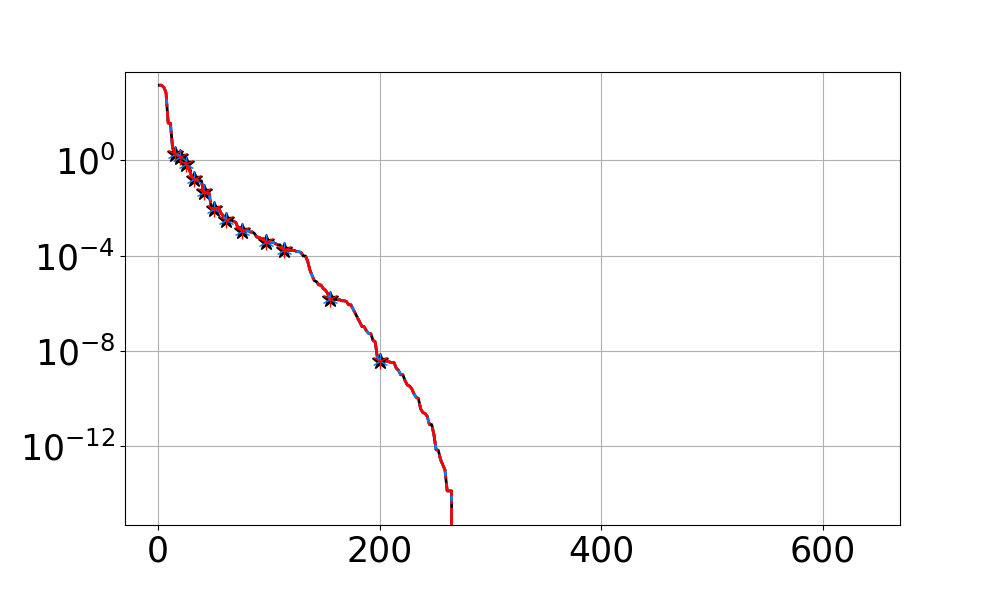}};
	\node[right] at (6.55,-7.8) {\includegraphics[width=3cm,trim=85 20 60 45,clip]{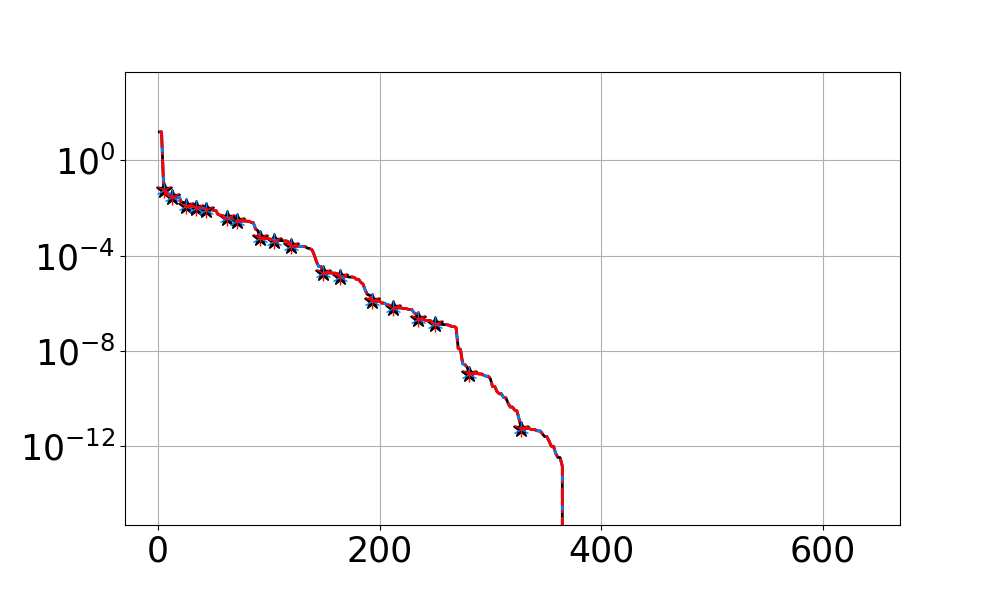}};
	\node[right] at (9.65,-7.8) {\includegraphics[width=3cm,trim=85 20 60 45,clip]{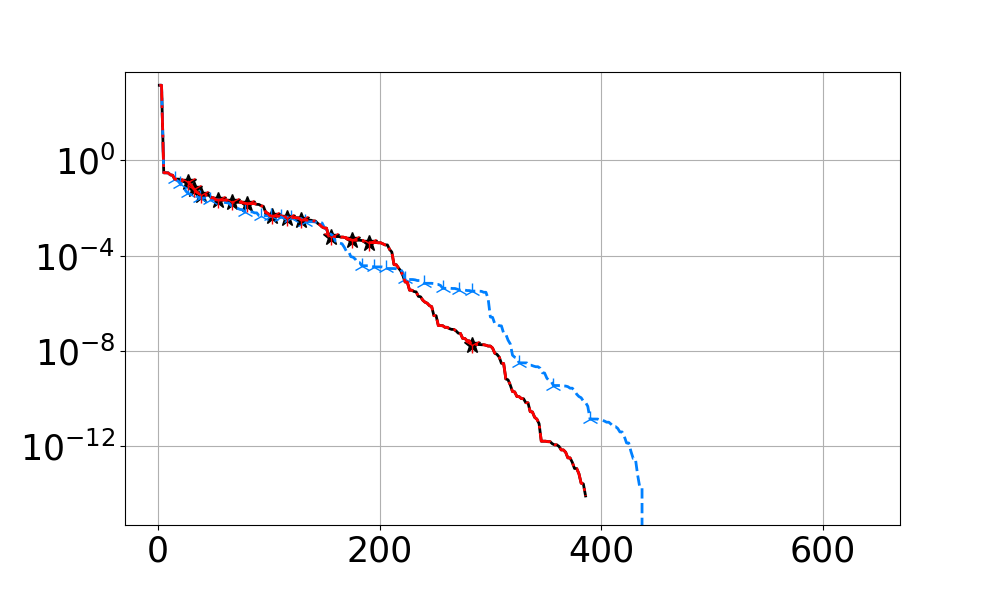}};
	\node[right] at (0.3,-9.1) {{\footnotesize(a)~ST \& \texttt{CIFAR10}}};
	\node[right] at (3.75,-9.1) {{\footnotesize(b)~ST \& \texttt{STL10}}};
	\node[right] at (6.85,-9.1) {{\footnotesize(c)~NLS \& \texttt{CIFAR10}}};
	\node[right] at (9.95,-9.1) {{\footnotesize(d)~NLS \& \texttt{STL10}}};
	\node at (12.97,0) {\rotatebox{-90}{{\tiny $m=5$}}};
	\node at (12.97,-2.0) {\rotatebox{-90}{{\tiny $m=10$}}};
	\node at (12.97,-3.9) {\rotatebox{-90}{{\tiny $m=15$}}};
	\node at (12.97,-5.8) {\rotatebox{-90}{{\tiny $m=20$}}};
	\node at (12.97,-7.8) {\rotatebox{-90}{{\tiny $m=30$}}};
\end{tikzpicture}
\caption{Ablation study of \cref{algo4} using different $c_1$, $c_2$, $c_3$, and $\gamma$. Each plot depicts $ (f(x^k) - f^{*}) / \max \{f^{*}, 1\} $ vs.\ {Oracle calls} for three different runs of \cref{algo4}. We compare the default parameters with the extreme choices $ \gamma = \frac{1}{2L} $, $ c_1=c_2=c_3 = 0 $ and $ \gamma = 0$, $c_1=c_3 = 10^{10}$, $c_2=\frac{1}{2mL} $. The $x$-axes of each plot in the rows $m \in \{10,15,20,30\}$ have the same scaling $0$\,--\,$700$. For $m = 5$, the scaling is $0$\,--\,$3,000$. \label{fig:extreme_f}}
\end{figure}

Finally, we provide an additional ablation study for the parameters $c_1$, $c_2$, $c_3$, and $\gamma$ used in \cref{algo4} and in the definition of the descent condition \eqref{eq:descent-condition}. Based on \cref{thm4-1}, $c_1$, $c_2$, $c_3$, and $\gamma$ need to satisfy the conditions $ 0 < \gamma < \frac{1}{2L} $, $ c_1, c_3 > 0 $, and $ 0 < c_2 < \frac{1}{2mL} $. We compare our default choice with the following extreme sets of parameters:
\begin{equation} \label{eq:extreme}  \gamma = \frac{1}{2L}, \; c_1=c_2=c_3 = 0, \quad \text{and} \quad \gamma = 0, \; c_1=c_3 = 10^{10},\; c_2=\frac{1}{2mL}. \end{equation}
These two choices correspond to highly strict and loose acceptance criteria for the $\AAn$ step $ x_{\AAn}^{k} $. Since $ \nu \in (2,3)$ has only limited  influence, we omit an explicit ablation study for $ \nu $ and use the default value $\nu = 2.1$. Figure~\ref{fig:extreme_f} demonstrates that \cref{algo4} is robust with respect to the choice of $c_1$, $c_2$, $c_3$, and $\gamma$. In particular, performance is only affected marginally when using the more extreme parameters \eqref{eq:extreme}. 


\section{Conclusion}
In this work, we study descent properties of an Anderson accelerated gradient method with restarting. We first show that the iterates generated by $\AAr$ are equivalent to the iterates generated by $\GMRES$ after an additional gradient step within each restarting cycle. Based on the symmetry of the underlying system matrix, we then analyze the error between the iterates generated by $\GMRES$ and $\CR$ and verify that this error is controllable and related to some higher-order perturbation terms. After connecting $\CR$ and $\CG$, the desired descent property for $\AAr$ can be expressed in terms of distances to the respective optimal solution for the iterates generated by $\CG$. We establish such a convergence result for $\CG$ utilizing classical techniques. Combining these different observations, we prove that $\AAr$ can decrease the objective function $f$ locally. These novel findings can be used in the design of effective, function value-based globalization mechanisms for $\AAr$ approaches. We propose one such possible $\AAr$ globalization and conduct numerical experiments on two large-scale learning problems that illustrate our theoretical results.
 

\section*{Acknowledgments} We would like to thank the Associate Editor and three anonymous reviewers for their detailed and constructive comments, which have helped greatly to improve the quality and presentation of the manuscript.

\appendix

\section{Proof of~\cref{q_linear_conv}} \label{app:q-proof}
\begin{proof}We show that \cref{q_linear_conv} is a direct application of \cite[Theorem 5.1]{pollock2021anderson}. Due to \eqref{contract_g}, the constant $\kappa_g$ in \cite{pollock2021anderson} reduces to $1-\frac{1}{\kap}$ and we have $\theta_k \leq 1$ and $\beta_k = 1$. Moreover, all the points of interest lie in $\B_r(\xopt)$, so all the expansions of the residuals in \cite[Section 3]{pollock2021anderson} are legitimate. The core estimate (5.18) in \cite[Theorem 5.1]{pollock2021anderson} then reduces to \eqref{eq:hk-q}. The proof is complete if all assumptions in \cite[Theorem 5.1]{pollock2021anderson} hold. 

Using $x^i \in \B_r(\xopt)$, $i = k-\hat m,\dots,k$ and \eqref{contract_g}, it follows
\[ \|h^{i}-h^{i-1}\| \geq \|x^{i}-x^{i-1}\| - \|g(x^{i})-g(x^{i-1})\| \geq \kap^{-1} \|x^{i}-x^{i-1}\| \quad \forall~i = k-\hat m+1,\dots,k. \]
This is exactly Assumption 2.3 in \cite{pollock2021anderson} (see also \cite[Remark 2.1]{pollock2021anderson}). Next, we verify the sufficient linear independence condition introduced in \cite[Lemma 5.2]{pollock2021anderson}. Let us define $\tilde H_k=[h^{k}-h^{k-1},\dots,h^{k-\hat m+1}-h^{k-\hat m}]=:[v^1,\dots,v^{\hat m}]$. We note that there is a fixed nonsingular matrix $P\in\R^{\hat m\times \hat m}$ such that $\tilde H_k=H_kP$ and $\kappa(\tilde H_k^\top \tilde H_k)\leq \kappa(H_k^\top H_k)\kappa(P^\top P)$. Therefore, by \cref{prop-HXeq} and using $x^i \in U_1$, $i = k-\hat m,\dots,k$, \ref{A5} implies that the condition number of $\tilde H_k^\top \tilde H_k$ is bounded by some $\tilde M^2$. Let  $\cV_i=\spa\{v^1,\dots,v^i\}$ denote the linear subspace spanned by the first $i$ columns of $\tilde H_k$ and let $\tilde H_k=Q_kR_k$ be the QR decomposition of $\tilde H_k$. We then have $\kappa(R_k^\top R_k) = \kappa(\tilde H_k^\top \tilde H_k) \leq \tilde M^2$. Furthermore, let $\{r_{ii}\}_{1\leq i\leq \hat m}$ denote the diagonal entries of $R_k$. By \cite[Proposition 5.2]{pollock2021anderson}, it follows $r_{11}^2 = \|v_1\|^2$ and $r_{ii}^2=\|v_i\|^2\sin^2(v_i,\cV_{i-1})$ for all $2\leq i\leq \hat m$.
    Since $R_k$ is upper triangular, the diagonal entries $r_{ii}$, $i=1,\dots,\hat m$, are exactly the eigenvalues of $R_k$. Consequently, we obtain
    \[  ({\|v_i\|^2}/{\|v_1\|^2})\cdot\sin^2(v_i,\cV_{i-1}) = {r_{ii}^2}/{r_{11}^2}\geq {\sigma_{\min}(R_k)^2}/{\sigma_{\max}(R_k)^2}\geq 1/{\tilde M^2}.             \]
    In addition, we have ${\|v_i\|^2}/{\|v_1\|^2}\leq{\sigma_{\max}(\tilde H_k)^2}/{\sigma_{\min}(\tilde H_k)^2} \leq \tilde M^2$.
    Combining these inequalities, this yields $|\sin(v_i,\cV_{i-1})|\geq \tilde M^{-2}$ which verifies the last remaining assumption in \cite[Lemma 5.2 and Theorem 5.1]{pollock2021anderson}. This concludes the proof. \end{proof}
 
\section{Proof of~\cref{thm4-10}} \label{app:cg-proof}

\begin{proof}
    Similar to \eqref{dist_decom} and utilizing the projection $y^{(k+1)}$, we have:
    \begin{align*}
        & \hspace{-4ex}\|\bar y^{k}-y^*\|^2-\|y^{k+1}-y^*\|=\|\bar y^k-y^{(k+1)}\|^2-\|y^{k+1}-y^{(k+1)}\|^2\\&=\|\bar y^k-y^{k+1}\|^2+2\langle \bar y^k-y^{k+1},y^{k+1}-y^{(k+1)}\rangle \\
        &=\|a_kp^k-{L^{-1}}r^k\|^2+2\langle a_kp^k-{L^{-1}}r^k ,\gamma_kp^k\rangle \\
        &=a_k^2\|p^k\|^2-{2a_k}L^{-1}\langle p^k,r^k\rangle+L^{-2}\|r^k\|^2+2a_k\gamma_k\|p^k\|^2-{2\gamma_k}L^{-1}\langle p^k,r^k\rangle \\
        &=(a_k^2+2a_k\gamma_k)\|p^k\|^2+\tfrac{1}{L}[L^{-1}-2\gamma_k-2a_k]\|r^k\|^2 \\
        &\geq \left[a_k^2+2a_k\gamma_k-{2\gamma_k}L^{-1}-2a_kL^{-1}+L^{-2}\right]\|r^k\|^2 \\ & = \left[2\gamma_k(a_k-{L^{-1}})+(a_k-{L^{-1}})^2\right]\|r^k\|^2\geq 0,
    \end{align*}
    where we have used Property (\rmnum3) to show that $\langle p^k,r^k \rangle=\|r^k\|^2$, Property (\rmnum4) to show that $\|p^k\|^2\geq \|r^k\|^2$, and Property (\rmnum9) to show that $\gamma_k\geq a_k\geq {1}/{L}$.
\end{proof}
 
\bibliographystyle{siamplain}
\bibliography{descent_aa_bib}

\begin{thebibliography}{10}

\bibitem{anderson1965iterative}
{\sc D.~G. Anderson}, {\em Iterative procedures for nonlinear integral equations}, J. ACM, 12 (1965), pp.~547--560.

\bibitem{aravkin2012robust}
{\sc A.~Aravkin, M.~P. Friedlander, F.~J. Herrmann, and T.~Van~Leeuwen}, {\em Robust inversion, dimensionality reduction, and randomized sampling}, Math. Program., 134 (2012), pp.~101--125.

\bibitem{aravkin2011robust}
{\sc A.~Aravkin, T.~Van~Leeuwen, and F.~Herrmann}, {\em Robust full-waveform inversion using the student's t-distribution}, in SEG Tech. Program Expanded Abstracts, 2011, pp.~2669--2673.

\bibitem{artacho2008siesta}
{\sc E.~Artacho, E.~Anglada, O.~Di{\'{e}}guez, J.~D. Gale, A.~Garc{\'{\i}}a, J.~Junquera, R.~M. Martin, P.~Ordej{\'{o}}n, J.~M. Pruneda, D.~S{\'{a}}nchez-Portal, and J.~M. Soler}, {\em The {SIESTA} method; developments and applicability}, J. Phys.-Condes. Matter, 20 (2008).

\bibitem{bai1994newton}
{\sc Z.~Bai, D.~Hu, and L.~Reichel}, {\em {A Newton basis GMRES implementation}}, IMA J. Numer. Anal., 14 (1994), pp.~563--581.

\bibitem{bian2021anderson}
{\sc W.~Bian, X.~Chen, and C.~Kelley}, {\em Anderson acceleration for a class of nonsmooth fixed-point problems}, SIAM J. Sci. Comput.,  (2021), pp.~S1--S20.

\bibitem{ChuDupLegSer21}
{\sc M.~Chupin, M.-S. Dupuy, G.~Legendre, and E.~S\'{e}r\'{e}}, {\em Convergence analysis of adaptive {DIIS} algorithms with application to electronic ground state calculations}, ESAIM Math. Model. Numer. Anal., 55 (2021), pp.~2785--2825.

\bibitem{coates2011analysis}
{\sc A.~Coates, A.~Ng, and H.~Lee}, {\em An analysis of single-layer networks in unsupervised feature learning}, in Proc. Int. Conf. Artif. Intell. Stat. (AISTATS), 2011, pp.~215--223.

\bibitem{ermis2020a3dqn}
{\sc M.~Ermis and I.~Yang}, {\em {A3DQN: Adaptive Anderson acceleration for deep Q-networks}}, in 2020 IEEE Symposium Series on Computational Intelligence (SSCI), IEEE, 2020, pp.~250--257.

\bibitem{evans2020proof}
{\sc C.~Evans, S.~Pollock, L.~G. Rebholz, and M.~Xiao}, {\em A proof that {Anderson} acceleration improves the convergence rate in linearly converging fixed-point methods (but not in those converging quadratically)}, SIAM J. Numer. Anal., 58 (2020), pp.~788--810.

\bibitem{eyert1996comparative}
{\sc V.~Eyert}, {\em A comparative study on methods for convergence acceleration of iterative vector sequences}, Journal of Computational Physics, 124 (1996), pp.~271--285.

\bibitem{fang2009two}
{\sc H.-r. Fang and Y.~Saad}, {\em Two classes of multisecant methods for nonlinear acceleration}, Numer. Linear Algebra Appl., 16 (2009), pp.~197--221.

\bibitem{fong2012cg}
{\sc D.~C.-L. Fong and M.~Saunders}, {\em {CG versus MINRES: An empirical comparison}}, Sultan Qaboos University Journal for Science [SQUJS], 17 (2012), pp.~44--62.

\bibitem{fu2020anderson}
{\sc A.~Fu, J.~Zhang, and S.~Boyd}, {\em {Anderson accelerated Douglas--Rachford splitting}}, SIAM J. Sci. Comput., 42 (2020), pp.~A3560--A3583.

\bibitem{geist2018anderson}
{\sc M.~Geist and B.~Scherrer}, {\em Anderson acceleration for reinforcement learning}, arXiv preprint arXiv:1809.09501,  (2018).

\bibitem{golub2013matrix}
{\sc G.~H. Golub and C.~F. Van~Loan}, {\em Matrix computations}, JHU Press, Baltimore, MD, 2013.

\bibitem{guo2018consistency}
{\sc X.~Guo, A.~Hu, R.~Xu, and J.~Zhang}, {\em Consistency and computation of regularized mles for multivariate hawkes processes}, arXiv preprint arXiv:1810.02955,  (2018).

\bibitem{gutknecht2007brief}
{\sc M.~H. Gutknecht}, {\em {A brief introduction to Krylov space methods for solving linear systems}}, in Front. Comput. Sci., Springer, 2007, pp.~53--62.

\bibitem{henderson2019damped}
{\sc N.~C. Henderson and R.~Varadhan}, {\em {Damped Anderson acceleration with restarts and monotonicity control for accelerating EM and EM-like algorithms}}, J. Comput. Graph. Stat., 28 (2019), pp.~834--846.

\bibitem{hestenes1952methods}
{\sc M.~R. Hestenes and E.~Stiefel}, {\em Methods of conjugate gradients for solving linear systems}, J. Res. Natl. Bur. Stand., 49 (1952), pp.~409--435.

\bibitem{krizhevsky2009learning}
{\sc A.~Krizhevsky}, {\em Learning multiple layers of features from tiny images}, Master's thesis, University of Tront,  (2009).

\bibitem{loffeld2016considerations}
{\sc J.~Loffeld and C.~S. Woodward}, {\em {Considerations on the implementation and use of Anderson acceleration on distributed memory and GPU-based parallel computers}}, in Adv. Math. Sci., Springer, 2016, pp.~417--436.

\bibitem{mai2020anderson}
{\sc V.~Mai and M.~Johansson}, {\em Anderson acceleration of proximal gradient methods}, in Int. Conf. Mach. Learn., PMLR, 2020, pp.~6620--6629.

\bibitem{meyer1976convergence}
{\sc R.~Meyer}, {\em On the convergence of algorithms with restart}, SIAM J. Numer. Anal., 13 (1976), pp.~696--704.

\bibitem{nesterov2018lectures}
{\sc Y.~Nesterov}, {\em Lectures on convex optimization}, vol.~137, Springer, 2018.

\bibitem{jorge2006numerical}
{\sc J.~Nocedal and S.~J. Wright}, {\em Numerical optimization}, Springer Series in Operations Research and Financial Engineering, Springer, New York, second~ed., 2006.

\bibitem{ouyang2020anderson}
{\sc W.~Ouyang, Y.~Peng, Y.~Yao, J.~Zhang, and B.~Deng}, {\em {Anderson acceleration for nonconvex ADMM based on Douglas-Rachford splitting}}, Comput. Graph. Forum, 39 (2020), pp.~221--239.

\bibitem{ouyang2020nonmonotone}
{\sc W.~Ouyang, J.~Tao, A.~Milzarek, and B.~Deng}, {\em {Nonmonotone globalization for Anderson acceleration using adaptive regularization}}, arXiv preprint arXiv:2006.02559,  (2020).

\bibitem{paige1982lsqr}
{\sc C.~C. Paige and M.~A. Saunders}, {\em {LSQR: An algorithm for sparse linear equations and sparse least squares}}, ACM transactions on mathematical software, 8 (1982), pp.~43--71.

\bibitem{pavlov2018aa}
{\sc A.~L. Pavlov, G.~W. Ovchinnikov, D.~Y. Derbyshev, D.~Tsetserukou, and I.~V. Oseledets}, {\em {AA-ICP}: Iterative closest point with {Anderson} acceleration}, in IEEE Int. Conf. Robot. Autom. (ICRA), IEEE, 2018, pp.~1--6.

\bibitem{peng2018anderson}
{\sc Y.~Peng, B.~Deng, J.~Zhang, F.~Geng, W.~Qin, and L.~Liu}, {\em Anderson acceleration for geometry optimization and physics simulation}, ACM Trans. Graph., 37 (2018), p.~42.

\bibitem{pham2021use}
{\sc X.-H. Pham, M.~Alamir, F.~Bonne, and P.~Bonnay}, {\em {On the use of Anderson acceleration in hierarchical control}}, arXiv preprint arXiv:2112.04299,  (2021).

\bibitem{pollock2021anderson}
{\sc S.~Pollock and L.~G. Rebholz}, {\em Anderson acceleration for contractive and noncontractive operators}, IMA J. Numer. Anal., 41 (2021), pp.~2841--2872.

\bibitem{pollock2019anderson}
{\sc S.~Pollock, L.~G. Rebholz, and M.~Xiao}, {\em {Anderson}-accelerated convergence of {P}icard iterations for incompressible {N}avier--{S}tokes equations}, SIAM J. Numer. Anal., 57 (2019), pp.~615--637.

\bibitem{potra2013characterization}
{\sc F.~A. Potra and H.~Engler}, {\em A characterization of the behavior of the {Anderson} acceleration on linear problems}, Linear Alg. Appl., 438 (2013), pp.~1002--1011.

\bibitem{powell1977restart}
{\sc M.~J.~D. Powell}, {\em Restart procedures for the conjugate gradient method}, Math. Program., 12 (1977), pp.~241--254.

\bibitem{pratapa2015restarted}
{\sc P.~P. Pratapa and P.~Suryanarayana}, {\em {Restarted Pulay mixing for efficient and robust acceleration of fixed-point iterations}}, Chem. Phys. Lett., 635 (2015), pp.~69--74.

\bibitem{rohwedder2011analysis}
{\sc T.~Rohwedder and R.~Schneider}, {\em {An analysis for the DIIS acceleration method used in quantum chemistry calculations}}, J. Math. Chem., 49 (2011), pp.~1889--1914.

\bibitem{saad2003iterative}
{\sc Y.~Saad}, {\em Iterative methods for sparse linear systems}, SIAM, 2003.

\bibitem{saad1986gmres}
{\sc Y.~Saad and M.~H. Schultz}, {\em {GMRES: A generalized minimal residual algorithm for solving nonsymmetric linear systems}}, SIAM J. Sci. Comput., 7 (1986), pp.~856--869.

\bibitem{scieur2016regularized}
{\sc D.~Scieur, A.~d'Aspremont, and F.~Bach}, {\em Regularized nonlinear acceleration}, in Adv. Neural Inf. Process. Syst., 2016, pp.~712--720.

\bibitem{scieur2017nonlinear}
{\sc D.~Scieur, A.~d'Aspremont, and F.~Bach}, {\em Nonlinear acceleration of stochastic algorithms}, arXiv preprint arXiv:1706.07270,  (2017).

\bibitem{shi2019regularized}
{\sc W.~Shi, S.~Song, H.~Wu, Y.-C. Hsu, C.~Wu, and G.~Huang}, {\em {Regularized Anderson acceleration for off-policy deep reinforcement learning}}, preprint arXiv:1909.03245,  (2019).

\bibitem{stiefel1955relaxationsmethoden}
{\sc E.~Stiefel}, {\em Relaxationsmethoden bester {Strategie} zur {L}{\"o}sung linearer {Gleichungssysteme}}, Commentarii Mathematici Helvetici, 29 (1955), pp.~157--179.

\bibitem{tang2022fast}
{\sc W.~Tang and P.~Daoutidis}, {\em Fast and stable nonconvex constrained distributed optimization: the ellada algorithm}, Optimization and Engineering, 23 (2022), pp.~259--301.

\bibitem{toth2015convergence}
{\sc A.~Toth and C.~Kelley}, {\em Convergence analysis for {Anderson} acceleration}, SIAM J. Numer. Anal., 53 (2015), pp.~805--819.

\bibitem{walker2011anderson}
{\sc H.~F. Walker and P.~Ni}, {\em Anderson acceleration for fixed-point iterations}, SIAM J. Numer. Anal., 49 (2011), pp.~1715--1735.

\bibitem{wang2021asymptotic}
{\sc D.~Wang, Y.~He, and H.~De~Sterck}, {\em {On the asymptotic linear convergence speed of Anderson acceleration applied to ADMM}}, J. Sci. Comput., 88 (2021), pp.~1--35.

\bibitem{WeiBaoLiu21}
{\sc F.~Wei, C.~Bao, and Y.~Liu}, {\em {Stochastic Anderson mixing for nonconvex stochastic optimization}}, in Adv. Neural Inf. Process. Syst., vol.~34, 2021, pp.~22995--23008.

\bibitem{xuNonconvexEmpirical2017}
{\sc P.~Xu, F.~Roosta, and M.~W. Mahoney}, {\em {Second-order optimization for non-convex machine learning: An empirical study}}, Proc. SIAM Int. Conf. Data Min.,  (2020), pp.~199--207.

\bibitem{zhang2020globally}
{\sc J.~Zhang, B.~O'Donoghue, and S.~Boyd}, {\em {Globally convergent type-I Anderson acceleration for nonsmooth fixed-point iterations}}, SIAM J. Optim., 30 (2020), pp.~3170--3197.

\bibitem{zhang2019accelerating}
{\sc J.~Zhang, Y.~Peng, W.~Ouyang, and B.~Deng}, {\em {Accelerating ADMM for efficient simulation and optimization}}, ACM Trans. Graph., 38 (2019), pp.~1--21.

\end{thebibliography}

\end{document}